\documentclass[preprint,10pt]{elsarticle}

\usepackage{amssymb}
\usepackage{amsmath}
\usepackage{amsthm}

\usepackage{lineno}

\usepackage[latin1]{inputenc}
\usepackage[T1]{fontenc}
\usepackage{xcolor}
\usepackage{color}
\usepackage{bm}
\usepackage{bbm}
\usepackage{tikz,circuitikz}
\usepackage{soul}
\usepackage{multirow}

\usepackage{subcaption}
\usepackage{graphicx}
\usetikzlibrary{shapes.geometric, arrows.meta, positioning}
\usepackage{float}

\usepackage[colorlinks=true, linkcolor=blue, citecolor=green, urlcolor=cyan]{hyperref}

\theoremstyle{definition}
\newtheorem{definition}{Definition}[section]
\newtheorem{proposition}{Proposition}[section]
\newtheorem{lemma}{Lemma}[section]
\newtheorem{theorem}{Theorem}[section]
\newtheorem{corollary}{Corollary}[section]

\newcommand{\E}{\mathbb{E}}
\renewcommand{\P}{\mathbb{P}}
\newcommand{\R}{\mathbb{R}}

\newcommand{\I}[1]{ \mathbf{1}_{\left\{ #1 \right\}} }
\newcommand{\Q}[2]{F_{#1, \, #2}}
\newcommand{\QQ}[2]{F_{[0, #1], \, #2}}
\newcommand{\QQQ}[3]{F_{[0, #1], \, [#2,#3]}}
\newcommand{\mini}[2]{#1 \wedge #2}%{\min\{#1 ,\, #2\}}

\newcommand{\Nd}{N^{disc}}%{N^D}
\newcommand{\Nc}{N}
\newcommand{\sys}{_{sys}}
\newcommand{\Tsys}{T_{fail}}
\newcommand{\Trep}{T_{rep}}

\renewcommand{\r}{^{(r)}}
\newcommand{\csys}{c_{sys}}
\newcommand{\ccmp}{c_{cmp}}

\usepackage{todonotes}
\usepackage{bm}

\begin{document}

%\linenumbers

\begin{frontmatter}

%% Title, authors and addresses

%% use the tnoteref command within \title for footnotes;
%% use the tnotetext command for theassociated footnote;
%% use the fnref command within \author or \affiliation for footnotes;
%% use the fntext command for theassociated footnote;
%% use the corref command within \author for corresponding author footnotes;
%% use the cortext command for theassociated footnote;
%% use the ead command for the email address,
%% and the form \ead[url] for the home page:
%% \title{Title\tnoteref{label1}}
%% \tnotetext[label1]{}
%% \author{Name\corref{cor1}\fnref{label2}}
%% \ead[url]{home page}
%% \fntext[label2]{}
%% \cortext[cor1]{}
%% \affiliation{organization={},
%%             addressline={},
%%             city={},
%%             postcode={},
%%             state={},
%%             country={}}
%% \fntext[label3]{}

\title{Simple repair policies and decompositions for semi-coherent systems with simultaneous failures}

%% use optional labels to link authors explicitly to addresses:
%% \author[label1,label2]{}
%% \affiliation[label1]{organization={},
%%             addressline={},
%%             city={},
%%             postcode={},
%%             state={},
%%             country={}}
%%
%% \affiliation[label2]{organization={},
%%             addressline={},
%%             city={},
%%             postcode={},
%%             state={},
%%             country={}}

\author[org:uai]{Guido Lagos\corref{cor1}}
\ead{guido.lagos@uai.cl}\cortext[cor1]{Corresponding author}
\affiliation[org:uai]{organization={Universidad Adolfo Ib\'a\~nez}, city={Vi\~na del Mar}, country={Chile}}
\author[org:um]{Jorge Navarro}
\ead{jorgenav@um.es}
\affiliation[org:um]{organization={Universidad de Murcia}, city={Murcia}, country={Spain}}
\author[org:uv]{H\'ector Olivero}
\ead{hector.olivero@uv.cl}
\affiliation[org:uv]{organization={Universidad de Valpara\'iso}, city={Valpara\'iso}, country={Chile}}

%% Abstract
\begin{abstract}

We consider semi-coherent binary systems that are subject to simultaneous failures of its components. These are systems whose components can be either working or failed; the system can also be working or failed depending on the state of the components; and repairing a component cannot cause the system to fail. We consider that one or more components can fail simultaneously, allowing us to model external shocks and disasters. For this, we use the L\'evy-frailty Marshall-Olkin (LFMO) multivariate distribution to model the failure times of the components. We aim to answer in which states of the system we should repair the components. This is a challenging question, as the number of repair policies grows super-exponentially in the number of components. To tackle this, we propose a simple family of repair policies, which we call \emph{$r$-out-of-$n$ repair policies}, where one repairs all failed components when the system fails or when there are $r$ or more failed components. Our main contribution is that we derive exact and simple expressions for key performance-evaluation quantities of the system operating under our proposed repair policies. That is, we give explicit expressions for the mean time-to-failure of the system, mean time-to-repair, probability of system-failure before repair, and component- and system-repair rate. We also give expressions for the expected cost and long-term average cost, when there are components' and system repair cost. The only relevant parameters involved in the derived expressions are the \emph{structural signature} of the system, and the Laplace exponent associated to the LFMO distribution.
\end{abstract}

%%Graphical abstract
%\begin{graphicalabstract}
%\includegraphics{grabs}
%\end{graphicalabstract}

%% Keywords
\begin{keyword}
L\'evy-frailty Marshall-Olkin \sep System Signature \sep Simultaneous failures \sep repair policies

%% PACS codes here, in the form: \PACS code \sep code

%% MSC codes here, in the form: \MSC code \sep code
%% or \MSC[2008] code \sep code (2000 is the default)

\end{keyword}

\end{frontmatter}

%% Add \usepackage{lineno} before \begin{document} and uncomment 
%% following line to enable line numbers
%% \linenumbers

%% main text
%%

\newpage

%%%%%%%%%%%%%%%%%%%%%%%%%%%%%%%%%%%%%%%%%%%%%%%%%%%%
%%%%%%%%%%%%%%%%%%%%%%%%%%%%%%%%%%%%%%%%%%%%%%%%%%%%
\section{Introduction}
%%%%%%%%%%%%%%%%%%%%%%%%%%%%%%%%%%%%%%%%%%%%%%%%%%%%
%%%%%%%%%%%%%%%%%%%%%%%%%%%%%%%%%%%%%%%%%%%%%%%%%%%%

In this paper we develop exact expressions for a class of repair policies when a binary system is subject to simultaneous failures of its components.
We do this for \emph{binary systems} with \emph{semi-coherent structure}, that is, systems whose components can be either working or failed, and where repairing (respectively, breaking down) a component cannot break down (respectively, repair) the system; see Figure~\ref{fig:basic} for a basic example.
%Specifically, we consider \emph{binary systems}, i.e., systems whose components can be working or not, and the system itself can also be either working or not, based on the state of its components. We assume that the structure of the system is \emph{semi-coherent}, which in simple terms means that repairing a component cannot break down the system, and also a component failure cannot repair the system.
To make the model more realistic, we assume that one or more components can fail simultaneously, for example, due to an external shock that takes down several components at once.
%We also assume that components can fail at the same time, simultaneously, say because of an external shock that takes down more than one component.
We model this using the \emph{L\'evy-frailty Marshall-Olkin} (LFMO) distribution for the times of failure of components.
This is a Markovian model that considers a degradation process common to all components, and an individual tolerance to degradation for each component.
In this setting, we consider repair policies where we either repair all failed components, or continue operating with the system as-is.
%In this setting, we consider repair policies based on the state of the components, i.e., for each state of working and failed components, we decide to either repair all failed components or continue operating with the system as-is. We assume that repairs are \emph{perfect}, in the sense that repairing a component takes it to a brand new operational state, and consider repair policies where %the repairs policy is restricted to a \emph{repair-all-failed} scheme, i.e., at each state of the system,
%we either repair \emph{all} failed components or repair none.
%Lastly, to model the times of failure of the components and allow for simultaneous failures, we use the \emph{L\'evy-frailty Marshall-Olkin} distribution. This is a Markovian model, where there is an underlying non-decreasing degradation process---a L\'evy subordinator process---that models the common degradation affecting all components of the system, and the component's tolerance to degradation is distributed iid standard exponential.

The main question we tackle is \emph{in which states of the system should we repair the failed components?} %Even when restricting to perfect repairs and \emph{repair-all-failed} policies, t
This is a particularly challenging question, as there is an exponential explosion of complexity with the number $n$ of components of the system, as it is illustrated in the basic example in Figure~\ref{fig:basic}.

To tackle this challenge, we consider a family of simple repair policies, that we call \emph{$r$-out-of-$n$ repair policies} (shortly written as \emph{$r$-out-of-$n$:R policy}), where for a given $r$, we repair all failed components when the system fails or when there are $r$ or more failed components; see Figure~\ref{fig:basic}. In this way, the policy for $r=1$ corresponds to repairing any component as soon as it fails, and the policy for $r=n$, where $n$ is the number of components, corresponds to repairing only when the system fails.

The main result of this paper is that we show explicit expressions for key performance evaluation quantities of the system operating under our proposed repair policies. That is, we give explicit expressions for the mean time-to-failure of the system, mean time-to-repair, probability of system-failure before repair, component- and system-repair rate, to name a few. We also give explicit formulas for the long-term mean cost when there is a cost for repairing components, and repairing a failed system incurs an additional system-repair cost due to operational interruption.
To the best of the authors' knowledge, this is the first result of this type for semi-coherent systems and for simultaneous failures.

%%%%%%%%%%%%%%%%%%%%%%%%%%%%%%%%%%%%%%%%%%%%%%%%%%%%
\begin{figure}[h]%[hbtp]
\centering
\resizebox{.45\linewidth}{!}{%
\begin{tikzpicture}[
    component/.style={draw, circle, inner sep=1pt, minimum size=0.6cm},
    %system/.style={draw=green!60!black, fill=green!10, rounded corners=5mm, minimum width=2.5cm, minimum height=2.5cm, align=center},
    %failed_system/.style={draw=red!60!black, fill=red!10, rounded corners=5mm, minimum width=2.5cm, minimum height=2.5cm, align=center},
    system/.style={fill=green!10, rounded corners=5mm, minimum width=2.5cm, minimum height=2.5cm, align=center},
    failed_system/.style={fill=red!10, rounded corners=5mm, minimum width=2.5cm, minimum height=2.5cm, align=center},
    cross out/.pic={
        \draw[red, very thick] (-0.3, -0.3) -- (0.3, 0.3);
        \draw[red, very thick] (-0.3, 0.3) -- (0.3, -0.3);
    }
]

% First row, first system (all working)
\node[system] (sys111) at (0, -4) {};
\node[component] (c1_1) at (sys111.north) [xshift=-0.5cm, yshift=-0.5cm] {1};
\node[component] (c1_2) at (sys111.north) [xshift= 0.5cm, yshift=-0.5cm] {2};
\node[component] (c1_3) at (sys111.south) [yshift= 0.5cm] {3};
\draw (c1_1.east) -- (c1_2.west);
\draw (sys111.east) -| +(-.3,0) |- (c1_2.east);
\draw (sys111.east) -| +(-.3,0) |- (c1_3.east);
\draw (sys111.west) -| +(.3,0) |- (c1_3.west);
\draw (sys111.west) -| +(.3,0) |- (c1_1.west);

% First row, second system (1 failed)
\node[system] (sys011) at (4, 0) {};
\node[component] (c2_1) at (sys011.north) [xshift=-0.5cm, yshift=-0.5cm] {1};
\node[component] (c2_2) at (sys011.north) [xshift= 0.5cm, yshift=-0.5cm] {2};
\node[component] (c2_3) at (sys011.south) [yshift= 0.5cm] {3};
\draw (c2_1.east) -- (c2_2.west);
\draw (sys011.east) -| +(-.3,0) |- (c2_2.east);
\draw (sys011.east) -| +(-.3,0) |- (c2_3.east);
\draw (sys011.west) -| +( .3,0) |- (c2_3.west);
\draw (sys011.west) -| +( .3,0) |- (c2_1.west);
\pic at (c2_1) {cross out};

\node[failed_system] (sys100) at (8, -0) {};
\node[component] (c9_1) at (sys100.north) [xshift=-0.5cm, yshift=-0.5cm] {1};
\node[component] (c9_2) at (sys100.north) [xshift= 0.5cm, yshift=-0.5cm] {2};
\node[component] (c9_3) at (sys100.south) [yshift= 0.5cm] {3};
\draw (c9_1.east) -- (c9_2.west);
\draw (sys100.east) -| +(-.3,0) |- (c9_2.east);
\draw (sys100.east) -| +(-.3,0) |- (c9_3.east);
\draw (sys100.west) -| +( .3,0) |- (c9_3.west);
\draw (sys100.west) -| +( .3,0) |- (c9_1.west);
\pic at (c9_2) {cross out};
\pic at (c9_3) {cross out};

% First row, fourth system (3 failed)
\node[failed_system] (sys000) at (12, -4) {};
\node[component] (c4_1) at (sys000.north) [xshift=-0.5cm, yshift=-0.5cm] {1};
\node[component] (c4_2) at (sys000.north) [xshift= 0.5cm, yshift=-0.5cm] {2};
\node[component] (c4_3) at (sys000.south) [yshift= 0.5cm] {3};
\draw (c4_1.east) -- (c4_2.west);
\draw (sys000.east) -| +(-.3,0) |- (c4_2.east);
\draw (sys000.east) -| +(-.3,0) |- (c4_3.east);
\draw (sys000.west) -| +( .3,0) |- (c4_3.west);
\draw (sys000.west) -| +( .3,0) |- (c4_1.west);
\pic at (c4_1) {cross out};
\pic at (c4_2) {cross out};
\pic at (c4_3) {cross out};

% Second row, first system (2 failed)
\node[system] (sys101) at (4, -4) {};
\node[component] (c5_1) at (sys101.north) [xshift=-0.5cm, yshift=-0.5cm] {1};
\node[component] (c5_2) at (sys101.north) [xshift= 0.5cm, yshift=-0.5cm] {2};
\node[component] (c5_3) at (sys101.south) [yshift= 0.5cm] {3};
\draw (c5_1.east) -- (c5_2.west);
\draw (sys101.east) -| +(-.3,0) |- (c5_2.east);
\draw (sys101.east) -| +(-.3,0) |- (c5_3.east);
\draw (sys101.west) -| +( .3,0) |- (c5_3.west);
\draw (sys101.west) -| +( .3,0) |- (c5_1.west);
\pic at (c5_2) {cross out};

% Second row, second system (2,3 failed)
\node[failed_system] (sys010) at (8, -4) {};
\node[component] (c6_1) at (sys010.north) [xshift=-0.5cm, yshift=-0.5cm] {1};
\node[component] (c6_2) at (sys010.north) [xshift= 0.5cm, yshift=-0.5cm] {2};
\node[component] (c6_3) at (sys010.south) [yshift= 0.5cm] {3};
\draw (c6_1.east) -- (c6_2.west);
\draw (sys010.east) -| +(-.3,0) |- (c6_2.east);
\draw (sys010.east) -| +(-.3,0) |- (c6_3.east);
\draw (sys010.west) -| +( .3,0) |- (c6_3.west);
\draw (sys010.west) -| +( .3,0) |- (c6_1.west);
\pic at (c6_1) {cross out};
\pic at (c6_3) {cross out};

% Second row, third system (1,3 failed)
\node[system] (sys001) at (8, -8) {};
\node[component] (c7_1) at (sys001.north) [xshift=-0.5cm, yshift=-0.5cm] {1};
\node[component] (c7_2) at (sys001.north) [xshift= 0.5cm, yshift=-0.5cm] {2};
\node[component] (c7_3) at (sys001.south) [yshift= 0.5cm] {3};
\draw (c7_1.east) -- (c7_2.west);
\draw (sys001.east) -| +(-.3,0) |- (c7_2.east);
\draw (sys001.east) -| +(-.3,0) |- (c7_3.east);
\draw (sys001.west) -| +( .3,0) |- (c7_3.west);
\draw (sys001.west) -| +( .3,0) |- (c7_1.west);
\pic at (c7_1) {cross out};
\pic at (c7_2) {cross out};

% Third row, first system (3 failed)
\node[system] (sys110) at (4, -8) {};
\node[component] (c8_1) at (sys110.north) [xshift=-0.5cm, yshift=-0.5cm] {1};
\node[component] (c8_2) at (sys110.north) [xshift= 0.5cm, yshift=-0.5cm] {2};
\node[component] (c8_3) at (sys110.south) [yshift= 0.5cm] {3};
\draw (c8_1.east) -- (c8_2.west);
\draw (sys110.east) -| +(-.3,0) |- (c8_2.east);
\draw (sys110.east) -| +(-.3,0) |- (c8_3.east);
\draw (sys110.west) -| +( .3,0) |- (c8_3.west);
\draw (sys110.west) -| +( .3,0) |- (c8_1.west);
\pic at (c8_3) {cross out};

%\draw [magenta, thick, rounded corners=1cm] (2,2) -- (10,2) -- (13.9,-2) -- (13.9,-6) -- (10, -9.5) -- (2, -9.5) -- cycle;
%\draw [magenta, thick, rounded corners=1cm] (6.3, 1.8) -- (10, 1.8) -- (13.7,-2) -- (13.7,-6) -- (10,-9.4) -- (6.3, -9.4) -- cycle;
%\draw [magenta, thick, rounded corners=1cm] (6.5, 1.6) -- (10, 1.6) -- (13.5,-2) -- (13.5,-5.5) -- (6.5, -5.5) -- cycle;

\draw[->, magenta] (sys111) .. controls +(2,2) and +(-2,-2) .. (sys011);
\draw[->, magenta] (sys111) .. controls +(2,-.3) and +(-2,.5) .. (sys101);
\draw[->, magenta] (sys111) .. controls +(2,-2) and +(-2,2) .. (sys110);
\draw[->, magenta] (sys111) .. controls +(2,1.7) and +(-2,-2) .. (sys100);
\draw[->, magenta] (sys111) .. controls +(2,-.6) and +(-2,.5) .. (sys010);
\draw[->, magenta] (sys111) .. controls +(2,-1.7) and +(-2,2) .. (sys001);
\draw[->, magenta] (sys111) .. controls +(2,-.9) and +(-2,.5) .. (sys000);

\draw[->, magenta] (sys011) .. controls +(2,-0.5) and +(-2,2) .. (sys000);
\draw[->, magenta] (sys011) .. controls +(2,-1.0) and +(-2,2) .. (sys010);
\draw[->, magenta] (sys011) .. controls +(2,-1.5) and +(-2,2) .. (sys001);

\draw[->, magenta] (sys110) .. controls +(2,1.5) and +(-2,-2) .. (sys100);
\draw[->, magenta] (sys110) .. controls +(2,1.0) and +(-2,-2) .. (sys010);
\draw[->, magenta] (sys110) .. controls +(2,0.5) and +(-2,-2) .. (sys000);

\draw[->, magenta] (sys101) .. controls +(2,1.5) and +(-2,-2) .. (sys100);
\draw[->, magenta] (sys101) .. controls +(2,-1.5) and +(-2,2) .. (sys001);
\draw[->, magenta] (sys101) .. controls +(2,-0.8) and +(-2,0.5) .. (sys000);

\draw[->, magenta] (sys100) .. controls +(2,-2) and +(-2,2) .. (sys000);
\draw[->, magenta] (sys010) .. controls +(2,-.5) and +(-2,.5) .. (sys000);
\draw[->, magenta] (sys001) .. controls +(2,2) and +(-2,-2) .. (sys000);

\end{tikzpicture}
}~\resizebox{.05\linewidth}{!}{\ }~\resizebox{.45\linewidth}{!}{%
\begin{tikzpicture}[
    component/.style={draw, circle, inner sep=1pt, minimum size=0.6cm},
    %system/.style={draw=green!60!black, fill=green!10, rounded corners=5mm, minimum width=2.5cm, minimum height=2.5cm, align=center},
    %failed_system/.style={draw=red!60!black, fill=red!10, rounded corners=5mm, minimum width=2.5cm, minimum height=2.5cm, align=center},
    system/.style={fill=green!10, rounded corners=5mm, minimum width=2.5cm, minimum height=2.5cm, align=center},
    failed_system/.style={fill=red!10, rounded corners=5mm, minimum width=2.5cm, minimum height=2.5cm, align=center},
    cross out/.pic={
        \draw[red, very thick] (-0.3, -0.3) -- (0.3, 0.3);
        \draw[red, very thick] (-0.3, 0.3) -- (0.3, -0.3);
    }
]

% First row, first system (all working)
\node[system] (sys111) at (0, -4) {};
\node[component] (c1_1) at (sys111.north) [xshift=-0.5cm, yshift=-0.5cm] {1};
\node[component] (c1_2) at (sys111.north) [xshift= 0.5cm, yshift=-0.5cm] {2};
\node[component] (c1_3) at (sys111.south) [yshift= 0.5cm] {3};
\draw (c1_1.east) -- (c1_2.west);
\draw (sys111.east) -| +(-.3,0) |- (c1_2.east);
\draw (sys111.east) -| +(-.3,0) |- (c1_3.east);
\draw (sys111.west) -| +(.3,0) |- (c1_3.west);
\draw (sys111.west) -| +(.3,0) |- (c1_1.west);

% First row, second system (1 failed)
\node[system] (sys011) at (4, 0) {};
\node[component] (c2_1) at (sys011.north) [xshift=-0.5cm, yshift=-0.5cm] {1};
\node[component] (c2_2) at (sys011.north) [xshift= 0.5cm, yshift=-0.5cm] {2};
\node[component] (c2_3) at (sys011.south) [yshift= 0.5cm] {3};
\draw (c2_1.east) -- (c2_2.west);
\draw (sys011.east) -| +(-.3,0) |- (c2_2.east);
\draw (sys011.east) -| +(-.3,0) |- (c2_3.east);
\draw (sys011.west) -| +( .3,0) |- (c2_3.west);
\draw (sys011.west) -| +( .3,0) |- (c2_1.west);
\pic at (c2_1) {cross out};

\node[failed_system] (sys100) at (8, -0) {};
\node[component] (c9_1) at (sys100.north) [xshift=-0.5cm, yshift=-0.5cm] {1};
\node[component] (c9_2) at (sys100.north) [xshift= 0.5cm, yshift=-0.5cm] {2};
\node[component] (c9_3) at (sys100.south) [yshift= 0.5cm] {3};
\draw (c9_1.east) -- (c9_2.west);
\draw (sys100.east) -| +(-.3,0) |- (c9_2.east);
\draw (sys100.east) -| +(-.3,0) |- (c9_3.east);
\draw (sys100.west) -| +( .3,0) |- (c9_3.west);
\draw (sys100.west) -| +( .3,0) |- (c9_1.west);
\pic at (c9_2) {cross out};
\pic at (c9_3) {cross out};

% First row, fourth system (3 failed)
\node[failed_system] (sys000) at (12, -4) {};
\node[component] (c4_1) at (sys000.north) [xshift=-0.5cm, yshift=-0.5cm] {1};
\node[component] (c4_2) at (sys000.north) [xshift= 0.5cm, yshift=-0.5cm] {2};
\node[component] (c4_3) at (sys000.south) [yshift= 0.5cm] {3};
\draw (c4_1.east) -- (c4_2.west);
\draw (sys000.east) -| +(-.3,0) |- (c4_2.east);
\draw (sys000.east) -| +(-.3,0) |- (c4_3.east);
\draw (sys000.west) -| +( .3,0) |- (c4_3.west);
\draw (sys000.west) -| +( .3,0) |- (c4_1.west);
\pic at (c4_1) {cross out};
\pic at (c4_2) {cross out};
\pic at (c4_3) {cross out};

% Second row, first system (2 failed)
\node[system] (sys101) at (4, -4) {};
\node[component] (c5_1) at (sys101.north) [xshift=-0.5cm, yshift=-0.5cm] {1};
\node[component] (c5_2) at (sys101.north) [xshift= 0.5cm, yshift=-0.5cm] {2};
\node[component] (c5_3) at (sys101.south) [yshift= 0.5cm] {3};
\draw (c5_1.east) -- (c5_2.west);
\draw (sys101.east) -| +(-.3,0) |- (c5_2.east);
\draw (sys101.east) -| +(-.3,0) |- (c5_3.east);
\draw (sys101.west) -| +( .3,0) |- (c5_3.west);
\draw (sys101.west) -| +( .3,0) |- (c5_1.west);
\pic at (c5_2) {cross out};

% Second row, second system (2,3 failed)
\node[failed_system] (sys010) at (8, -4) {};
\node[component] (c6_1) at (sys010.north) [xshift=-0.5cm, yshift=-0.5cm] {1};
\node[component] (c6_2) at (sys010.north) [xshift= 0.5cm, yshift=-0.5cm] {2};
\node[component] (c6_3) at (sys010.south) [yshift= 0.5cm] {3};
\draw (c6_1.east) -- (c6_2.west);
\draw (sys010.east) -| +(-.3,0) |- (c6_2.east);
\draw (sys010.east) -| +(-.3,0) |- (c6_3.east);
\draw (sys010.west) -| +( .3,0) |- (c6_3.west);
\draw (sys010.west) -| +( .3,0) |- (c6_1.west);
\pic at (c6_1) {cross out};
\pic at (c6_3) {cross out};

% Second row, third system (1,3 failed)
\node[system] (sys001) at (8, -8) {};
\node[component] (c7_1) at (sys001.north) [xshift=-0.5cm, yshift=-0.5cm] {1};
\node[component] (c7_2) at (sys001.north) [xshift= 0.5cm, yshift=-0.5cm] {2};
\node[component] (c7_3) at (sys001.south) [yshift= 0.5cm] {3};
\draw (c7_1.east) -- (c7_2.west);
\draw (sys001.east) -| +(-.3,0) |- (c7_2.east);
\draw (sys001.east) -| +(-.3,0) |- (c7_3.east);
\draw (sys001.west) -| +( .3,0) |- (c7_3.west);
\draw (sys001.west) -| +( .3,0) |- (c7_1.west);
\pic at (c7_1) {cross out};
\pic at (c7_2) {cross out};

% Third row, first system (3 failed)
\node[system] (sys110) at (4, -8) {};
\node[component] (c8_1) at (sys110.north) [xshift=-0.5cm, yshift=-0.5cm] {1};
\node[component] (c8_2) at (sys110.north) [xshift= 0.5cm, yshift=-0.5cm] {2};
\node[component] (c8_3) at (sys110.south) [yshift= 0.5cm] {3};
\draw (c8_1.east) -- (c8_2.west);
\draw (sys110.east) -| +(-.3,0) |- (c8_2.east);
\draw (sys110.east) -| +(-.3,0) |- (c8_3.east);
\draw (sys110.west) -| +( .3,0) |- (c8_3.west);
\draw (sys110.west) -| +( .3,0) |- (c8_1.west);
\pic at (c8_3) {cross out};

\draw [blue, thick, rounded corners=1cm, dotted] (2,2) -- (10,2) -- (13.9,-2) -- (13.9,-6) -- (10, -9.5) -- (2, -9.5) -- cycle;
\draw [blue, rounded corners=1cm, densely dashed] (6.3, 1.8) -- (10, 1.8) -- (13.7,-2) -- (13.7,-6) -- (10,-9.4) -- (6.3, -9.4) -- cycle;
\draw [blue, rounded corners=1cm, densely dashdotted] (6.5, 1.6) -- (10, 1.6) -- (13.5,-2) -- (13.5,-5.5) -- (6.5, -5.5) -- cycle;

\end{tikzpicture}
}%
\caption{
(Left)~For a binary system with $n=3$ components, we show all possible combinations of working and failed components (crossed in red). The system works when there is a path of working components from left to right, and the green or red background shows if the system is working or not in that configuration. The magenta lines are all possible transitions due to failures; note that more than one component can fail at the same time. The question we aim to tackle is \emph{in which states should we repair all failed components?} Crucially, there is an exponential explosion of complexity with the number $n$ of components in the system: there are $2^n=8$ states, $3^n-2^n = 19$ transitions, and $2^{2^n-1-3}=16$ possible repair policies.
Indeed, there are $2^n$ states of the system, however we subtract the all-working state (we cannot repair components if none has failed) and the 3 system-failed states (we are forced to repair, since continuing to operate a broken system is meaningless). In these $2^n-1-3$ states, we need to decide in which subset of them to repair all failed components, and there are a total of $2^{2^n-1-3}$ of these subsets.
The count is doubly exponential in $n$, which is the key complexity issue we address.
(Right)~We propose simple repair policies that we denote $r$-out-of-$n$:R (\emph{repair all failed components when there are $r$ or more failed components or the system fails}). In this case, there are $n=3$ of these policies: $r=1$ in dots, $r=2$ in dashes, and $r=3$ in dash-dots. The boundary indicates the states where, upon reaching them, all failed components should be repaired.
}
\label{fig:basic}
\end{figure}
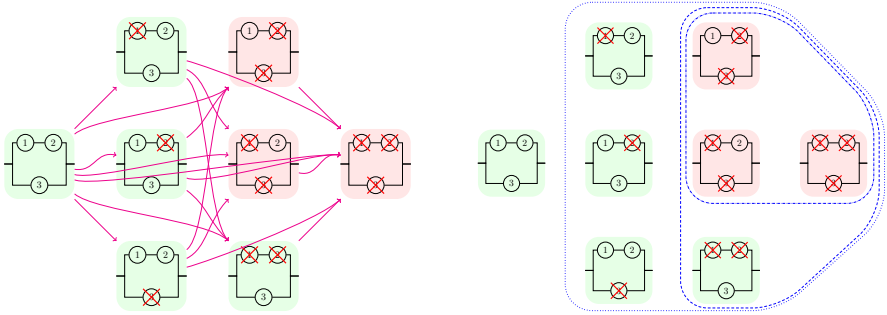
%%%%%%%%%%%%%%%%%%%%%%%%%%%%%%%%%%%%

Our result is heavily based on the fact that, in this particular setting, we can easily analyze our policies using only the Markov chain of the number of failed components, and decompose the structure of the system into much simpler $k$-out-of-$n$:F systems---i.e., systems that fail when $k$ or more of its components fail. Indeed, the LFMO distribution allows for easy analysis of the Markov chain of number of failures, and also compute its rates and probabilities; see Figure~\ref{fig:number}.
In turn, the Samaniego decomposition result states that, from a probabilistic perspective, the system can be seen as a probabilistic mixture of $k$-out-of-$n$:F systems, using the so-called structural signature of the system that summarizes the structure of the system; see Figures~\ref{fig:signature} and~\ref{fig:samaniego}.
Furthermore, our proposed $r$-out-of-$n$:R policies can exploit the same state aggregation of $k$-out-of-$n$:F systems.

%%%%%%%%%%%%%%%%%%%%%%%%%%%%%%%%%%%%%%%%%%%%%%%%%%%%
\paragraph{Main contributions}
%%%%%%%%%%%%%%%%%%%%%%%%%%%%%%%%%%%%%%%%%%%%%%%%%%%%

The main contributions of this work are the following:

\begin{enumerate}
    \item We propose simple repair policies, that we call $r$-out-of-$n$ repair policies (or $r$-out-of-$n$:R policy, in short), and show explicit mathematical formulas for key performance indicators (e.g., mean time-to-failure, probability of failure before repair, etc.)~of general semi-coherent systems operating under these policies, when it is subject to simultaneous failures of its components. We do this using the LFMO distribution for the failure times of the components. Our formulas involve only the structural signature of the system and the Laplace exponent of the underlying L\'evy subordinator. As far as the authors know, this is the first time that explicit expressions have been derived for general semi-coherent systems with simultaneous failures.

    \item As particular cases, we also obtain explicit expressions in two cases of interest in the literature: the case of $k$-out-of-$n$:F systems with LFMO distributed failure times of its components (see Proposition~\ref{corollary}); and the case of general semi-coherent systems with iid exponentially distributed failure times of its components (see Corollary~\ref{corol:iid}).

\item While in the classical setting the structural signature is used essentially to represent the system reliability function, our results show that under simultaneous failures it drives a much wider set of objects: number of failed components, times of repair, and costs. As the relevant object capturing this broader role, we define the \emph{process signature} of the system as the probability distribution of the number of failed components at the system-failure epoch (see Definition~\ref{def:proc sig} below).
    This new signature is also of relevance for systems with discrete lifetimes of components, since these too feature simultaneous failures.
    We derive explicit expressions in the setting of semi-coherent systems with LFMO distributed components' lifetimes. We do this when there is no repair policy (see Proposition~\ref{proposition1}) and when an $r$-out-of-$n$:R policy is in operation (see Theorem~\ref{theo}).
    
    %In the previous setting, we derive the \emph{signature} of the system, in the classical sense (not to be confused with \emph{structural signature}, see Definition~\ref{def:signature}).
    %That is, we derive explicit expressions for the probability distribution of the number of failed components when the system fails.
    %We derive this quantity in the setting where no repair policy is working (see Proposition~\ref{proposition1}) and when an $r$-out-of-$n$:R policy is in operation (see Theorem~\ref{theo}).
    %As above, this is the first time that this type of explicit expressions have been derived for general semi-coherent systems with simultaneous failures, as far as the authors know.

    \item From a methodology perspective, our analysis extends the classical decomposition result of Samaniego~\cite{samaniego1985closure,navarro2008application,marichal2011signature}, to decompose further quantities of interest, such as the number of failed components (process signature), the times of repair, and costs; see Section~\ref{sec:signature}.
    This is an important development, as the Samaniego result is a mainstay in the probabilistic analysis of system reliability, which, in part, has greatly motivated the study of systems' signatures.
    %As far as the authors know, this is the first time that our type of extensions have been made.
    In addition, our results are heavily based on the exchangeability property of the LFMO distribution, so our work opens new research venues to more general exchangeable distributions.
\end{enumerate}

%%%%%%%%%%%%%%%%%%%%%%%%%%%%%%%%%%%%%%%%%%%%%%%%%%%%
\subsection{Relation to prior literature}
%%%%%%%%%%%%%%%%%%%%%%%%%%%%%%%%%%%%%%%%%%%%%%%%%%%%

In this subsection we position the present manuscript against three strands of prior work: (i)~the literature on system signatures, (ii)~the literature on multivariate models for simultaneous failures, and (iii)~the literature on repairable systems and cost-rate optimization. After each strand we make explicit which limitation of the existing results is removed in the present paper.

\paragraph{System signatures}
The formal study of binary reliability systems started in the late 1950s, see~\cite{black1959optimal,birnbaum1961multi}, with its study driven by structure functions, redundancy, and properties such as monotonicity and coherency; see e.g.~\cite{barlow1996mathematical} and the recent survey~\cite{aven2025fifty}. An important development was the definition and study of the signature of a system in~\cite{samaniego1985closure}, which allows to analyze mean failure times, aging properties, and residual lifetimes, see the survey~\cite{naqvi2022system}, and more recently in~\cite{barrera2020limit,barrera2020approximating,lagos2024limiting}, to approximate the behavior of either large times or for large systems.
Our results are also related with minimal signatures and dominations~\cite{Boland,Grabisch,navarro2021introduction}.
In all of these works the structural signature is used essentially as a building block for the system reliability function. In the present paper we instead show that, under simultaneous failures, the signature drives a much wider set of objects --- number of failed components at the system-failure epoch, repair times, and long-term mean costs --- and we introduce the \emph{process signature} as the appropriate object that captures the case when components may fail simultaneously (Definition~\ref{def:proc sig}).

\paragraph{Multivariate models for simultaneous failures}
For the modeling of the simultaneous failures of components, we use the L\'evy-frailty Marshall-Olkin distribution, originally proposed as a copula in~\cite{mai2009levy} and later extended to a multivariate distribution with exchangeable components in~\cite{mai2011reparameterizing}. It is a particular case of the Marshall-Olkin distribution, originally proposed in the late 1960s in~\cite{marshall1967multivariate}, which is a classic model for simultaneous failures in reliability modeling.  Other multivariate distributions used to model simultaneous failures are multivariate additive processes~\cite{mercier2023general}, shock models~\cite{mallor2003classification}, and especially copulas, see~\cite{davies2024residual} and references therein.
Within this strand, prior work focuses largely on \emph{distributional} properties (joint reliability function, residual lifetimes, dependence structure). The present paper exploits the specific Markovian and exchangeability structure of the LFMO model to obtain \emph{closed-form performance indicators} under active repair policies, which to the best of our knowledge has not been done for any of the other distributions in this strand.

\paragraph{Repairable systems and cost-rate optimization}
The literature on mathematical modeling of repairable systems is considerable. Some approaches in the literature are age-based policies, where the average cost is optimized based on the age of the system, see~\cite{niu2025preventive} and references therein; also conditioned-based maintenance (CBM) where there is a stochastic process modeling the degradation of the system and a maintenance is programmed based on the state of the degradation, see~\cite{grall2023continuous}; and so-called opportunistic maintenance, where system failures are used to replace other components that are operative but have a high probability of failure, see~\cite{barde2024efficient}. Other classical approaches are imperfect repairs, where a repair does not restore components or the system to be ``as good as new'' condition.
See~\cite{nakagawa2005maintenance,nakagawa2007shock,nakagawa2008advanced} and the survey~\cite{nicolai2008optimal} for overviews of maintenance topics in Reliability Theory.

Deriving expressions for the mean cost rate to optimize it is a central theme in Reliability Theory; however, there are no works proposing formulas for these quantities in a general setting of monotone or coherent systems with simultaneous failures. Indeed, the following works have explored cost rate policies and expressions, but for specific coherent structures such as series, parallel, $k$-out-of-$n$:F, and related systems. The paper~\cite{eryilmaz2023age}~considers age-based policies and derives conditions for general coherent systems under discrete iid components' lifetimes; however, they derive explicit expressions only for $k$-out-of-$n$ and linear consecutive-$k$-out-of-$n$:F systems. Analogously, \cite{niu2025preventive}~considers series and parallel systems with dependent components, modeled using general copulas, and~\cite{eryilmaz2020optimization} considers parallel systems with dependent exchangeable distribution of components' lifetimes.
See the references therein for further works on replacement policies.

Summarizing the comparison with this strand: the cited works obtain explicit cost-rate formulas only under one of two restrictions --- either the structure is fixed to a narrow class (series, parallel, $k$-out-of-$n$:F, consecutive-$k$-out-of-$n$:F), or simultaneous failures are excluded. The present paper removes both restrictions simultaneously, by covering all semi-coherent structures and explicitly modelling simultaneous failures via the LFMO distribution, while still delivering closed-form expressions for the mean time-to-failure, mean time-to-repair, probability of system-failure-before-repair, and the long-term mean cost under our proposed $r$-out-of-$n$:R policies. To the best of the authors' knowledge, this is the first set of closed-form results of this scope in the simultaneous-failures setting.

%%%%%%%%%%%%%%%%%%%%%%%%%%%%%%%%%%%%%%%%%%%%%%%%%%%%
\subsection*{Notation}
%%%%%%%%%%%%%%%%%%%%%%%%%%%%%%%%%%%%%%%%%%%%%%%%%%%%
For vectors $\mathbf x$ and $\mathbf y$ in $\{0, 1\}^n$, $\mathbf x \leq \mathbf y$ means that $x_i \leq y_i$ for all $i=1, \ldots, n$, and $\lvert \mathbf x \rvert = \sum_{i=1}^n x_i$. We denote the cardinality of a set $A$ by $|A|$, and sometimes by $\# A$ whenever there may be confusion with $\lvert \mathbf x \rvert$ for a vector $\mathbf x$. For a function or stochastic process, say $(N(s) \ : \ s \geq 0)$ and some $t>0$, we denote $N(t^-) := \lim_{s \nearrow t} N(s)$. Also, we denote $\mini{k}{r} := \min\{k,r\}$ for $k,r$ in $\R$. For a nondecreasing and positive function $f(n)$, we denote $\mathcal{O}(f(n))$ as a function such that $\limsup_n \mathcal{O}(f(n)) / f(n) < +\infty$.

%%%%%%%%%%%%%%%%%%%%%%%%%%%%%%%%%%%%%%%%%%%%%%%%%%%%
\subsection*{Abbreviations}
%%%%%%%%%%%%%%%%%%%%%%%%%%%%%%%%%%%%%%%%%%%%%%%%%%%%

The following abbreviations are used throughout the paper: 
CBM, condition-based maintenance; PP, Poisson process; CPP, compound Poisson process; EMO, exchangeable Marshall-Olkin; EXC, exchangeable; iid, independent and identically distributed; LFMO, L\'evy-frailty Marshall-Olkin; LTMC, long-term mean cost; LTMN, long-term mean number of failures; MO, Marshall-Olkin; MTTF, mean time to failure.

%%%%%%%%%%%%%%%%%%%%%%%%%%%%%%%%%%%%%%%%%%%%%%%%%%%%
%%%%%%%%%%%%%%%%%%%%%%%%%%%%%%%%%%%%%%%%%%%%%%%%%%%%
\section{Mathematical model}\label{sec:model}
%%%%%%%%%%%%%%%%%%%%%%%%%%%%%%%%%%%%%%%%%%%%%%%%%%%%
%%%%%%%%%%%%%%%%%%%%%%%%%%%%%%%%%%%%%%%%%%%%%%%%%%%%

In this section we show the basic mathematical models we use.
We will consider so-called \emph{binary} systems where each component, and the system itself, can be in either a working or failed state.
In Section~\ref{sec:lifetimes} we specify the probabilistic behavior we use for components having simultaneous failures; then in Section~\ref{sec:spoiler} we give a result for $k$-out-of-$n$:F systems with repairs, as a spoiler of our main result; then in Section~\ref{sec:structure} we specify the system structure we consider; and in Section~\ref{sec:coherent} we give preliminary results for the system structure and distribution considered.

%%%%%%%%%%%%%%%%%%%%%%%%%%%%%%%%%%%%
\subsection{Lifetimes of components}\label{sec:lifetimes}
%%%%%%%%%%%%%%%%%%%%%%%%%%%%%%%%%%%%

Throughout this work, we consider a system with $n>1$ components, where each component can be in a working or failed state. For each component $i = 1, \ldots, n$, denote by $T_i$ the random time at which it fails. We want to model that several components can fail simultaneously, say because of degradation shocks that hit the system, and for that we focus on the joint distribution of the random vector of lifetimes $\mathbf T = (T_1, \ldots, T_n)$ in $\R^n_+$. A classical model for this is the following Marshall-Olkin (exponential) distribution, originally proposed in~\cite{marshall1967multivariate}.

%%%%%%%%%%%%%%%%%%%%%%%%%%%%%%%%%%%%%%%%%%%%%%%%%%%%
\begin{definition}[Marshall-Olkin distribution]\label{def:MO}
A random vector $\mathbf T = (T_1, \ldots, T_n)$ in $\R^n$ is said to have a \emph{Marshall-Olkin} (MO) distribution if
\begin{align}\label{def:MO eq}
	T_i = \min_{V\subseteq \{1,\dots,n\} \, : \, i\in V} X_V, \qquad i = 1, \ldots, n,
\end{align}
where, for all $V\subseteq \{1,\dots,n\}$, $X_V$ is an exponential random variable with parameter $\lambda_V \geq 0$, and is independent of the other random variables.
\end{definition}
%%%%%%%%%%%%%%%%%%%%%%%%%%%%%%%%%%%%%%%%%%%%%%%%%%%%

The random variables $X_V$ represent the time of arrival of a shock that simultaneously hits all components in the set $V$. Hence,~\eqref{def:MO eq} specifies that the time of failure of a component is the first arrival time of any of the shocks that hit it. It is known that the resulting distribution for $\mathbf T$ has a multidimensional version of the memoryless property of exponential random variables, see e.g.~\cite[Section 3.1]{matthias2017simulating}, and thus induces a Markovian structure on the failure times. However, a considerable drawback of the MO distribution is its \emph{parametric complexity} --- we need to specify the value of the parameters $\lambda_V \geq 0$ for all $V\subseteq \{1,\dots,n\}$, $V\neq\emptyset$, i.e., a total of $2^n-1$ parameters.

A particular subfamily of the MO family is the one of \emph{exchangeable} Marshall-Olkin (EMO) distributions, that hold when $\lambda_U=\lambda_V$ whenever $|U| = |V|$, i.e., they have the same cardinal; and that, indeed, induce a random vector $(T_1, \ldots, T_n)$ with exchangeable (EXC) components, meaning that for all permutation $\sigma$ of the set $\{1,\ldots,n\}$, the vectors $(T_1,\dots,T_n)$ and $(T_{\sigma(1)},\dots,T_{\sigma(n)})$ have the same distribution.

A further subfamily of the EMO distributions is the following, proposed in~\cite{mai2009levy,mai2014multivariate}.
%in~\eqref{def:LFMO eq} --- indeed, the latter case is equal in distribution to the random vector in~\eqref{def:MO eq} when the rates $\lambda_V$ are chosen as in~\eqref{def:LV}.

%%%%%%%%%%%%%%%%%%%%%%%%%%%%%%%%%%%%%%%%%%%%%%%%%%%%
\begin{definition}[L\'evy-frailty Marshall-Olkin distribution]\label{def:LFMO}
A random vector $\mathbf T$ in $\R^n$ is said to have a \emph{L\'evy-frailty Marshall-Olkin} (LFMO) distribution if its components $(T_1, \ldots, T_n)$ can be jointly defined as
\begin{align}\label{def:LFMO eq}
T_i := \inf \left\{ t \geq 0 \ : \ L_t > \varepsilon_i \right\}, \qquad i = 1, \ldots, n,
\end{align}
where $\mathbf L = (L_t : t \geq 0)$ is a L\'evy subordinator stochastic process with $L_0=0$, and $\varepsilon_1, \ldots, \varepsilon_n$ is a collection of $n$ iid standard exponential random variables, that are independent of $\mathbf L$.
\end{definition}
%%%%%%%%%%%%%%%%%%%%%%%%%%%%%%%%%%%%%%%%%%%%%%%%%%%%

In Figure~\ref{fig:LFMO ex} we show a simulation of a LFMO distributed vector with $n=3$ components. Note that the L\'evy subordinator process $\mathbf L$ acts as a common degradation process that ``kills'' each component, say~$i$, once it crosses its corresponding individual tolerance, which from now on we will refer to as the \textit{trigger}
$\varepsilon_i$, to evoke the fact that the threshold $\varepsilon_i$ triggers the failure of component $i$.

%%%%%%%%%%%%%%%%%%%%%%%%%%%%%%%%%%%%%%%%%%%%%%%%%%%%
\begin{figure}[h]%[hbtp]
\centering
\begin{tikzpicture}[
    component/.style={draw, circle, inner sep=1pt, minimum size=0.6cm},
    system/.style={rounded corners=5mm, minimum width=2.5cm, minimum height=2.5cm, align=center},
    scale=1.
    ]
    % Define coordinates and parameters
    \def\xmax{8}
    \def\ymax{4}
    \def\tsys{1.5}
    \def\ttwo{7}
    
    % Draw axes
    \draw[->] (0,0) -- (\xmax+0.5,0) node[below] {Time};
    \draw[->] (0,0) -- (0,\ymax+0.5);
    
    % Define epsilon levels
    \def\epsilonthree{0.6}
    \def\epsilonone{0.8}
    \def\epsilontwo{3.3}
    
    % Draw horizontal reference lines
    \draw[gray, thick] (0,\epsilontwo) node[left, black] {$\varepsilon_2$} (0,\epsilontwo) -- (\xmax,\epsilontwo);
    \draw[gray, thick] (0,\epsilonone) node[left, black] {$\varepsilon_1$} (0,\epsilonone) -- (\xmax,\epsilonone);
    \draw[gray, thick] (0,\epsilonthree) node[left, black] {$\varepsilon_3$} (0,\epsilonthree) -- (\xmax,\epsilonthree); %node[right, black] {$\varepsilon_3$};
    
    % Draw the Lévy subordinator process
    \draw[blue, very thick] (0,0) -- (\tsys,0.3);
    \draw[blue, very thick, dashed] (\tsys,0.3) -- (\tsys,1.2);
    \draw[blue, very thick] (\tsys,1.2) -- (2.8,1.5);
    \draw[blue, very thick, dashed] (2.8,1.5) -- (2.8,2.1);
    \draw[blue, very thick] (2.8,2.1) -- (4.2,2.4);
    \draw[blue, very thick] (4.2,2.4) -- (5.5,2.7);
    \draw[blue, very thick] (5.5,2.7) -- (\ttwo,3.1);
    \draw[blue, very thick, dashed] (\ttwo,3.1) -- (\ttwo,3.7);
    \draw[blue, very thick] (\ttwo,3.7) -- (\xmax,3.9);
    
    % Add jump points
    %\fill[blue] (\tsys,0.3) circle (1.5pt);
    \fill[blue] (\tsys,1.2) circle (1.5pt);
    %\fill[blue] (2.8,1.5) circle (1.5pt);
    \fill[blue] (2.8,2.1) circle (1.5pt);
    %\fill[blue] (\ttwo,3.1) circle (1.5pt);
    \fill[blue] (\ttwo,3.7) circle (1.5pt);
    
    % Draw vertical red line at T_sys
    \draw[black, thick] (\tsys,0) -- (\tsys,.3);
    %\draw[red, very thick] (\tsys,0) -- (\tsys,\ymax);
    
    % Draw vertical line at T_2
    \draw[black, thick] (\ttwo,0) -- (\ttwo,3.7);
    %\draw[black, thick] (\ttwo,0) -- (\ttwo,\ymax);
    
    % Add time labels
    \node[below] at (\tsys,0) {$T_3 = T_1 = T_{fail}$};
    \node[below] at (\ttwo,0) {$T_2$};
    
    % Add legend
    \node[draw, rectangle, fill=white, anchor=south west] at (0.5,\ymax-.2) {
        %\begin{tabular}{l}
        \textcolor{blue}{{\bf ---}} L\'evy subordinator process $L$
        %\end{tabular}
    };

% First row, first system (all working)
\node[system] (sys111) at (10, 2) {};
\node[component] (c1_1) at (sys111.north) [xshift=-0.5cm, yshift=-0.5cm] {1};
\node[component] (c1_2) at (sys111.north) [xshift= 0.5cm, yshift=-0.5cm] {2};
\node[component] (c1_3) at (sys111.south) [yshift= 0.5cm] {3};
\draw (c1_1.east) -- (c1_2.west);
\draw (sys111.east) -| +(-.3,0) |- (c1_2.east);
\draw (sys111.east) -| +(-.3,0) |- (c1_3.east);
\draw (sys111.west) -| +(.3,0) |- (c1_3.west);
\draw (sys111.west) -| +(.3,0) |- (c1_1.west);

\end{tikzpicture}

\caption{
    A simulation of the random vector $(T_1, T_2, T_3)$ in $\R^3$ with an LFMO distribution: the failure time $T_i$ of component $i$ is the first time the L\'evy subordinator process $L$ surpasses the trigger $\varepsilon_i$. Note that two components fail simultaneously at time $T_1$: due to a jump of the L\'evy subordinator, components 1 and 3 fail at the same time, $T_1=T_3$. \emph{This is how  the LFMO model induces simultaneous failures of components.} Also, these are the failure times of the components of the system on the right, which works when there is a path of working components from left to right. Hence, the system failure time $\Tsys$ is also $T_1=T_3$.
%They correspond to the failure times of the system in the diagram in the right, that works as long as there is a path of working components from left to right. Components 1 and 3 fail first and at the same time, making the system fail at time $\Tsys$.
}
\label{fig:LFMO ex}
\end{figure}
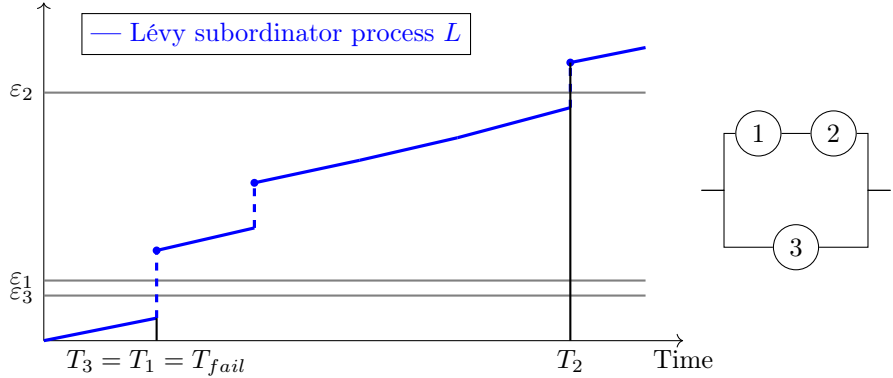
%%%%%%%%%%%%%%%%%%%%%%%%%%%%%%%%%%%%

Some important properties of the LFMO distribution are the following.
First, marginally, the failure times of the components $T_1$, \ldots, $T_n$ are exponentially distributed with common mean $\E[T_i] = 1/\Psi(1)$.
Here, $\Psi$ is the so-called \emph{Laplace exponent function} $\Psi$ of the L\'evy subordinator $\mathbf L$, defined as $\Psi(x) = -\log \E [ e^{-x L_1} ]$ for all $x$ in $\R$; see below for examples, and~\cite{kyprianou2014fluctuations} for further details.
Second, as argued in~\cite[p.~103]{matthias2017simulating}, the LFMO distribution is equal in distribution to the MO random vector in~\eqref{def:MO eq} when the rates $\lambda_V$ are chosen as
\begin{align}\label{def:LV}
	\lambda^{(n)}_V := & \sum_{i=0}^{|V|-1} \binom{|V|-1}i (-1)^i \left(\Psi(n-|V|+i+1)-\Psi(n-|V|+i)\right) \\
    = & \sum_{j=n-|V|}^{n} (-1)^{j-n+|V|+1} \binom{|V|}{n-j} \Psi(j). \label{def:LV 2}
\end{align}
 Note that $\lambda^{(n)}_V$ only depends on $|V|$, hence onward we will abuse notation and write $\lambda_{|V|}^{(n)}$ instead of $\lambda^{(n)}_V$. % for $|V|=1,\dots,n$.
 Another important property is that an LFMO distributed random vector $(T_1,\dots,T_n)$ has an EXC distribution.
 In fact it is \emph{conditionally-iid}, as conditional on the path of~$\mathbf L$, the times $T_i$ are iid; see~\cite[Section~3.2]{matthias2017simulating}.
 Also, the joint reliability function of $(T_1,\dots,T_n)$ is
\begin{align*}
\P(T_1>t_1,\dots,T_n>t_n) = & \exp\left(-\sum_{i=1}^n (\Psi(i)-\Psi(i-1)) t_{n-i+1:n}\right) \\
= & \exp\left(-\sum_{\emptyset\neq V\subset\{1,\dots,n\}} \lambda^{(n)}_V \max_{i\in V}t_{i}\right),
\end{align*}
for $t_1,\dots,t_n\geq 0$, where $\Psi(0)=0$ and $0\leq t_{1:n}\leq\dots\leq t_{n:n}$ are the ordered values obtained from $t_1,\dots,t_n$; see~\cite{mai2013sampling}.

 More important for reliability modeling, the LFMO distribution has low parametric complexity: the only parameters needed are the number $n$ of components of the system and the L\'evy subordinator process $\mathbf L$, which in turn is fully characterized by its Laplace exponent function $\Psi(x)$, see~\cite[Chapter 1]{kyprianou2014fluctuations}.
 Even more, the only information of the LFMO distribution that is needed for the results in this paper is the values of $\Psi(1)$, \ldots, $\Psi(n)$.

Some examples of L\'evy subordinator $\mathbf L$ are the following. The most basic one is choosing $L_t = \mu t$ for some $\mu>0$.
%A basic example for the L\'evy subordinator $\mathbf L$ is the one of iid exponential failures --- 
In this case we obtain that the components' lifetimes are iid random variables with an exponential distribution and rate $\mu$.

Another important class is the family %of $\mathbf L$ is the case
of compound Poisson processes (CPP) with non-negative drift $\mu \geq 0$, rate $\lambda>0$ and some non-negative jumps distribution, say $\mathcal{J}$; i.e., when the degradation process is $L_t = \mu t + \sum_{j=1}^{N_t} J_j$ for all $t>0$, where $\mathbf N$ is a Poisson process with rate $\lambda>0$ and $J_j$ are iid non-negative random variables with distribution $\mathcal{J}$.
In that case, the Laplace exponent function is $\Psi(x) = \mu x + \lambda (1-\E [e^{-x J_1}])$. This case can be informally described as there being a steady degradation making that, nominally, each component fails after an expected time of $1/\mu$; however, on average every $1/\lambda$ time units there is a degradation shock that can make several components fail at once, and where each surviving component fails with probability $p:=1-\E[e^{-J_1}]$. In fact, for $1 \leq l \leq k \leq n$, if there are $k$ surviving components, a degradation jump $J_j$ kills $l$ components with probability $\binom{k}{l} p^l(1-p)^{k-l}$.
We note that compound Poisson processes are referred to as \emph{finite activity} processes, as the number of jumps in any finite time interval is finite almost surely.

In contrast, some classical examples of \emph{infinite activity} subordinators (i.e., almost surely infinite jumps on any finite time interval) are Gamma processes, whose Laplace exponent is $\Psi(x) = \beta \log (1+x/\eta)$ for parameters $\beta, \eta > 0$; inverse Gaussian processes, with $\Psi(x) = \beta (\sqrt{ 2x+ \eta^2}-\eta)$ for $\beta, \eta > 0$; and stable subordinators for $\alpha$ in $(0,1)$, with $\Psi(x)=x^\alpha$.
Importantly, even though infinite activity subordinators cannot be simulated without bias due to the discrete nature of computers, the only information we need to compute our results are the values of $\Psi(1)$, \ldots, $\Psi(n)$.
See~\cite[Annex A.2]{matthias2017simulating} for a brief summary and properties of subordinators.

%Finally, other L\'evy subordinators are Gamma processes, with Laplace exponent $\Psi(x) = \beta \log (1+x/\eta)$ for parameters $\beta, \eta > 0$, usually used to model degradation processes, see, e.g.,~\cite{}. inverse Gaussian subordinator, with Laplace exponent $\Psi(x) = \beta (\sqrt{ 2x+ \eta^2}-\eta)$ for parameters $\beta, \eta > 0$. stable subordinator for $\alpha$ in $(0,1)$, with Laplace exponent $\Psi(x)=x^\alpha$

%%%%%%%%%%%%%%%%%%%%%%%%%%%%%%%%%%%%%%%%%%%%%%%%%%%%
\paragraph{Order statistics of the LFMO distribution}
%%%%%%%%%%%%%%%%%%%%%%%%%%%%%%%%%%%%%%%%%%%%%%%%%%%%
An element of the LFMO distribution that is crucial for our work is the order statistics of the components' failure times $T_1, \ldots, T_n$, denoted by $T_{1:n}, \ldots, T_{n:n}$. That is, $\{T_1, \ldots, T_n\} = \{T_{1:n}, \ldots, T_{n:n} \}$ and $T_{1:n} \leq \ldots \leq T_{n:n}$. We use the same colon-subscript convention also for ordered \emph{deterministic} values $t_1, \ldots, t_n \in \R$, writing $t_{1:n} \leq \ldots \leq t_{n:n}$ for their ordered re-arrangement (e.g.\ in the joint reliability function above); the two usages are notationally parallel but refer to different objects (random variables versus deterministic values).
From the reliability perspective, $T_{k:n}$ is the time of failure of the $k$-th component that fails.
From~\cite[Proposition 1]{barrera2020limit} we have that %\footnote{$\P( T_{k:n}>t ) = \P\left( L_t < \varepsilon_{k:n} \right) = \E \left[ \sum_{i=0}^{k-1} \binom{n}{i} \P(L_t \geq \varepsilon_1 | \mathbf L)^i \P(L_t < \varepsilon_1 | \mathbf L)^{n-i} \right] = \sum_{i=0}^{k-1} \binom{n}{i} \E \left[(1-e^{-L_t})^i (e^{-L_t})^{n-i} \right]$} $\P( T_{k:n}>t) = \sum_{i = n-k+1}^n \binom{n}{i} \binom{i-1}{n-k} (-1)^{i-n+k-1} e^{-t \Psi(i)}$ for $1 \leq k \leq n$, and
%\begin{align}\label{eq:PTkn}
\begin{align}
\nonumber
\P( T_{k:n}>t) & = \sum_{i = n-k+1}^n \binom{n}{i} \binom{i-1}{n-k} (-1)^{i-n+k-1} e^{-t \Psi(i)} \\
\label{eq:ETkn}
\E T_{k:n} & = \sum_{i = n-k+1}^n \binom{n}{i} \binom{i-1}{n-k} (-1)^{i-n+k-1}\frac{1}{\Psi(i)}
\end{align}
for $1 \leq k \leq n$. %, since $\P( T_{k:n}>t ) = \P\left( L_t < \varepsilon_{k:n} \right) = \E \left[ \sum_{i=0}^{k-1} \binom{n}{i} \P(L_t \geq \varepsilon_1 | \mathbf L)^i \P(L_t < \varepsilon_1 | \mathbf L)^{n-i} \right] = \sum_{i=0}^{k-1} \binom{n}{i} \E \left[(1-e^{-L_t})^i (e^{-L_t})^{n-i} \right]$.
See~\cite{barrera2020limit,barrera2020approximating} for further results on the approximation of lower-, middle- and upper-order statistics of the LFMO distribution.

In particular, note that using the latter notation, $T_{1:n}$ is the failure time of a series system with $n$ components whose lifetimes are LFMO distributed. In this case, the reliability function is 
$\P(T_{1:n}>t) = \exp\left(-\Psi(n)t\right)$
for $t\geq 0$, and its mean time-to-failure (MTTF) is  $\E T_{1:n} = 1/\Psi(n)$.
In general, the series system of $k$ components, $1\leq k\leq n$, has lifetime $T_{1:k}$ that is exponentially distributed with mean $1/\Psi(k)$.
%In particular, thus, $T_1,\dots,T_n$ have  exponential distributions with  common  mean $E(T_i)=1/\Psi(1)$.

%%%%%%%%%%%%%%%%%%%%%%%%%%%%%%%%%%%%%%%%%%%%%%%%%%%%
\paragraph{Markov chains of number of failed components}
%%%%%%%%%%%%%%%%%%%%%%%%%%%%%%%%%%%%%%%%%%%%%%%%%%%%
We now consider two Markov chains associated to the number of failed components; see Figure~\ref{fig:number} for an illustrative example. Their analysis and associated probabilities are key for our results.

%%%%%%%%%%%%%%%%%%%%%%%%%%%%%%%%%%%%%%%%%%%%%%%%%%%%

\begin{figure}[h]%[hbtp]
\centering

\scalebox{0.5}{
\begin{tikzpicture}[
    component/.style={draw, circle, inner sep=1pt, minimum size=0.6cm},
    %system/.style={draw=green!60!black, fill=green!10, rounded corners=5mm, minimum width=2.5cm, minimum height=2.5cm, align=center},
    %failed_system/.style={draw=red!60!black, fill=red!10, rounded corners=5mm, minimum width=2.5cm, minimum height=2.5cm, align=center},
    system/.style={fill=green!10, rounded corners=5mm, minimum width=2.5cm, minimum height=2.5cm, align=center},
    failed_system/.style={fill=red!10, rounded corners=5mm, minimum width=2.5cm, minimum height=2.5cm, align=center},
    cross out/.pic={
        \draw[red, very thick] (-0.3, -0.3) -- (0.3, 0.3);
        \draw[red, very thick] (-0.3, 0.3) -- (0.3, -0.3);
    }
]

% First row, first system (all working)
\node[system] (sys111) at (0, -4) {};
\node[component] (c1_1) at (sys111.north) [xshift=-0.5cm, yshift=-0.5cm] {1};
\node[component] (c1_2) at (sys111.north) [xshift= 0.5cm, yshift=-0.5cm] {2};
\node[component] (c1_3) at (sys111.south) [yshift= 0.5cm] {3};
\draw (c1_1.east) -- (c1_2.west);
\draw (sys111.east) -| +(-.3,0) |- (c1_2.east);
\draw (sys111.east) -| +(-.3,0) |- (c1_3.east);
\draw (sys111.west) -| +(.3,0) |- (c1_3.west);
\draw (sys111.west) -| +(.3,0) |- (c1_1.west);

% First row, second system (1 failed)
\node[system] (sys011) at (4, 0) {};
\node[component] (c2_1) at (sys011.north) [xshift=-0.5cm, yshift=-0.5cm] {1};
\node[component] (c2_2) at (sys011.north) [xshift= 0.5cm, yshift=-0.5cm] {2};
\node[component] (c2_3) at (sys011.south) [yshift= 0.5cm] {3};
\draw (c2_1.east) -- (c2_2.west);
\draw (sys011.east) -| +(-.3,0) |- (c2_2.east);
\draw (sys011.east) -| +(-.3,0) |- (c2_3.east);
\draw (sys011.west) -| +( .3,0) |- (c2_3.west);
\draw (sys011.west) -| +( .3,0) |- (c2_1.west);
\pic at (c2_1) {cross out};

\node[failed_system] (sys100) at (8, -0) {};
\node[component] (c9_1) at (sys100.north) [xshift=-0.5cm, yshift=-0.5cm] {1};
\node[component] (c9_2) at (sys100.north) [xshift= 0.5cm, yshift=-0.5cm] {2};
\node[component] (c9_3) at (sys100.south) [yshift= 0.5cm] {3};
\draw (c9_1.east) -- (c9_2.west);
\draw (sys100.east) -| +(-.3,0) |- (c9_2.east);
\draw (sys100.east) -| +(-.3,0) |- (c9_3.east);
\draw (sys100.west) -| +( .3,0) |- (c9_3.west);
\draw (sys100.west) -| +( .3,0) |- (c9_1.west);
\pic at (c9_2) {cross out};
\pic at (c9_3) {cross out};

% First row, fourth system (3 failed)
\node[failed_system] (sys000) at (12, -4) {};
\node[component] (c4_1) at (sys000.north) [xshift=-0.5cm, yshift=-0.5cm] {1};
\node[component] (c4_2) at (sys000.north) [xshift= 0.5cm, yshift=-0.5cm] {2};
\node[component] (c4_3) at (sys000.south) [yshift= 0.5cm] {3};
\draw (c4_1.east) -- (c4_2.west);
\draw (sys000.east) -| +(-.3,0) |- (c4_2.east);
\draw (sys000.east) -| +(-.3,0) |- (c4_3.east);
\draw (sys000.west) -| +( .3,0) |- (c4_3.west);
\draw (sys000.west) -| +( .3,0) |- (c4_1.west);
\pic at (c4_1) {cross out};
\pic at (c4_2) {cross out};
\pic at (c4_3) {cross out};

% Second row, first system (2 failed)
\node[system] (sys101) at (4, -4) {};
\node[component] (c5_1) at (sys101.north) [xshift=-0.5cm, yshift=-0.5cm] {1};
\node[component] (c5_2) at (sys101.north) [xshift= 0.5cm, yshift=-0.5cm] {2};
\node[component] (c5_3) at (sys101.south) [yshift= 0.5cm] {3};
\draw (c5_1.east) -- (c5_2.west);
\draw (sys101.east) -| +(-.3,0) |- (c5_2.east);
\draw (sys101.east) -| +(-.3,0) |- (c5_3.east);
\draw (sys101.west) -| +( .3,0) |- (c5_3.west);
\draw (sys101.west) -| +( .3,0) |- (c5_1.west);
\pic at (c5_2) {cross out};

% Second row, second system (2,3 failed)
\node[failed_system] (sys010) at (8, -4) {};
\node[component] (c6_1) at (sys010.north) [xshift=-0.5cm, yshift=-0.5cm] {1};
\node[component] (c6_2) at (sys010.north) [xshift= 0.5cm, yshift=-0.5cm] {2};
\node[component] (c6_3) at (sys010.south) [yshift= 0.5cm] {3};
\draw (c6_1.east) -- (c6_2.west);
\draw (sys010.east) -| +(-.3,0) |- (c6_2.east);
\draw (sys010.east) -| +(-.3,0) |- (c6_3.east);
\draw (sys010.west) -| +( .3,0) |- (c6_3.west);
\draw (sys010.west) -| +( .3,0) |- (c6_1.west);
\pic at (c6_1) {cross out};
\pic at (c6_3) {cross out};

% Second row, third system (1,3 failed)
\node[system] (sys001) at (8, -8) {};
\node[component] (c7_1) at (sys001.north) [xshift=-0.5cm, yshift=-0.5cm] {1};
\node[component] (c7_2) at (sys001.north) [xshift= 0.5cm, yshift=-0.5cm] {2};
\node[component] (c7_3) at (sys001.south) [yshift= 0.5cm] {3};
\draw (c7_1.east) -- (c7_2.west);
\draw (sys001.east) -| +(-.3,0) |- (c7_2.east);
\draw (sys001.east) -| +(-.3,0) |- (c7_3.east);
\draw (sys001.west) -| +( .3,0) |- (c7_3.west);
\draw (sys001.west) -| +( .3,0) |- (c7_1.west);
\pic at (c7_1) {cross out};
\pic at (c7_2) {cross out};

% Third row, first system (3 failed)
\node[system] (sys110) at (4, -8) {};
\node[component] (c8_1) at (sys110.north) [xshift=-0.5cm, yshift=-0.5cm] {1};
\node[component] (c8_2) at (sys110.north) [xshift= 0.5cm, yshift=-0.5cm] {2};
\node[component] (c8_3) at (sys110.south) [yshift= 0.5cm] {3};
\draw (c8_1.east) -- (c8_2.west);
\draw (sys110.east) -| +(-.3,0) |- (c8_2.east);
\draw (sys110.east) -| +(-.3,0) |- (c8_3.east);
\draw (sys110.west) -| +( .3,0) |- (c8_3.west);
\draw (sys110.west) -| +( .3,0) |- (c8_1.west);
\pic at (c8_3) {cross out};

%%%%%%%%%%%%%%%%

\draw[->, cyan, thick, dashed] (sys111) .. controls +(2,2) and +(-2,-2) .. (sys011);
\draw[->, cyan, thick, dashed] (sys111) .. controls +(2,-.3) and +(-2,.5) .. (sys101);
\draw[->, cyan, thick, dashed] (sys111) .. controls +(2,-2) and +(-2,2) .. (sys110);
\draw[->, cyan, thick, dotted] (sys111) .. controls +(2,1.7) and +(-2,-2) .. (sys100);
\draw[->, cyan, thick, dotted] (sys111) .. controls +(2,-.6) and +(-2,.5) .. (sys010);
\draw[->, cyan, thick, dotted] (sys111) .. controls +(2,-1.7) and +(-2,2) .. (sys001);
\draw[->, cyan, thick, densely dashdotted] (sys111) .. controls +(2,-.9) and +(-2,.5) .. (sys000);

\draw[->, thick, magenta, dashed] (sys011) .. controls +(2,-1.0) and +(-2,2) .. (sys010);
\draw[->, thick, magenta, dashed] (sys011) .. controls +(2,-1.5) and +(-2,2) .. (sys001);
\draw[->, thick, magenta, dotted] (sys011) .. controls +(2,-0.5) and +(-2,2) .. (sys000);

\draw[->, thick, magenta, dashed] (sys110) .. controls +(2,1.5) and +(-2,-2) .. (sys100);
\draw[->, thick, magenta, dashed] (sys110) .. controls +(2,1.0) and +(-2,-2) .. (sys010);
\draw[->, thick, magenta, dotted] (sys110) .. controls +(2,0.5) and +(-2,-2) .. (sys000);

\draw[->, thick, magenta, dashed] (sys101) .. controls +(2,1.5) and +(-2,-2) .. (sys100);
\draw[->, thick, magenta, dashed] (sys101) .. controls +(2,-1.5) and +(-2,2) .. (sys001);
\draw[->, thick, magenta, dotted] (sys101) .. controls +(2,-0.8) and +(-2,0.5) .. (sys000);

\draw[->, thick, teal, dashed] (sys100) .. controls +(2,-2) and +(-2,2) .. (sys000);
\draw[->, thick, teal, dashed] (sys010) .. controls +(2,-.5) and +(-2,.5) .. (sys000);
\draw[->, thick, teal, dashed] (sys001) .. controls +(2,2) and +(-2,-2) .. (sys000);

%%%%%%%%%%%%%%%%

\draw [thick, rounded corners=.5cm] (   -1.4, 1.5) -- (   -1.4, -9.5) -- (   1.4, -9.5) -- (   1.4, 1.5) -- cycle;
\draw [thick, rounded corners=.5cm] ( -1.4+4, 1.5) -- ( -1.4+4, -9.5) -- ( 1.4+4, -9.5) -- ( 1.4+4, 1.5) -- cycle;
\draw [thick, rounded corners=.5cm] ( -1.4+8, 1.5) -- ( -1.4+8, -9.5) -- ( 1.4+8, -9.5) -- ( 1.4+8, 1.5) -- cycle;
\draw [thick, rounded corners=.5cm] (-1.4+12, 1.5) -- (-1.4+12, -9.5) -- (1.4+12, -9.5) -- (1.4+12, 1.5) -- cycle;

%\draw[line width=4pt, -{Triangle[length=2pt, width=10pt]}] (0,0) -- (0,-4);
\draw[line width=8pt, -{Triangle[width=15pt,length=10pt]}] (0,-9.5) -- (0,-9.5-.9);
\draw[line width=8pt, -{Triangle[width=15pt,length=10pt]}] (4,-9.5) -- (4,-9.5-.9);
\draw[line width=8pt, -{Triangle[width=15pt,length=10pt]}] (8,-9.5) -- (8,-9.5-.9);
\draw[line width=8pt, -{Triangle[width=15pt,length=10pt]}] (12,-9.5) -- (12,-9.5-.9);

%%%%%%%%%%%%%%%%

\draw (0,-12-0.6) circle (1.3cm) node {\begin{tabular}{c} 0 failed \\ components \end{tabular}} ;
\draw (4,-12-0.6) circle (1.3cm) node {\begin{tabular}{c} 1 failed \\ components \end{tabular}} ;
\draw (8,-12-0.6) circle (1.3cm) node {\begin{tabular}{c} 2 failed \\ components \end{tabular}} ;
\draw (12,-12-0.6) circle (1.3cm) node {\begin{tabular}{c} 3 failed \\ components \end{tabular}} ;

\draw[->, cyan, thick, densely dashdotted]   (0+1,-12+1-0.7) .. controls +(2.5,2.5) and +(-.5,.5) .. (12-1,-12+1-0.7);
\draw[->, cyan, ultra thick, densely dotted]   (0+1,-12+1-0.7) .. controls +(1.5,1.5) and +(-.5,.5) .. (8-1,-12+1-0.7);
\draw[->, cyan, ultra thick, dashed]   (0+1,-12+1-0.7) .. controls +(.5,.5) and +(-.5,.5) .. (4-1,-12+1-0.7);

\draw[->, magenta, ultra thick, densely dotted] (4+1,-12+1-0.7) .. controls +(1.5,1.5) and +(-.5,.5) .. (12-1,-12+1-0.7);
\draw[->, magenta, ultra thick, dashed] (4+1,-12+1-0.7) .. controls +(.5,.5) and +(-.4,.4) .. (8-1,-12+1-0.7);

\draw[->, teal, ultra thick, dashed] (8+1,-12+1-0.7) .. controls +(.5,.5) and +(-.4,.4) .. (12-1,-12+1-0.7);

%%%%%%%%%%%%%%%%

\end{tikzpicture}
}%
\caption{
    The Markov chain $\mathbf \Nc = (\Nc(t) \, : \, t \geq 0)$ of number of failed components (bottom) is obtained by aggregating the original states and transitions of the system (top); this is formalized in Lemma~\ref{lemma1}. For example, Lemma~3 states that, indeed, the rate of the blue dotted transition from 0 to 2 failed components in the bottom, is obtained by summing the $\binom{3-0}{2-0} = 3$ blue dotted transition rates in the top. Similarly, the time $T_{k:n}$ until $k$ components have failed, is just the time required to go from $0$ to $k$ in the chain at the bottom; hence, the event $N(T_{r:n}) = k$ of the system having $k$ failed components at the time of the $r$-th failure, is the event of the chain jumping from $\{ 0, \ldots, r-1\}$ to $k$ failed components, see Lemma~3 part 5.
}
\label{fig:number}

\end{figure}
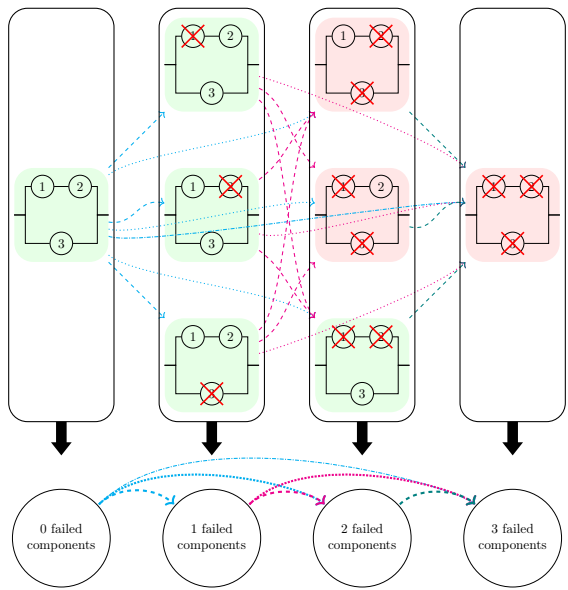
%%%%%%%%%%%%%%%%%%%%%%%%%%%%%%%%%%%%

Consider a system with $n$ components, whose lifetimes follow the LFMO distribution in~\eqref{def:LFMO eq}.
Denote by $\mathbf \Nc = (\Nc(t) \, : \, t \geq 0)$ the continuous-time stochastic process that counts the number of failed components in the system at each time $t \geq 0$, i.e., $\Nc(t) := \sum_{i=1}^n \I{L_t \geq \varepsilon_i}$, where $\mathbf L$ and $\varepsilon_i$ are, respectively, the L\'evy subordinator process and the triggers in Definition~\ref{def:LFMO}. From~\eqref{def:LFMO eq} and $e^{-t \Psi(x)} = \E [ e^{-x L_t} ]$, it holds that\footnote{$\P\left( \Nc(t)=k \right) = \E \left[ \P\left( \Nc(t)=k | \mathbf L\right) \right] = \E \left[ \binom{n}{k} \P(L_t \geq \varepsilon_1 | \mathbf L)^k \P(L_t < \varepsilon_1 | \mathbf L)^{n-k} \right] = \binom{n}{k} \E \left[(1-e^{-L_t})^k (e^{-L_t})^{n-k} \right]$}
\begin{align*}%\label{eq:PN}
\P\left( \Nc(t)=k \right) = \sum_{i=n-k}^n \binom{n}{i} \binom{i}{n-k} (-1)^{i-n+k} e^{-t \Psi(i)}.
\end{align*}
%\begin{eqnarray} \P\left( \Nc(t)=k \right) = \binom{n}{k} \E \left[(e^{L_t}-1)^k e^{-n L_t } \right] = \binom{n}{k} \E \left[\sum_{i=0}^k \binom{k}{i} (-1)^{k-i} e^{-(n-i) L_t } \right] = \binom{n}{k} \sum_{i=0}^k \binom{k}{i} (-1)^{k-i} e^{-t \Psi(n-i)} \\
%= \sum_{j=n-k}^n \binom{n}{k} \binom{k}{n-j} (-1)^{k-n+j} e^{-t \Psi(j)} = \sum_{i=n-k}^n \frac{n!}{(n-i)! (i-n+k)! (n-k)!} (-1)^{i-n+k} e^{-t \Psi(i)} \\
%= \sum_{i=n-k}^n \binom{n}{i} \binom{i}{n-k} (-1)^{i-n+k} e^{-t \Psi(i)}\end{eqnarray}
Alternatively, $\P\left( T_{k:n}>t \right) = \P\left( \Nc(t)<k \right)$ and $\P\left( \Nc(t)=k \right) = \P( T_{k+1:n}>t ) - \P( T_{k:n}>t )$ with $T_{n+1:n} = +\infty$.
In fact, $\mathbf \Nc$ is a continuous-time Markov chain, due to the memoryless and Markov properties, respectively, of the triggers $\varepsilon_i$ and the L\'evy subordinator $\mathbf L$.
%We assume that initially all components are working, so $\Nc(0) = 0$.
It also holds that $\lim_{t \to +\infty} \Nc(t) = n$.

An important fact, that we use profusely in our results, is that there is a ``dual'' relation between the ordered failure times $T_{k:n}$, $k=1,\ldots,n$, and the number $N(T_{k:n})$ of failed components at these times. Indeed, it holds that $\{ N(T_{r:n}) < k \} = \{ T_{r:n} < T_{k:n} \}$, $\{ N(T_{r:n}) \geq k \} = \{ T_{r:n} \geq T_{k:n} \}$ and $\{ N(T_{r:n}) = k \} = \{ T_{r:n} = T_{k:n} < T_{k+1:n} \}$.
The results~\eqref{PTrnTknl}, \eqref{PTrnTkne} and~\eqref{PTkn eq j} below can be viewed as probabilities exploiting this connection.

We also consider the associated discrete-time Markov chain $(\Nd_i \, : \, i=0,1,\ldots)$ that counts the number of failed components after the $i$-th time of failure of components, i.e., $\Nd_i = \Nc(\tau_i)$ where $\tau_0 :=0$ and $\tau_{i+1} := \inf \{ t > \tau_{i} : \Delta \Nc(t) := \Nc(t) - \Nc(t^-) > 0 \}$ for all $i=0,1,\ldots, n-1$, with the convention that $\inf \emptyset = +\infty$.
Note that $0 < \tau_1 < \tau_2 < \ldots$ and $\tau_1$, $\tau_2$, \ldots correspond to the times at which \emph{one or more} components fail, because several components can fail simultaneously. For instance, in the example shown in Figure~\ref{fig:LFMO ex} we see that $\tau_1 = T_1 = T_3 < \tau_2 = T_2$, so $\Nd_0 = 0$, $\Nd_1 = 2$ and $\Nd_2 = 3$.
It holds that $\Nd$ is a discrete-time Markov chain over the state space $\{0, 1, \ldots, n\}$, and the chain is absorbed at state $n$.

The following lemma establishes several key results and quantities that we use in our results.
Its proof is deferred to Section~\ref{sec:proofs}.
%for the Markov chains $N$ and $\Nd$ of number of failed components, which are later used in our main results in Section~\ref{sec:main}.

%In this process the probability of to go from the state with $i$ failures to the  state with $j>i$ failures is given by 
%\begin{align}
%	p_{i,j}=\binom{n-i}{j-i}\frac{1}{\Psi(n-i)}\lambda_{j-i}^{(n-i)} 
%\end{align}
%for $0\leq i< j\leq n$.

%%%%%%%%%%%%%%%%%%%%%%%%%%%%%%%%%%%%%%%%%%%%%%%%%%%%
\begin{lemma}\label{lemma1}
Consider a system with $n$ components where the components' lifetimes follow a LFMO distribution with Laplace exponent function $\Psi$. %(x) = - \log E[\exp(-x L_1)]$ for all $x$ in $\R$, where $\mathbf L$ is the subordinator process in~\eqref{def:LFMO eq}.
%Recall the definition~\eqref{def:LV} for the rates $\lambda^{(l)}_{k}$, for any $1 \leq k \leq l \leq n$.
%For any $1 \leq d_1 \leq d_2$ define the rate
%\begin{align}
%\lambda^{(d_2)}_{d_1} & := \sum_{l = 0}^{d_1-1} (-1)^{l} \binom{d_1-1}{l} \left( \Psi(d_2-d_1+l+1) - \Psi(d_2-d_1+l) \right).
%\end{align}

\begin{enumerate}
\item For the continuous-time Markov chain $\mathbf \Nc = (\Nc(t) \, : \, t \geq 0)$ of number of failed components at each time, the rate of the transition from having $i$ to $j$ failed components is
$$\binom{n-i}{j-i} \lambda_{j-i}^{(n-i)} = \sum_{k=i}^{j} (-1)^{j-k+1} \frac{\binom{n}{j} \binom{j}{k} \binom{k}{i} }{ \binom{n}{i}} \Psi(n-k)$$ for $0 \leq i < j \leq n$ and zero otherwise, 
where the rate $\lambda_{j-i}^{(n-i)}$ was defined in \eqref{def:LV}.
That is, for $i<j$, given that the system with $n$ components has $i$ of them failed, the time until the next arrival of any shock that simultaneously hits $j-i$ of the remaining $n-i$ working components, is distributed exponentially with rate $\binom{n-i}{j-i} \lambda_{j-i}^{(n-i)}$.
%Given that the system with $n$ components has $i$ of them failed, the time until the next failure of any working component is distributed exponential with rate $\Psi(n-i)$, i.e., $T_{i+1:n}-t | \Nc(t) = i \sim expo(\Psi(n-i))$. Also, given that the system with $n$ components has $i$ of them failed, the time until the next arrival of any shock that simultaneously hits $j-i$ of the remaining $i$ working components (for $i < j \leq n$), is distributed exponential with rate $\binom{n-i}{j-i} \lambda_{j-i}^{(n-i)}$.

\item For the discrete-time Markov chain $\mathbf \Nd = (\Nd_i \, : \, i=0,1,\ldots)$ of number of failed components after the $i$-th time of failure of components, the probability $P_{i,j}$ of making a (one-step) transition from having $i$ to $j$ failed components (i.e., given that the system with $n$ components has $i$ of them failed, next arrives a shock that hits $j-i$ of the working components) is
\begin{align}\label{TP}
P_{i, j} = & \binom{n-i}{j-i} \frac{ \lambda^{(n-i)}_{j-i} }{ \Psi(n-i) } = \sum_{k=i}^{j} (-1)^{j-k+1} \frac{\binom{n}{j} \binom{j}{k} \binom{k}{i} }{ \binom{n}{i}} \frac{\Psi(n-k)}{\Psi(n-i)}
%\\ & \sum_{l=n-j}^{n-i} (-1)^{l-(n-j)-1} \binom{n-i}{l} \binom{l}{n-j} \frac{\Psi(l)}{\Psi(n-i)} = \sum_{k=i}^{j} (-1)^{j-k-1} \frac{\binom{n}{j}}{\binom{n}{i}} \binom{k}{i} \binom{j}{k} \frac{\Psi(n-k)}{\Psi(n-i)} 
\end{align}
for $0 \leq i < j \leq n$ and zero otherwise.

\item For $0 \leq i < j \leq k \leq n$, denote by $\QQ{i}{j}$ (respectively, $\QQQ{i}{j}{k}$) the probability that the discrete-time chain $\mathbf \Nd$ of number of failed components, goes from having zero to $i$ failed components in one or more steps, and then jumps \emph{in one step} from the set $\{0, \ldots, i\}$ into the state $j$ (respectively, into the set $\{j, \ldots, k\}$).
It holds that $\QQ{i}{j}$ it satisfies the following recursive formula: %goes from having at most $i$ failed components and then jumping in one step to $j$ failed components is
\begin{align}\label{def:Q}
\QQ{i}{j} =
\begin{cases}
P_{0, j} & \text{for $i = 0$ and $j>i$ }  \\ %\binom{n}{j} \lambda^{(n)}_{j} / \Psi(n)
\QQ{i-1}{j} + \QQ{i-1}{i} \, P_{i,j} & \text{for $i \geq 1$ and $j>i$} \\ %\binom{n-i}{j-i} \lambda^{(n-i)}_{j-i} / \Psi(n-i)
0 & \text{otherwise,}
\end{cases}
\end{align}
and $\QQQ{i}{j}{k} = \sum_{l=j}^k \QQ{i}{l}$.

\item Denote by $\Q{i}{j}$ the probability that the chain $\mathbf \Nd$ goes from having zero to $i$ failed components in one or more steps, and then jumps, \emph{in one step}, from $i$ to $j$ failed components. It holds that $\Q{i}{j} = \sum_{l = 0}^i (P^l)_{0, i} P_{i, j}$ for $0 \leq i < j \leq n$ and 0 otherwise, where $(P^l)_{0, i}$ is the component $(0,i)$ of the matrix $P^l$, and the matrix $P = (P_{i,j} : i,j=0,\ldots,n)$ has the probabilities $P_{i,j}$ defined in~\eqref{TP}. Also, $\QQ{i}{j} = \sum_{l_1=0}^i \Q{l_1}{j}$ and $\QQQ{i}{j}{k} = \sum_{l_1=0}^i \sum_{l_2=j}^k \Q{l_1}{l_2}$.

\item For $1 \leq r, k \leq n$,
%\begin{align}\label{PTrnTkn}
%\P( T_{r:n} < T_{k:n} ) = \sum_{i=0}^{r-1} \sum_{j=r}^{k-1} \Q{i}{j}
%\qquad \text{and} \qquad
%\P( T_{r:n} = T_{k:n} ) = \sum_{i=0}^{r-1} \sum_{j=k}^n \Q{i}{j}.
%\end{align}
\begin{align}
\label{PTrnTknl}
\P( T_{r:n} < T_{k:n} ) & = 
\begin{cases}
\QQQ{r-1}{r}{k-1} & \text{for } r<k \\ %\sum_{j=r}^{k-1} \QQ{r-1}{j} & \text{for } r<k \\
0 & \text{for } r \geq k
\end{cases}
 \\
\label{PTrnTkne}
\P( T_{r:n} = T_{k:n} ) & = 
\begin{cases}
\QQQ{r-1}{k}{n} & \text{for } r<k \\ %\sum_{j=k}^n \QQ{r-1}{j} & \text{for } r<k \\
1 & \text{for } r = k
\end{cases} \\
\label{PTkn eq j}
\P\left( N(T_{r:n}) = k \right) & =
\begin{cases}
\QQ{r-1}{k} & \text{for } r \leq k \\ % \sum_{i=0}^{r-1} \Q{i}{k}
0 & \text{for } r > k.
\end{cases}
\end{align}
\end{enumerate}
\end{lemma}
%%%%%%%%%%%%%%%%%%%%%%%%%%%%%%%%%%%%%%%%%%%%%%%%%%%%

In the rest of the paper, our results will be mostly expressed using the probabilities $\P( T_{r:n} < T_{k:n} )$, $\P( T_{r:n} = T_{k:n} )$ and $\P\left( N(T_{r:n}) = k \right)$. Lemma~\ref{lemma1} shows that these reduce to the terms $\QQ{i}{j}$ and $\QQQ{i}{j}{k}$ (see Part 5.), and also gives two ways of computing the latter: a recursive formula in Part 3., and a power matrix formula in Part 4.
Also note that, from the computational perspective, the only information of the LFMO distribution needed to compute the probabilities in Lemma~\ref{lemma1} %$\QQ{i}{j}$ and $\QQQ{i}{j}{k}$ in~\eqref{PTrnTknl}, \eqref{PTrnTkne} and~\eqref{PTkn eq j}, $P_{i, j}$ in~\eqref{TP} %, and the rates $\lambda^{(n-i)}_{j-i}$ in~\eqref{def:LV}, 
are the values of $\Psi(1)$, \ldots, $\Psi(n)$.

We also note that  Parts 1.~and 2.~of Lemma~\ref{lemma1} were already observed in~\cite[Section 3.3.3]{matthias2017simulating} in giving an efficient simulation algorithm for the failure times of the components. The latter is, in essence, that if there are $i$ failed components, sample the next time of a failure with an exponential distribution of rate $\Psi(n-i)$; then sample the new number of failed components $j$ in $i+1$, \ldots, $n$ according to the probabilities $P_{i,j}$ of~\eqref{TP}; and then choose uniformly at random, between the $n-i$ alive components, the newly $j-i$ failed components.

%%%%%%%%%%%%%%%%%%%%%%%%%%%%%%%%%%%%
\subsection{Spoiler of simple repair policies}\label{sec:spoiler}
%%%%%%%%%%%%%%%%%%%%%%%%%%%%%%%%%%%%

To give a glimpse of the main result of this paper, Theorem~\ref{theo} in Section~\ref{sec:main}, we now give a result that shows the long-term cost of a simple repair policy for a $k$-out-of-$n$:F system; i.e., systems that fail when $k$ or more components have failed.
%Its proof is deferred to Section~\ref{sec:proofs}.
Its proof is placed in Section~\ref{sec:proofs}, and is direct from Theorem~\ref{theo}, by using the signature vector $\mathbf s = \mathbf e_k$, the $k$-th canonical vector, and equations~\eqref{PTrnTknl}, \eqref{PTrnTkne} and~\eqref{PTkn eq j}.

%%%%%%%%%%%%%%%%%%%%%%%%%%%%%%%%%%%%%%%%%%%%%%%
\begin{proposition}\label{corollary}
Consider a $k$-out-of-$n$:F system. % with $n>1$ components and $1 \leq k \leq n$, i.e., it fails when $k$ or more of its components fail.
Assume that its components' lifetimes follow an LFMO distribution. Further assume that we can repair (instantaneously) all the failed components, at cost $\ccmp(j)$ for repairing $j$ failed components, plus an additional cost $\csys$ if the system has also failed. Lastly, consider a repair policy in which all failed components are instantly repaired when the system fails, or when $r$ or more of its components have failed, for a given $1 \leq r \leq n$. % (note that the analysis is trivial when $k \leq r$, as there the system always fails before any repair). Then the following hold.

\begin{enumerate}
\item The probability $p\r$ that the first repair is due to a system failure is $p\r = \QQQ{r-1}{k}{n}$ if $k>r$ and $1$ if $k \leq r$.

\item Letting $C[0, t]$ be the cumulative (random) cost of operating the system up to time $t$, the long-term mean cost is %of the system operating under this repair policy is
\begin{align*}
\lim_{t \to \infty} \frac{C[0, t]}{t}
=
\frac{
p\r \csys + \sum_{j=\mini{k}{r}}^{n} \QQ{(\mini{k}{r})-1}{j} \, \ccmp(j)
}{
\E \left[ T_{\mini{k}{r}:n} \right]
}
\quad \text{a.s.,}
\end{align*}
where $\QQ{(\mini{k}{r})-1}{j}$ and $\E \left[ T_{\mini{k}{r}:n} \right]$ are computed as in~\eqref{def:Q} and~\eqref{eq:ETkn}, respectively, since $\mini{k}{r} = \min\{k,r\}$ is a deterministic value.

\item The mean times until the first repair and first system failure are, respectively, $\E \left[ T_{\mini{k}{r}:n} \right]$ and $\E \left[ T_{\mini{k}{r}:n} \right] / p\r$.

\item The rate at which repairs, system failures, and components' failures occur, are $1 / \E \left[ T_{\mini{k}{r}:n} \right]$, $p\r / \E \left[ T_{\mini{k}{r}:n} \right]$ and $\sum_{j=1}^{n} \QQ{(\mini{k}{r})-1}{j} \, j / \E \left[ T_{\mini{k}{r}:n} \right]$, respectively.
\end{enumerate}
\end{proposition}

%%%%%%%%%%%%%%%%%%%%%%%%%%%%%%%%%%%%%%%%%%%%%%%%%%%%

%%%%%%%%%%%%%%%%%%%%%%%%%%%%%%%%%%%%%%%%%%%%%%%%%%%%
\subsection{System structure}\label{sec:structure}
%%%%%%%%%%%%%%%%%%%%%%%%%%%%%%%%%%%%%%%%%%%%%%%%%%%%

We consider a \emph{binary system} with $n$ components, where each state of the system is represented with a vector $\mathbf x$ in $\{0,1\}^n$, where $x_i = 1$ if component $i$ is working and $0$ otherwise.
We assume that there is a deterministic so-called \emph{structure function} $\Phi : \{0,1\}^n \rightarrow \{0,1\}$, where for a state $\mathbf x$ in $\{0,1\}^n$ we have $\Phi(\mathbf x)=1$ iff the system is working in state $\mathbf x$.

%We model the structure of a reliability system with $n$ components by using the vectors $x$ in $\{0,1\}^n$ where $x_i = 1$ iff component $i$ is working and $0$ otherwise, and assume that there is a deterministic so-called \emph{structure function} $\Phi : \{0,1\}^n \rightarrow \{0,1\}$ where for a given system state $\mathbf x$ in $\{0,1\}^n$ we have $\Phi(\mathbf x)=1$ iff the system is working.

Throughout this work, we consider that the system structure $\Phi$ is \emph{semi-coherent}, defined as follows; for further details see~\cite{navarro2021introduction}.

%%%%%%%%%%%%%%%%%%%%%%%%%%%%%%%%%%%%%%%%%%%%%%%%%%%%
\begin{definition}[Monotone, semi-coherent and coherent systems]
Consider a system with $n>1$ components and structure function $\Phi : \{0,1\}^n \to \{0,1\}$.
\begin{enumerate}
    \item The system is \emph{monotone} iff for all states $\mathbf x$ and $\mathbf y$ in $\{0, 1\}^n$ such that $\mathbf x \leq \mathbf y$ we have $\Phi( \mathbf x) \leq \Phi( \mathbf y)$.
    \item The system is \emph{semi-coherent} iff it is monotone with $\Phi( \mathbf 0) = 0$ and $\Phi( \mathbf 1) = 1$.
    \item The component $i$ is an \emph{irrelevant component} iff for all states $\mathbf x \in \{0,1\}^n$, we have $\Phi( \mathbf x \rvert_{x_i=0}) = \Phi(\mathbf x \rvert_{x_i=1})$, where $ \mathbf x \rvert_{x_i=y}$ denotes the state $\mathbf x$ but with $x_i$ replaced by $y$.
    \item The system \emph{has no irrelevant components} iff none of its components is irrelevant.	
    \item The system is \emph{coherent} iff it is monotone and has no irrelevant components.
\end{enumerate}
\end{definition}
%%%%%%%%%%%%%%%%%%%%%%%%%%%%%%%%%%%%%%%%%%%%%%%%%%%%

Note that a semi-coherent system is just a non-trivial monotone system, in the sense that it excludes the two trivial cases of $\Phi( \mathbf x)=0$ for all $\mathbf x$ (an always-failed system), or $\Phi( \mathbf x)=1$ for all $\mathbf x$ (an always-working system). %Indeed, by monotonicity, this happens iff we have $\Phi( \mathbf 0) = 0$ and $\Phi( \mathbf 1) = 1$.

We now consider the \emph{signature} of a system, a key idea in the study of system reliability, see~\cite{samaniego2007system}.
The concept of signature was initially introduced by Samaniego in~\cite{samaniego1985closure} for coherent systems with continuously distributed iid~lifetimes of components, however here we give the more general definition of \emph{structural signature} given in~\cite[Definition~2.1]{navarro2021introduction}.

%%%%%%%%%%%%%%%%%%%%%%%%%%%%%%%%%%%%%%%%%%%%%%%%%%%%
\begin{definition}[Structural signature]\label{def:signature}
For a binary system with $n$ components and structure function $\Phi$, we define its \emph{structural signature} $\mathbf s$ in $\R^n$ as
\begin{align}\label{def:s}
s_k := \overline S_{k-1} - \overline S_{k}, \qquad k = 1, \ldots, n,
\end{align}
where, for $k=0, \, \ldots, \, n$, %$\Phi_{n-k} := \sum_{x \in \{0,1\}^n, \lvert x \rvert=n-k} \Phi(x) / \binom{n}{n-k} = \# \{x \in \{0,1\}^n : \lvert x \rvert=n-k, \ \Phi(x)=1 \} / \# \{x \in \{0,1\}^n : \lvert x \rvert=n-k \}$
\begin{align}\label{def:S}
\overline S_k := & \frac{1}{ \binom{n}{n-k} } \sum_{\mathbf x \in \{0,1\}^n, \lvert \mathbf x \rvert=n-k} \Phi(\mathbf x) %\frac{\sum_{x \in \{0,1\}^n, \lvert x \rvert=n-k} \Phi(x)}{\binom{n}{n-k}} \\
%= & \# \{x \in \{0,1\}^n : \lvert x \rvert=n-k, \ \Phi(x)=1 \} \ / \ \# \{x \in \{0,1\}^n : \lvert x \rvert=n-k \}
\end{align}
is the proportion of states, among the total number of states with exactly $k$ failed components, i.e.~$\# \{\mathbf x \in \{0,1\}^n : \lvert \mathbf x \rvert=n-k \} = \binom{n}{n-k}$, that continue working despite having these $k$ failed components, i.e.~$\# \{\mathbf x \in \{0,1\}^n : \lvert \mathbf x \rvert=n-k, \ \Phi(\mathbf x)=1 \} = \sum_{\mathbf x \in \{0,1\}^n, \lvert \mathbf  x \rvert=n-k} \Phi(\mathbf x)$.
\end{definition}
%%%%%%%%%%%%%%%%%%%%%%%%%%%%%%%%%%%%%%%%%%%%%%%%%%%%

The following result remarks that a signature vector $\mathbf s$ can also be seen as a proportion of sequences of failures of the $n$ components, with the caveat that components fail one-by-one, i.e., not simultaneously.
%Interestingly, it is known that the signature is useful even when there are simultaneous failures; see Section~\ref{sec:coherent} below.
The result was already known, see e.g.~\cite{navarro2010joint}, but for completeness we provide a proof in Section~\ref{sec:proofs}.

%%%%%%%%%%%%%%%%%%%%%%%%%%%%%%%%%%%%%%%%%%%%%%%%%%%%
\begin{proposition}\label{prop:signature}
Consider the setting of Definition~\ref{def:signature} and assume that the system is semi-coherent.
Then $s_k$ corresponds to the proportion of sequences of the $n$ components ---when they fail separately one by one--- where the system fails at exactly the $k$-th failure.
Similarly, $\overline S_{k}$ is the proportion of sequences of the $n$ components where the system continues working after the $k$-th failure, when the components fail separately one by one.
\end{proposition}
%%%%%%%%%%%%%%%%%%%%%%%%%%%%%%%%%%%%%%%%%%%%%%%%%%%%

%%%%%%%%%%%%%%%%%%%%%%%%%%%%%%%%%%%%%%%%%%%%%%%%%%%%
\begin{figure}[hbtp]
\centering

\begin{tikzpicture}[
    box/.style={rectangle, draw, minimum width=1.2cm, minimum height=2.5cm},
    highlight/.style={fill=red!60, opacity=0.7},
    component/.style={draw, circle, inner sep=1pt, minimum size=0.6cm},
    system/.style={minimum width=2.5cm, minimum height=2.5cm, align=center}
]

\node[system] (sys111) at (6.5, 1.5) {}; %(1.1, 4.5) {};
\node[component] (c1_1) at (sys111.north) [xshift=-0.5cm, yshift=-0.5cm] {1};
\node[component] (c1_2) at (sys111.north) [xshift= 0.5cm, yshift=-0.5cm] {2};
\node[component] (c1_3) at (sys111.south) [yshift= 0.5cm] {3};
\draw (c1_1.east) -- (c1_2.west);
\draw (sys111.east) -| +(-.3,0) |- (c1_2.east);
\draw (sys111.east) -| +(-.3,0) |- (c1_3.east);
\draw (sys111.west) -| +(.3,0) |- (c1_3.west);
\draw (sys111.west) -| +(.3,0) |- (c1_1.west);

% Left side text (rotated)
\node[rotate=90, anchor=center] at (-2.2, 1.6) {
    \footnotesize{
    \begin{tabular}{c}
    All $3!=6$ possible sequences \\
    of failure times without ties
    \end{tabular}
    }
    };

% Three main boxes
\node[box] (box1) at (-1,1.5) {};
\node[box] (box2) at (1,1.5) {};
\node[box] (box3) at (3,1.5) {};

% Content in first box
\node at (-1,2.5) {$T_1$};
\node at (-1,2.1) {$T_1$};
\node at (-1,1.7) {$T_2$};
\node at (-1,1.3) {$T_2$};
\node at (-1,0.9) {$T_3$};
\node at (-1,0.5) {$T_3$};

% Content in second box with highlights
\node at (1,2.5) {$T_2$};
\node[highlight] at (1,2.1) {$T_3$};
\node at (1,1.7) {$T_1$};
\node[highlight, inner sep=2pt] at (1,1.3) {$T_3$};
\node[highlight, inner sep=2pt] at (1,0.9) {$T_1$};
\node[highlight, inner sep=2pt] at (1,0.5) {$T_2$};

% Content in third box with highlights
\node[highlight, inner sep=2pt] at (3,2.5) {$T_3$};
\node at (3,2.1) {$T_2$};
\node[highlight, inner sep=2pt] at (3,1.7) {$T_3$};
\node at (3,1.3) {$T_1$};
\node at (3,0.9) {$T_2$};
\node at (3,0.5) {$T_1$};

% Less than symbols between entries
\foreach \y in {0.5, 0.9, 1.3, 1.7, 2.1, 2.5} {
    %\node at (-0.4,\y) {$<$};
    \node at (0,\y) {$<$};
    %\node at (0.4,\y) {$<$};
    \node at (2,\y) {$<$};
}

% Arrows below boxes
\draw[thick, ->] (-1, .25) -- (-1, -.3);
\draw[thick, ->] ( 1, .25) -- ( 1, -.3);
\draw[thick, ->] ( 3, .25) -- ( 3, -.3);

% Bottom formula
\node at (-2,-1) {$s = ($};
\node at (-1,-1) {$\frac{0}{6}$};
\node at (0,-1.1) {$,$};
\node at (1,-1) {$\frac{4}{6}$};
\node at (2,-1.1) {$,$};
\node at (3,-1) {$\frac{2}{6}$};
\node at (3.7,-1) {$)$};

\end{tikzpicture}

\caption{We show how to compute the signature vector $s=(s_1, \ s_2, \ s_3)$ for the system in the right that has $n=3$ components. According to Proposition~\ref{prop:signature}, $s_k$ is the proportion, between the total number of $n!=3!=6$ sequences of failures without ties, of sequences where the system first fails at the $k$-th failure. In the figure, the red background highlights, for each sequence, the first time at which the system fails. Hence, $s_2 = 4/6$ because the system first fails at the second failure in 4 of the 6 sequences.}
\label{fig:signature}
\end{figure}
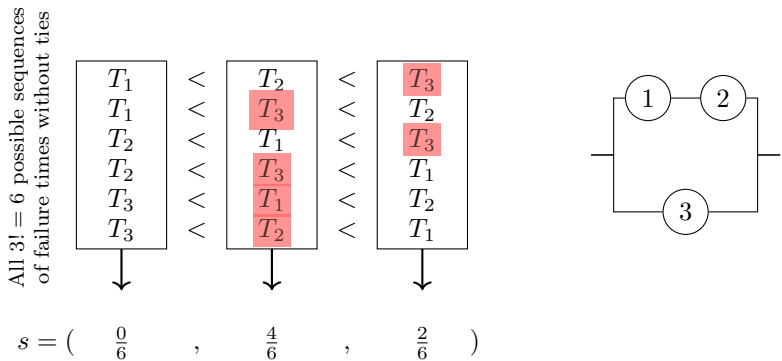
%%%%%%%%%%%%%%%%%%%%%%%%%%%%%%%%%%%%

We remark that Definition~\ref{def:signature} is made in terms of proportion \emph{of states}, while Proposition~\ref{prop:signature} of proportion \emph{of sequences}.
Computationally, it is more efficient to compute the signature vector $\mathbf s$ using the former than the latter, since Definition~\ref{def:signature} requires evaluating $\mathcal{O}(2^n)$ times the structure function $\Phi$, whereas Proposition~\ref{prop:signature}, $\mathcal{O}(n \cdot n!) \sim \mathcal{O}((n/e)^{n+3/2})$ times.
Nonetheless, Proposition~\ref{prop:signature} gives a better intuition of the values in the signature; see Figure~\ref{fig:signature} for a simple example. % we show how an example where the signature is computed as shown in the proposition, by listing all possible sequences of failures and registering the first time the system fails

%%%%%%%%%%%%%%%%%%%%%%%%%%%%%%%%%%%%
\subsection{Semi-coherent systems with LFMO components}\label{sec:coherent}
%%%%%%%%%%%%%%%%%%%%%%%%%%%%%%%%%%%%

We focus now on semi-coherent systems whose components' lifetimes follow a LFMO distribution.
Denote from now on $\Tsys$ as the time when the system fails.

%%%%%%%%%%%%%%%%%%%%%%%%%%%%%%%%%%%%%%%%%%%%%%%%%%%%

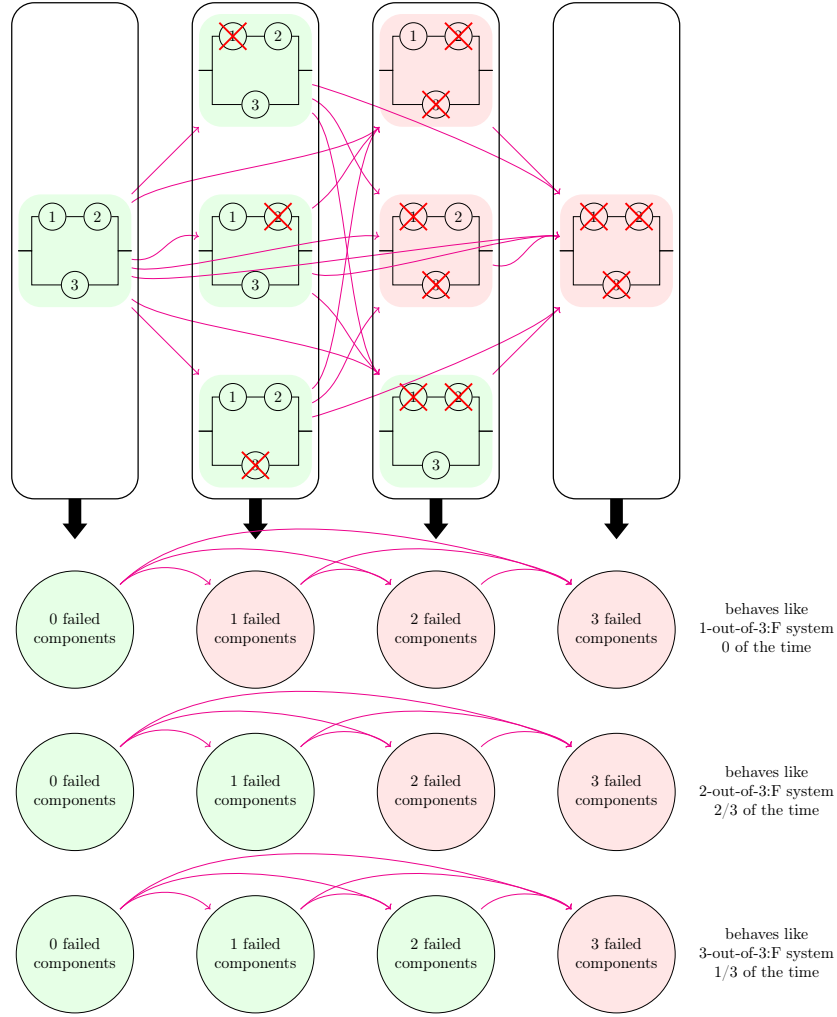
\begin{figure}%[h]%[hbtp]
\centering
\scalebox{0.6}{
\begin{tikzpicture}[
    component/.style={draw, circle, inner sep=1pt, minimum size=0.6cm},
    system/.style={fill=green!10, rounded corners=5mm, minimum width=2.5cm, minimum height=2.5cm, align=center},
    failed_system/.style={fill=red!10, rounded corners=5mm, minimum width=2.5cm, minimum height=2.5cm, align=center},
    cross out/.pic={
        \draw[red, very thick] (-0.3, -0.3) -- (0.3, 0.3);
        \draw[red, very thick] (-0.3, 0.3) -- (0.3, -0.3);
    }
]

% First row, first system (all working)
\node[system] (sys111) at (0, -4) {};
\node[component] (c1_1) at (sys111.north) [xshift=-0.5cm, yshift=-0.5cm] {1};
\node[component] (c1_2) at (sys111.north) [xshift= 0.5cm, yshift=-0.5cm] {2};
\node[component] (c1_3) at (sys111.south) [yshift= 0.5cm] {3};
\draw (c1_1.east) -- (c1_2.west);
\draw (sys111.east) -| +(-.3,0) |- (c1_2.east);
\draw (sys111.east) -| +(-.3,0) |- (c1_3.east);
\draw (sys111.west) -| +(.3,0) |- (c1_3.west);
\draw (sys111.west) -| +(.3,0) |- (c1_1.west);

% First row, second system (1 failed)
\node[system] (sys011) at (4, 0) {};
\node[component] (c2_1) at (sys011.north) [xshift=-0.5cm, yshift=-0.5cm] {1};
\node[component] (c2_2) at (sys011.north) [xshift= 0.5cm, yshift=-0.5cm] {2};
\node[component] (c2_3) at (sys011.south) [yshift= 0.5cm] {3};
\draw (c2_1.east) -- (c2_2.west);
\draw (sys011.east) -| +(-.3,0) |- (c2_2.east);
\draw (sys011.east) -| +(-.3,0) |- (c2_3.east);
\draw (sys011.west) -| +( .3,0) |- (c2_3.west);
\draw (sys011.west) -| +( .3,0) |- (c2_1.west);
\pic at (c2_1) {cross out};

\node[failed_system] (sys100) at (8, -0) {};
\node[component] (c9_1) at (sys100.north) [xshift=-0.5cm, yshift=-0.5cm] {1};
\node[component] (c9_2) at (sys100.north) [xshift= 0.5cm, yshift=-0.5cm] {2};
\node[component] (c9_3) at (sys100.south) [yshift= 0.5cm] {3};
\draw (c9_1.east) -- (c9_2.west);
\draw (sys100.east) -| +(-.3,0) |- (c9_2.east);
\draw (sys100.east) -| +(-.3,0) |- (c9_3.east);
\draw (sys100.west) -| +( .3,0) |- (c9_3.west);
\draw (sys100.west) -| +( .3,0) |- (c9_1.west);
\pic at (c9_2) {cross out};
\pic at (c9_3) {cross out};

% First row, fourth system (3 failed)
\node[failed_system] (sys000) at (12, -4) {};
\node[component] (c4_1) at (sys000.north) [xshift=-0.5cm, yshift=-0.5cm] {1};
\node[component] (c4_2) at (sys000.north) [xshift= 0.5cm, yshift=-0.5cm] {2};
\node[component] (c4_3) at (sys000.south) [yshift= 0.5cm] {3};
\draw (c4_1.east) -- (c4_2.west);
\draw (sys000.east) -| +(-.3,0) |- (c4_2.east);
\draw (sys000.east) -| +(-.3,0) |- (c4_3.east);
\draw (sys000.west) -| +( .3,0) |- (c4_3.west);
\draw (sys000.west) -| +( .3,0) |- (c4_1.west);
\pic at (c4_1) {cross out};
\pic at (c4_2) {cross out};
\pic at (c4_3) {cross out};

% Second row, first system (2 failed)
\node[system] (sys101) at (4, -4) {};
\node[component] (c5_1) at (sys101.north) [xshift=-0.5cm, yshift=-0.5cm] {1};
\node[component] (c5_2) at (sys101.north) [xshift= 0.5cm, yshift=-0.5cm] {2};
\node[component] (c5_3) at (sys101.south) [yshift= 0.5cm] {3};
\draw (c5_1.east) -- (c5_2.west);
\draw (sys101.east) -| +(-.3,0) |- (c5_2.east);
\draw (sys101.east) -| +(-.3,0) |- (c5_3.east);
\draw (sys101.west) -| +( .3,0) |- (c5_3.west);
\draw (sys101.west) -| +( .3,0) |- (c5_1.west);
\pic at (c5_2) {cross out};

% Second row, second system (2,3 failed)
\node[failed_system] (sys010) at (8, -4) {};
\node[component] (c6_1) at (sys010.north) [xshift=-0.5cm, yshift=-0.5cm] {1};
\node[component] (c6_2) at (sys010.north) [xshift= 0.5cm, yshift=-0.5cm] {2};
\node[component] (c6_3) at (sys010.south) [yshift= 0.5cm] {3};
\draw (c6_1.east) -- (c6_2.west);
\draw (sys010.east) -| +(-.3,0) |- (c6_2.east);
\draw (sys010.east) -| +(-.3,0) |- (c6_3.east);
\draw (sys010.west) -| +( .3,0) |- (c6_3.west);
\draw (sys010.west) -| +( .3,0) |- (c6_1.west);
\pic at (c6_1) {cross out};
\pic at (c6_3) {cross out};

% Second row, third system (1,3 failed)
\node[system] (sys001) at (8, -8) {};
\node[component] (c7_1) at (sys001.north) [xshift=-0.5cm, yshift=-0.5cm] {1};
\node[component] (c7_2) at (sys001.north) [xshift= 0.5cm, yshift=-0.5cm] {2};
\node[component] (c7_3) at (sys001.south) [yshift= 0.5cm] {3};
\draw (c7_1.east) -- (c7_2.west);
\draw (sys001.east) -| +(-.3,0) |- (c7_2.east);
\draw (sys001.east) -| +(-.3,0) |- (c7_3.east);
\draw (sys001.west) -| +( .3,0) |- (c7_3.west);
\draw (sys001.west) -| +( .3,0) |- (c7_1.west);
\pic at (c7_1) {cross out};
\pic at (c7_2) {cross out};

% Third row, first system (3 failed)
\node[system] (sys110) at (4, -8) {};
\node[component] (c8_1) at (sys110.north) [xshift=-0.5cm, yshift=-0.5cm] {1};
\node[component] (c8_2) at (sys110.north) [xshift= 0.5cm, yshift=-0.5cm] {2};
\node[component] (c8_3) at (sys110.south) [yshift= 0.5cm] {3};
\draw (c8_1.east) -- (c8_2.west);
\draw (sys110.east) -| +(-.3,0) |- (c8_2.east);
\draw (sys110.east) -| +(-.3,0) |- (c8_3.east);
\draw (sys110.west) -| +( .3,0) |- (c8_3.west);
\draw (sys110.west) -| +( .3,0) |- (c8_1.west);
\pic at (c8_3) {cross out};

%\draw [magenta, thick, rounded corners=1cm] (2,2) -- (10,2) -- (13.9,-2) -- (13.9,-6) -- (10, -9.5) -- (2, -9.5) -- cycle;
%\draw [magenta, thick, rounded corners=1cm] (6.3, 1.8) -- (10, 1.8) -- (13.7,-2) -- (13.7,-6) -- (10,-9.4) -- (6.3, -9.4) -- cycle;
%\draw [magenta, thick, rounded corners=1cm] (6.5, 1.6) -- (10, 1.6) -- (13.5,-2) -- (13.5,-5.5) -- (6.5, -5.5) -- cycle;

\draw[->, magenta] (sys111) .. controls +(2,2) and +(-2,-2) .. (sys011);
\draw[->, magenta] (sys111) .. controls +(2,-.3) and +(-2,.5) .. (sys101);
\draw[->, magenta] (sys111) .. controls +(2,-2) and +(-2,2) .. (sys110);
\draw[->, magenta] (sys111) .. controls +(2,1.7) and +(-2,-2) .. (sys100);
\draw[->, magenta] (sys111) .. controls +(2,-.6) and +(-2,.5) .. (sys010);
\draw[->, magenta] (sys111) .. controls +(2,-1.7) and +(-2,2) .. (sys001);
\draw[->, magenta] (sys111) .. controls +(2,-.9) and +(-2,.5) .. (sys000);

\draw[->, magenta] (sys011) .. controls +(2,-0.5) and +(-2,2) .. (sys000);
\draw[->, magenta] (sys011) .. controls +(2,-1.0) and +(-2,2) .. (sys010);
\draw[->, magenta] (sys011) .. controls +(2,-1.5) and +(-2,2) .. (sys001);

\draw[->, magenta] (sys110) .. controls +(2,1.5) and +(-2,-2) .. (sys100);
\draw[->, magenta] (sys110) .. controls +(2,1.0) and +(-2,-2) .. (sys010);
\draw[->, magenta] (sys110) .. controls +(2,0.5) and +(-2,-2) .. (sys000);

\draw[->, magenta] (sys101) .. controls +(2,1.5) and +(-2,-2) .. (sys100);
\draw[->, magenta] (sys101) .. controls +(2,-1.5) and +(-2,2) .. (sys001);
\draw[->, magenta] (sys101) .. controls +(2,-0.8) and +(-2,0.5) .. (sys000);

\draw[->, magenta] (sys100) .. controls +(2,-2) and +(-2,2) .. (sys000);
\draw[->, magenta] (sys010) .. controls +(2,-.5) and +(-2,.5) .. (sys000);
\draw[->, magenta] (sys001) .. controls +(2,2) and +(-2,-2) .. (sys000);

%%%%%%%%%%%%%%%%

\draw [thick, rounded corners=.5cm] (   -1.4, 1.5) -- (   -1.4, -9.5) -- (   1.4, -9.5) -- (   1.4, 1.5) -- cycle;
\draw [thick, rounded corners=.5cm] ( -1.4+4, 1.5) -- ( -1.4+4, -9.5) -- ( 1.4+4, -9.5) -- ( 1.4+4, 1.5) -- cycle;
\draw [thick, rounded corners=.5cm] ( -1.4+8, 1.5) -- ( -1.4+8, -9.5) -- ( 1.4+8, -9.5) -- ( 1.4+8, 1.5) -- cycle;
\draw [thick, rounded corners=.5cm] (-1.4+12, 1.5) -- (-1.4+12, -9.5) -- (1.4+12, -9.5) -- (1.4+12, 1.5) -- cycle;

%\draw[line width=4pt, -{Triangle[length=2pt, width=10pt]}] (0,0) -- (0,-4);
\draw[line width=8pt, -{Triangle[width=15pt,length=10pt]}] (0,-9.5) -- (0,-9.5-.9);
\draw[line width=8pt, -{Triangle[width=15pt,length=10pt]}] (4,-9.5) -- (4,-9.5-.9);
\draw[line width=8pt, -{Triangle[width=15pt,length=10pt]}] (8,-9.5) -- (8,-9.5-.9);
\draw[line width=8pt, -{Triangle[width=15pt,length=10pt]}] (12,-9.5) -- (12,-9.5-.9);

%%%%%%%%%%%%%%%%

\draw[fill=green!10] (0,-12-0.4) circle (1.3cm) node {\begin{tabular}{c} 0 failed \\ components \end{tabular}} ;
\draw[fill=red!10] (4,-12-0.4) circle (1.3cm) node {\begin{tabular}{c} 1 failed \\ components \end{tabular}} ;
\draw[fill=red!10] (8,-12-0.4) circle (1.3cm) node {\begin{tabular}{c} 2 failed \\ components \end{tabular}} ;
\draw[fill=red!10]   (12,-12-0.4) circle (1.3cm) node {\begin{tabular}{c} 3 failed \\ components \end{tabular}} ;
\draw[->, magenta]   (0+1,-12+1-0.4) .. controls +(.5,.5) and +(-.5,.5) .. (4-1,-12+1-0.4);
\draw[->, magenta]   (0+1,-12+1-0.4) .. controls +(1.5,1.5) and +(-.5,.5) .. (8-1,-12+1-0.4);
\draw[->, magenta]   (0+1,-12+1-0.4) .. controls +(2.5,2.5) and +(-.5,.5) .. (12-1,-12+1-0.4);

\draw[->, magenta] (4+1,-12+1-0.4) .. controls +(.5,.5) and +(-.4,.4) .. (8-1,-12+1-0.4);
\draw[->, magenta] (4+1,-12+1-0.4) .. controls +(1.5,1.5) and +(-.5,.5) .. (12-1,-12+1-0.4);

\draw[->, magenta] (8+1,-12+1-0.4) .. controls +(.5,.5) and +(-.4,.4) .. (12-1,-12+1-0.4);

\node[anchor=west] at (12+1.5,-12-0.4) {\begin{tabular}{c} behaves like \\ 1-out-of-3:F system \\ 0 of the time \end{tabular}} ;

%%%%%%%%%%%%%%%%

\draw[fill=green!10] (0,-12-4) circle (1.3cm) node {\begin{tabular}{c} 0 failed \\ components \end{tabular}} ;
\draw[fill=green!10] (4,-12-4) circle (1.3cm) node {\begin{tabular}{c} 1 failed \\ components \end{tabular}} ;
\draw[fill=red!10] (8,-12-4) circle (1.3cm) node {\begin{tabular}{c} 2 failed \\ components \end{tabular}} ;
\draw[fill=red!10]   (12,-12-4) circle (1.3cm) node {\begin{tabular}{c} 3 failed \\ components \end{tabular}} ;
\draw[->, magenta]   (0+1,-12+1-4) .. controls +(.5,.5) and +(-.5,.5) .. (4-1,-12+1-4);
\draw[->, magenta]   (0+1,-12+1-4) .. controls +(1.5,1.5) and +(-.5,.5) .. (8-1,-12+1-4);
\draw[->, magenta]   (0+1,-12+1-4) .. controls +(2.5,2.5) and +(-.5,.5) .. (12-1,-12+1-4);

\draw[->, magenta] (4+1,-12+1-4) .. controls +(.5,.5) and +(-.4,.4) .. (8-1,-12+1-4);
\draw[->, magenta] (4+1,-12+1-4) .. controls +(1.5,1.5) and +(-.5,.5) .. (12-1,-12+1-4);

\draw[->, magenta] (8+1,-12+1-4) .. controls +(.5,.5) and +(-.4,.4) .. (12-1,-12+1-4);

\node[anchor=west] at (12+1.5,-12-4) {\begin{tabular}{c} behaves like \\ 2-out-of-3:F system \\ 2/3 of the time \end{tabular}} ;

%%%%%%%%%%%%%%%%

\draw[fill=green!10] (0,-12-7.6) circle (1.3cm) node {\begin{tabular}{c} 0 failed \\ components \end{tabular}} ;
\draw[fill=green!10] (4,-12-7.6) circle (1.3cm) node {\begin{tabular}{c} 1 failed \\ components \end{tabular}} ;
\draw[fill=green!10] (8,-12-7.6) circle (1.3cm) node {\begin{tabular}{c} 2 failed \\ components \end{tabular}} ;
\draw[fill=red!10]   (12,-12-7.6) circle (1.3cm) node {\begin{tabular}{c} 3 failed \\ components \end{tabular}} ;
\draw[->, magenta]   (0+1,-12+1-7.6) .. controls +(.5,.5) and +(-.5,.5) .. (4-1,-12+1-7.6);
\draw[->, magenta]   (0+1,-12+1-7.6) .. controls +(1.5,1.5) and +(-.5,.5) .. (8-1,-12+1-7.6);
\draw[->, magenta]   (0+1,-12+1-7.6) .. controls +(2.5,2.5) and +(-.5,.5) .. (12-1,-12+1-7.6);

\draw[->, magenta] (4+1,-12+1-7.6) .. controls +(.5,.5) and +(-.4,.4) .. (8-1,-12+1-7.6);
\draw[->, magenta] (4+1,-12+1-7.6) .. controls +(1.5,1.5) and +(-.5,.5) .. (12-1,-12+1-7.6);

\draw[->, magenta] (8+1,-12+1-7.6) .. controls +(.5,.5) and +(-.4,.4) .. (12-1,-12+1-7.6);

\node[anchor=west] at (12+1.5,-12-7.6) {\begin{tabular}{c} behaves like \\ 3-out-of-3:F system \\ 1/3 of the time \end{tabular}} ;

%%%%%%%%%%%%%%%%

\end{tikzpicture}
}%
\caption{
    The classical Samaniego decomposition result $\P(\Tsys > t) = \sum_{k=1}^n s_k \, \P( T_{k:n} > t)$ in~\eqref{eq:samaniego} establishes that, probabilistically, the system failure time $\Tsys$ behaves, a proportion $s_k$ of the times, as the $k$-th failure time $T_{k:n}$, which is actually the system failure time of a $k$-out-of-$n$:F system. In the upper-half of the figure, we show a system with $n=3$ components whose signature vector $s=(0, 2/3, 1/3)$ was computed in Figure~\ref{fig:signature}. The Samaniego decomposition result allows us to group the states of the system according to the number of failed components (in the lower-half of the figure), and analyze the $n=3$ $k$-out-of-$n$:F systems (each with only $n+1=4$ states), instead of analyzing the original system with $2^n=8$ states, $3^n-2^n=19$ transitions, and $2^{2^n-1}=128$ repair policies (see the caption of Figure~\ref{fig:basic}). This is a significant reduction in the complexity of the analysis.
}
\label{fig:samaniego}

\end{figure}
%%%%%%%%%%%%%%%%%%%%%%%%%%%%%%%%%%%%

We are inspired by the following classical signature decomposition result by Samaniego
\begin{align}\label{eq:samaniego}
\P(\Tsys > t) = \sum_{k=1}^n s_k \, \P( T_{k:n} > t)
\end{align}
for all $t \geq 0$; see Figure~\ref{fig:samaniego} for an intuitive interpretation of this result.
It was first shown in~\cite{samaniego1985closure} for coherent systems with continuously distributed iid~lifetimes of components, however, it holds for general coherent systems with exchangeable distribution of the components' failure times $(T_1,\dots,T_n)$; see \cite{navarro2008application} or \cite[p.~37]{navarro2021introduction}.
In our particular case, of semi-coherent systems with LFMO lifetimes of components, the decomposition~\eqref{eq:samaniego} is a corollary of Lemma~\ref{lemma:samaniego} below.
Furthermore, in this case, we have that
\begin{align}\label{MTTFa}
\P(\Tsys > t) = \sum_{i=1}^n a_i e^{-\Psi(i) t}
\qquad \text{and} \qquad
\E [ \Tsys ] = \sum_{i=1}^n \frac {a_i}{\Psi(i)},
\end{align}
where the values $a_i := \binom{n}{i} \sum_{k=n-i+1}^n \binom{i-1}{n-k}(-1)^{(i-1)-(n-k)} s_k$ for $i=1,\ldots,n$ correspond to the so-called \emph{minimal} signature of the system, consisting of integer values (not necessarily positive) that sum one; see~\cite{navarro2007properties}.
This comes from the fact that $T_{1:i}$ has an exponential distribution with mean $1/\Psi(i)$. We remark that the
minimal signature $\mathbf a = (a_1,\dots,a_n)$ is closely related to two classical concepts. From the network-reliability side, it coincides with the so-called \textit{domination polynomial} $p(x)=a_1x+\dots+a_nx^n$ of the system; this connection between dominations and signatures
has been explored in detail by Boland, Samaniego and Vestrup \cite{Boland} and~\cite[pp.~61 and 62]{navarro2021introduction}. From the
combinatorial side, it is related with the M\"{o}bius transform of an associated set function on the
lattice of subsets of components, in the sense of Grabisch \cite{Grabisch}, see \cite[pp.~13 and 61]{navarro2021introduction}. Both viewpoints
provide an alternative algebraic justification for why the same vector $\mathbf{a}$  governs
so many different quantities in our analysis (reliability function, mean time to
failure, expected costs).
%This comes from plugging into~\eqref{eq:samaniego} the expression for $\P( T_{k:n}>t)$ right before~\eqref{eq:ETkn}. %and~\cite[Equation~(8)]{lagos2024limiting}.

We claim that, in a way, the main results of this paper, Proposition~\ref{proposition1} and Theorem~\ref{theo}, can be seen as extensions of decomposition~\eqref{eq:samaniego} for the repair policies we propose in Section~\ref{sec:main}.
Indeed, in Section~\ref{sec:signature} we show that a similar decomposition to~\eqref{eq:samaniego} can be made for: the probability $\P( \Trep\r = \Tsys\r )$ in~\eqref{eq:samaniego 1} that the first repair is due to a system failure;
the complementary probability $\P ( \Trep\r < \Tsys\r )$ in~\eqref{eq:samaniego 2};
the probability $\P ( N(\Trep\r) = j )$ in~\eqref{eq:samaniego 3} of the number of failed components at the time of the first repair;
for the expected time until the first repair $\E \, \Trep\r$ in~\eqref{eq:samaniego 5};
and for the expected cumulative cost until the first repair time $\E \, C[0,\Trep\r]$ in~\eqref{eq:samaniego 4}.

The following result shows the distribution of the number of failed components when the system fails.
%Its proof uses Lemma~\ref{lemma:samaniego} below and is deferred to Section~\ref{sec:proofs}.
%%%%%%%%%%%%%%%%%%%%%%%%%%%%%%%%%%%%%%%%%%%%%%%%%%%%
\begin{proposition}\label{proposition1}
Consider a semi-coherent system with $n>1$ components and structural signature vector $\mathbf s$ in $\R^n$, and assume that the components' lifetimes follow a LFMO distribution with Laplace exponent function $\Psi$. %(x) = - \log E[\exp(-x L_1)]$ for all $x$ in $\R$.
Consider the number $\Nc(t)$ of failed components at time $t$ and assume that $N(0)=0$.
It holds that for any $j=1, \ldots, n$, 
\begin{align}\label{def:w}
\P(N(\Tsys) = j) = \sum_{k=1}^j s_k \, \P \left( N\left( T_{k:n} \right) = j \right) = \sum_{k=1}^j s_k \, \QQ{k-1}{j}.
\end{align}
%where the values $\P \left( N\left( T_{k:n} \right) = j \right)$ are computed as shown in~\eqref{PTkn eq j}.
\end{proposition}
%%%%%%%%%%%%%%%%%%%%%%%%%%%%%%%%%%%%%%%%%%%%%%%%%%%%

The proof is direct from Lemma~\ref{lemma:samaniego} below. Indeed, by the duality relation $\{N(T_{r:n}) = k\} = \{T_{r:n} = T_{k:n} < T_{k+1:n}\}$ recalled before Lemma~\ref{lemma1} ---with the convention $T_{n+1:n} := +\infty$--- and by~\eqref{eq:TsysTK}, the event $\{N(\Tsys) = j\}$ equals $\{T_{K:n} = T_{j:n} < T_{j+1:n}\}$, which is of the form $\{g(K, \mathbf T_{(\cdot)}) = 1\}$ for $g(k, \mathbf T_{(\cdot)}) := \I{T_{k:n} = T_{j:n} < T_{j+1:n}}$. Applying~\eqref{eq:samaniego master} to this $g$ we obtain
\begin{align*}
\P(N(\Tsys) = j) & = \P( \Tsys = T_{j:n} < T_{j+1:n}) = \sum_{k=1}^n s_k \, \P( T_{k:n} = T_{j:n} < T_{j+1:n}) \\
& = \sum_{k=1}^n s_k \, \P(N(T_{k:n}) = j) = \sum_{k=1}^j s_k \, \P(N(T_{k:n}) = j) = \sum_{k=1}^j s_k \, \QQ{k-1}{j},
\end{align*}
where the fourth equality holds since $N(T_{k:n}) \geq k$ implies $\P(N(T_{k:n}) = j) = 0$ for $k>j$, and the last one is~\eqref{PTkn eq j}.

The previous result motivates the following definition.
%%%%%%%%%%%%%%%%%%%%%%%%%%%%%%%%%%%%%%%%%%%%%%%
\begin{definition}\label{def:proc sig}
For a system with $n>1$ components, we define the \emph{process signature} of the system as the vector $\mathbf q = (q_1, \ldots, q_n)$  of values $q_j := \P(N(\Tsys) = j)$.
\end{definition}
%%%%%%%%%%%%%%%%%%%%%%%%%%%%%%%%%%%%%%%%%%%%%%%

%We can define the vector $\mathbf q = (q_1, \ldots, q_n)$ in $\R^n$ of values $q_j := \P(N(\Tsys) = j)$ as the \emph{process signature} of the system.
We are motivated by the fact that the process signature coincides with the structural signature in~\eqref{def:s}, i.e., $s_j=q_j$, if the random vector $(T_1,\dots,T_n)$ has an EXC \emph{but} absolutely continuous distribution. Indeed, in the latter case there are no simultaneous failures, so $\P(N(\Tsys) = j) = \P(\Tsys = T_{j:n}) = s_j$, where the last equality comes from~\cite{navarro2007reliability}.
However, Proposition~\ref{proposition1} above shows that, for a semi-coherent system with LFMO components' lifetimes,
%---which has EXC but not necessarily absolutely continuous distribution, due to the possibility of having simultaneous failures---,
the value $s_j$ is not enough and has to be corrected into $q_j = \sum_{k=1}^j s_k \, \QQ{k-1}{j}$.

We pause to clarify the conceptual jump that Proposition~\ref{proposition1} effects. In the absolutely-continuous EXC setting (so in particular there are no simultaneous failures), the structural signature $s_j$ admits two equivalent readings: (a)~the proportion of failure-sequences in which the system fails at the $j$-th failure, and (b)~the probability that the system fails precisely when exactly $j$ components have failed, i.e., $\P(N(T_{fail}) = j) = s_j$. The two coincide because, under absolute continuity, the $n$ components fail strictly one by one almost surely, so a failure-sequence is the same object as a record of values of $N$ at successive failure epochs.

Under the LFMO distribution, where simultaneous failures can occur with positive
probability, readings (a) and (b) decouple: (a) remains valid as a combinatorial
proportion, but (b) no longer holds, because at the system failure
epoch several components may fail at once and so $N(T_{fail})$  can exceed the
index of the critical failure. Proposition 2.3 repairs (b) by inserting the shock mass
weights $F_{[0,k-1],j}$, yielding the system process signature $q_j = \sum_{k=1}^j s_k \, \QQ{k-1}{j}$.

We also emphasize that in this paper the structural signature plays a broader
role than is traditional: rather than serving only as a building block for the
system reliability function $\P(T_{fail} > t)$, it enters our explicit formulas for the
number of failed components, the repair times, and the long-term mean costs.

%The vector $\mathbf q = (q_1, \ldots, q_n)$ in $\R^n$ of values $q_j = \P(N(\Tsys) = j)$ defined in~\eqref{def:w} can be considered the signature of a semi-coherent system with LFMO components.
%Indeed, by~\cite{navarro2007reliability} it is known that if the random vector $(T_1,\dots,T_n)$ has an EXC \emph{but} absolutely continuous distribution, it holds that $q_j = s_j$.
%However, Proposition~\ref{proposition1} above shows that for the LFMO distribution ---which has EXC but not necessarily absolutely continuous distribution, due to the possibility of having simultaneous failures---, the value $s_j$ is not enough and has to be corrected into $q_j = \sum_{k=1}^j s_k \, \QQ{k-1}{j}$.

Lastly, we remark that to prove Proposition~\ref{proposition1} and most of our results we need the following result.
%It will help in extending the Samaniego decomposition result $\P(\Tsys > t) = \sum_{k=1}^n s_k \, \P(T_{k:n} > t))$ in~\eqref{eq:samaniego} into, e.g., $\P(N(\Tsys) = j) = \sum_{k=1}^n s_k \, \P \left( N\left( T_{k:n} \right) = j \right)$, since the order statistics $T_{k:n}$ and the number of failed components $N(T_{k:n})$ are ``dually'' related, per the discussion before Lemma~\ref{lemma1} regarding~\eqref{PTrnTknl}, \eqref{PTrnTkne} and~\eqref{PTkn eq j}.

%%%%%%%%%%%%%%%%%%%%%%%%%%%%%%%%%%%%%%%%%%%%%%%%%%%%

\begin{lemma}\label{lemma:samaniego}
Consider a semi-coherent system with $n>1$ components and structural signature vector $\mathbf s$ in $\R^n$, and assume that the components' lifetimes follow a LFMO distribution. Then, enlarging the probability space if necessary, there exists a random index $K$ with values in $\{1,\ldots,n\}$ such that:
\begin{enumerate}
\item[(i)] $\P(K=k) = s_k$ for all $k=1,\ldots,n$;
\item[(ii)] $K$ is independent of the vector of ordered failure times $\mathbf T_{(\cdot)} := (T_{1:n}, \ldots, T_{n:n})$;
\item[(iii)] $\Tsys = T_{K:n}$ almost surely.
\end{enumerate}
Consequently, for every measurable $g : \{1,\ldots,n\} \times \R^n \to \R$ such that the expectations below are well defined,
\begin{align}\label{eq:samaniego master}
\E \left[ g\left( K, \mathbf T_{(\cdot)} \right) \right] = \sum_{k=1}^n s_k \, \E \left[ g\left( k, \mathbf T_{(\cdot)} \right) \right];
\end{align}
in particular, taking $g(k,\mathbf T_{(\cdot)}) = \I{T_{k:n}>t, \ T_{1:n}>t_1, \ \ldots, \ T_{n:n}>t_n}$, for all $t, t_1, \ldots, t_n$,
\begin{align}\label{eq:samaniego joint}
& \P(\Tsys > t, \ T_{1:n} > t_1, \ \ldots, \ T_{n:n}>t_n) \nonumber \\
& \qquad\qquad = \sum_{k=1}^n s_k \, \P(T_{k:n} > t, \ T_{1:n} > t_1, \ \ldots, \ T_{n:n}>t_n).
\end{align}
\end{lemma}

%%%%%%%%%%%%%%%%%%%%%%%%%%%%%%%%%%%%%%%%%%%%%%%%%%%%

We prove this extension for the LFMO distribution in Section~\ref{sec:proofs} and we conjecture that it should hold for any exchangeable distribution.

We stress that~\eqref{eq:samaniego master} is what we shall use throughout: the events of interest below ---such as $\{\Tsys = T_{j:n} < T_{j+1:n}\}$ or $\{\min\{T_{r:n},\Tsys\} = \Tsys\}$--- are not of the orthant form appearing in~\eqref{eq:samaniego joint}, since they compare $\Tsys$ with the \emph{random} order statistics rather than with deterministic thresholds; they are, however, all of the form $g(K, \mathbf T_{(\cdot)})$.

%%%%%%%%%%%%%%%%%%%%%%%%%%%%%%%%%%%%%%%%%%%%%%%%%%%%
%%%%%%%%%%%%%%%%%%%%%%%%%%%%%%%%%%%%%%%%%%%%%%%%%%%%
\section{Cost of simple repair policies}\label{sec:main} %$r$-out-of-$n$:R
%%%%%%%%%%%%%%%%%%%%%%%%%%%%%%%%%%%%%%%%%%%%%%%%%%%%
%%%%%%%%%%%%%%%%%%%%%%%%%%%%%%%%%%%%%%%%%%%%%%%%%%%%

%%%%%%%%%%%%%%%%%%%%%%%%%%%%%%%%%%%%
%\subsection{Cost and repair structure}
%%%%%%%%%%%%%%%%%%%%%%%%%%%%%%%%%%%%

In this section we give our main result for the long-term mean cost of simple repair policies. For that, we consider the following cost and repair structure.

\paragraph{Assumptions} 
\begin{enumerate}
\item The repair of $j$ failed components costs $\ccmp(j)$.
\item The repair of a failed system costs $\csys$ plus the cost of repairing all failed components.
\item Any failure of a component or the system is detected instantaneously.
\item Both types of repairs, i.e., of the system and of components, are performed instantly.
\item At time $t=0$ the system starts with all its components working.
\end{enumerate}

In other words, for example, the \emph{preventive} repair of a system that is still working but has two failed components, costs $\ccmp(2)$; however, the \emph{corrective} repair of a system that failed and has three failed components, costs $\csys + \ccmp(3)$.
In this way, the system repair cost $\csys$ models the additional loss in productivity of the system.

We will consider the following simple repair policies that do not discriminate on which components ---either critical or not, in whichever sense--- have failed, and only use the number of failed components.
%%%%%%%%%%%%%%%%%%%%%%%%%%%%%%%%%%%%%%%%%%%%%%%
\begin{definition}\label{def:repair}
For a system with $n>1$ components, and for any $r$ in $\{1, \ldots, n\}$, we define the \emph{$r$-out-of-$n$ repair policy} (shortly written as $r$-out-of-$n$:R policy) as the one where all failed components, and the system itself if failed, is repaired in either of the following cases:
\begin{itemize}
    \item when $r$ or more components fail;
    \item when the system fails.
\end{itemize}
In a system operating under an $r$-out-of-$n$:R policy, denote by $\Tsys\r$ and $\Trep\r$ the first time of system failure and repair, respectively.
\end{definition}
%%%%%%%%%%%%%%%%%%%%%%%%%%%%%%%%%%%%%%%%%%%%%%%

The $r$-out-of-$n$:R family has two distinctive features. First, under any repair
action, all currently-failed components are repaired at once, irrespective of their
structural role. This renewal property is what allows us to obtain closed-form
expressions, and it lets us treat the chain after a repair as stochastically identical
to the original chain (see the proof of Theorem~\ref{theo} in Section~\ref{sec:proofs}). Second, the decision to repair upon arrival of a shock depends only on the total number of failed components, not on which specific ones; this collapses the policy space from
doubly-exponential to linear in $n$.

We acknowledge that these policies might be sub-optimal with respect to
the long-term mean cost when components play heterogeneous structural roles. The second feature, in particular, is restrictive, since a shock that takes down
only ``peripheral'' components is treated identically to a shock that takes down ``critical'' ones at the same count. The compensating gain, however, is that
we obtain a full closed-form analysis of the performance indicators, which (to
the best of our knowledge) is not available for any richer policy class in the
simultaneous-failure setting. The $r$-out-of-$n$:R family should therefore be read
as a tractable, fully analyzable benchmark, against which the additional value
of structure-aware policies can be measured in future work. More general repair policies can be analyzed in future projects.

We also remark that, by the memoryless property of the LFMO distribution and the strong Markov property of the L\'evy subordinator (Section~\ref{sec:lifetimes}), no information is gained by examining the system at any time strictly between two consecutive shock arrivals: in that interval, no component changes its status, and the conditional distribution of the future given the current state is the same as the unconditional distribution of a new system in that state. As a consequence, the natural decision epochs of \emph{any} state-dependent repair policy in this model are precisely the shock-arrival times (together with $t=0$). The $r$-out-of-$n$:R family inherits this property, and the doubly-exponential count of possible state-to-action assignments in Figure~\ref{fig:basic} should therefore be read as the count of distinct state-to-action mappings the modeller must specify, not as a count of decision epochs.

%%%%
Note that two events can happen at the first repair time $\Trep\r$: either the system failed, in which case it holds that $\Trep\r = \Tsys\r$; or the system could have continued working with the failed components, in which case $\Trep\r < \Tsys\r$.
Hence ---recalling that $N(t)$ is the number of failed components at time $t$---, at the repair time $\Trep\r$ the cost incurred is $\ccmp(N(\Trep\r))$, %where $N(\Trep\r)$ is the number of failed components at time $\Trep\r$,
plus $\csys$ if the system also failed, i.e., $\ccmp(N(\Trep\r)) + \csys \I{\Trep\r = \Tsys\r}$.

The following is the main result of this paper. Its proof is deferred to Section~\ref{sec:proofs}. There, we also formally introduce the dynamic objects used in the proof --- the component-status process $\mathbf \chi(t) \in \{0,1\}^n$, the system-status process $\chi\sys(t) := \Phi(\mathbf \chi(t))$, and the sequences of repair and system-failure times $\Trep\r[m]$ and $\Tsys\r[m]$ as stopping times with respect to the natural filtration of $(\mathbf \chi(t), \chi\sys(t))$, which underpin the renewal argument that we invoke repeatedly through the strong Markov property. The reader uninterested in the technical detail may proceed without consulting Section~\ref{sec:proofs}.

%%%%%%%%%%%%%%%%%%%%%%%%%%%%%%%%%%%%%%%%%%%%%%%
\begin{theorem}\label{theo}
Consider a semi-coherent system with $n>1$ components and signature $\mathbf s$ in $\R^n$, and assume that its components' lifetimes follow a LFMO distribution with Laplace exponent function $\Psi$. %(x) = - \log E[\exp(-x L_1)]$ for all $x$ in $\R$, where $\mathbf L$ is the subordinator process in~\eqref{def:LFMO eq}.
Assume that the system operates under an $r$-out-of-$n$:R policy, for a given $r$ in $\{1,\ldots, n\}$.
Let $\Tsys\r$ and $\Trep\r$ be as in Definition~\ref{def:repair}.
Denote by $N(t)$ the number of failed components at time $t>0$, and  $C[0,t]$ the total cost incurred by operating the system during the time interval $[0, t]$.
Then the following results hold:

\begin{enumerate}

\item Denoting by $p\r := \P\left( \Trep\r = \Tsys\r \right)$ the probability that the first repair is due to a system failure, it holds that
\begin{align}
p\r = \sum_{k = 1}^n s_k \P \left( T_{\mini{k}{r}:n} = T_{k:n} \right) & = \sum_{k = 1}^r s_k + \sum_{k = r+1}^n s_k \P \left( T_{r:n} = T_{k:n} \right), \label{def:pF} \\ % = \sum_{k = 1}^n s_k \sum_{i = 0}^{\min\{k,r\}-1} \sum_{j = k}^{n} \Q{i}{j} 
1-p\r = \sum_{k = 1}^n s_k \P\left( T_{\mini{k}{r}:n} < T_{k:n} \right) & = \sum_{k = r+1}^n s_k \P\left( T_{r:n} < T_{k:n} \right). \label{eq:1-pF}
\end{align} % $ = \sum_{k = r+1}^n s_k \sum_{i = 0}^{r-1} \sum_{j = r}^{k-1} Q{i,j}$.
where $\P \left( T_{r:n} = T_{k:n} \right)$ and $\P \left( T_{r:n} < T_{k:n} \right)$ are computed as shown in~\eqref{PTrnTkne} and~\eqref{PTrnTknl}, respectively.

\item The probability that at the time of repair there are exactly $j$ failed components, for $1 \leq j \leq n$, is
%\begin{align}
%w\r_j := \P\left( N(\Trep\r) = j \right) =
%\begin{cases}
%\sum_{k = 1}^j s_k \sum_{i = 0}^{\min\{k,r\}-1} Q{i,j}, & \text{if } j=1, \ldots, r-1 \\
%\sum_{k = 1}^n s_k \sum_{i = 0}^{\min\{k,r\}-1} Q{i,j}, & \text{if } j=r, \ldots, n.
%\end{cases}
%\label{def:wr}
%\end{align}
\begin{align}
\P\left( N(\Trep\r) = j \right) & = \sum_{k = 1}^{n} s_k \, \P \left( N\left( T_{\mini{r}{k}:n} \right) = j \right) \nonumber \\
& = \sum_{k = 1}^{j \I{j<r} + n \I{j \geq r}} s_k \, \P \left( N\left( T_{\mini{r}{k}:n} \right) = j \right),
%\sum_{i = 0}^{\min\{k,r\}-1} \Q{i}{j}
\label{def:wr}
\end{align}
where the probability $\P \left( N\left( T_{\mini{r}{k}:n} \right) = j \right)$ is computed as shown in~\eqref{PTkn eq j}. %, since since $\mini{r}{k} = \min\{r,k\}$ is a deterministic value.
%, and, recall, $Q{i,j}=0$ for $j \leq i$, by~\eqref{def:Q}.

\item When~$p\r>0$, for $j =1, \ \ldots, \ n$ it holds that %The probabilities $\P( N(\Trep\r) = j |\, \Trep\r = \Tsys\r )$, $\P( N(\Trep\r) = j |\, \Trep\r < \Tsys\r )$ and $\P( N(\Tsys\r) = j )$ satisfy
\begin{align}
%\P\left( N(\Tsys\r) = j \right) & = \sum_{k = 1}^j s_k \, \P \left( N\left( T_{\min\{r,k\}:n} \right) = j \right) \qquad \text{for } j=1,\ldots,n \label{eq:PNTfail} \\ %\sum_{i=0}^{\min\{k, r\}-1} \Q{i}{j}
\P\left( N(\Trep\r) = j \left|\, \Trep\r = \Tsys\r \right.\right) & = \P\left( N(\Tsys\r) = j \right) \label{eq:PNTrep eq 1} \\
& = \frac{1}{p\r} \sum_{k = 1}^j s_k \, \P \left( N\left( T_{\mini{r}{k}:n} \right) = j \right); \label{eq:PNTrep eq 2} %\sum_{i=0}^{\min\{k, r\}-1} \Q{i}{j}
%\P\left( N(\Trep\r) = j \left|\, \Trep\r < \Tsys\r \right.\right) & = \begin{cases}
%\displaystyle \frac{1}{1-p\r}  \sum_{k = j+1}^n s_k \, \P \left( N\left( T_{r:n} \right) = j \right) & \text{for } j \geq r \\ %\sum_{i = 0}^{r-1} \Q{i}{j}
%0 & \text{for } j< r.
%\end{cases} \label{eq:PNTrep ineq}
\end{align}
and for $j=r, \, \ldots,  \, n$,
\begin{align}
\P\left( N(\Trep\r) = j \left|\, \Trep\r < \Tsys\r \right.\right) & = \frac{1}{1-p\r}  \sum_{k = j+1}^n s_k \, \P \left( N\left( T_{r:n} \right) = j \right) \label{eq:PNTrep ineq}
\end{align}
where $\sum_{k=n+1}^n$ is zero, per usual convention; while $\P( N(\Trep\r) = j |\, \Trep\r < \Tsys\r )=0$ for $j<r$.

%Otherwise, when~$p\r=0$, almost surely the system never fails, i.e., $\Tsys\r=+\infty$ almost surely, and~\eqref{eq:PNTrep ineq} reduces to $\P\left( N(\Trep\r) = j \right) = \sum_{k = j+1}^n s_k \, \P \left( N\left( T_{r:n} \right) = j \right)$ for $j \geq r$, and $=0$ for $j<r$.

\item When~$p\r>0$ the system will eventually fail, i.e., $\Tsys\r < +\infty$ almost surely, and we have that the long-term mean cost of the system operating under this repair policy is
\begin{align}\label{eq:LTMC}
\lim_{t \to \infty} \frac{C[0, t]}{t} = \frac{\E \, C[0,\Tsys\r]}{\E \, \Tsys\r} = \frac{\E \, C[0,\Trep\r]}{\E \, \Trep\r} \qquad \text{a.s.,}
%\overline{C}^{(r)}}{\overline{T}_{sys}^{(r)}}
\end{align}
where the mean costs $\E \, C[0,\Tsys\r]$ and $\E \, C[0,\Trep\r]$, and  mean times $\E \, \Tsys\r$ and $\E \, \Trep\r$, satisfy
\begin{align}
\E \, C[0,\Tsys\r] & = \frac{1}{p\r} \E \, C[0,\Trep\r] = \csys + \frac{1}{p\r} \sum_{j=1}^{n} \ccmp(j) \, \P\left( N(\Trep\r) = j \right) \label{eq:ECost} \\
\E \, \Tsys\r & = \frac{1}{p\r} \E \, \Trep\r = \frac{1}{p\r} \sum_{k=1}^n s_k \, \E \left[ T_{\mini{k}{r}:n} \right] \label{eq:ETfail}
\end{align}
where~$p\r$, $\P( N(\Trep\r) = j )$ and $\E [ T_{\mini{k}{r}:n} ]$ are computed in~\eqref{def:pF}, \eqref{def:wr} and~\eqref{eq:ETkn}, respectively.

\item When~$p\r=0$ the system never fails, i.e., $\Tsys\r = +\infty$ almost surely, %so $\E \, \Tsys\r = \E \, C[0,\Tsys\r] = +\infty$, and~\eqref{eq:LTMC}, \eqref{eq:ECost} and~\eqref{eq:ETfail} hold, in the sense that
and it holds that $\lim_{t \to \infty} C[0, t] / t = \E \, C[0,\Trep\r] / \E \, \Trep\r$, where $\E \, C[0,\Trep\r] = \sum_{j=1}^{n} \ccmp(j) \, \P\left( N(\Trep\r) = j \right)$ and also~$\E \, \Trep\r = \sum_{k=1}^n s_k \, \E \left[ T_{\mini{k}{r}:n} \right] $.

\end{enumerate}
\end{theorem}
%%%%%%%%%%%%%%%%%%%%%%%%%%%%%%%%%%%%%%%%%%%%%%%

We remark that, as a direct corollary, with these expressions we can compute the rate of occurrence (i.e., events per time unit) of system failures and repairs, as $1 / \E \, \Tsys\r$ and $1 / \E \, \Trep\r$, respectively.
We can also compute the long-term mean number of failures (LTMN), i.e., the rate $\lim_{t \to \infty} N[0,t]/t$ where $N[0,t]$ is the total number of failures in $[0, t]$, as $\sum_{j=1}^{n} j \, \P( N(\Trep\r) = j ) / \E \, \Trep\r$. This is obtained using~\eqref{eq:LTMC} with $\ccmp(i)=i$ and $\csys=0$.

We also remark that the previous result holds for any \emph{mixed} coherent system with $n$ components; that is, a system whose structure function $\Phi$ is actually a randomized choice between a finite set of deterministic coherent structure functions with given known probabilities, see~\cite[Chapter 3]{samaniego2007system}. In this case the result holds for the signature $\mathbf s$ of the mixed coherent system.
Indeed, the proof of Theorem~\ref{theo} in Section~\ref{sec:proofs} essentially relies on the decomposition~\eqref{eq:samaniego}, extended in Lemma~\ref{lemma:samaniego}, which also holds for the latter type of systems.

We also note that it is not unusual that $p\r=0$: in the basic example of Figure~\ref{fig:basic} (Right), in the $r$-out-of-$n$:R policy with $r=1$ and iid exponential failure times of components (i.e., when $\Psi(x)=\mu x$ for some $\mu>0$), the system \emph{never} fails.

Importantly, from a numerical perspective, the only values needed to compute the quantities in Theorem~\ref{theo}, are $\Psi(1)$, \ldots, $\Psi(n)$, which parameterize the LFMO distribution of failure times, and the structural signature vector $\mathbf s = (s_1, \ldots, s_n)$, which parameterizes the structure of the system; see Lemma~\ref{lemma1} and Proposition~\ref{prop:signature} and the commentaries following them.

Finally, we note that~\eqref{eq:PNTrep eq 1} gives an expression for the process signature when operating under an $r$-out-of-$n$:R policy, say $\mathbf q\r$: $$q\r_j = \P(N(\Tsys\r) = j) = \frac{1}{p\r}\sum_{k=1}^j s_k \, \P \left( N\left( T_{\mini{r}{k}:n} \right) = j \right);$$ whereas, when there were no repairs, by~\eqref{def:w} the process signature was $$q_j = \P(N(\Tsys) = j) = \sum_{k=1}^j s_k \, \P \left( N\left( T_{k:n} \right) = j \right).$$

%%%%%%%%%%%%%%%%%%%%%%%%%%%%%%%%%%%%%%%%%%%%%%%%%%%%
%%%%%%%%%%%%%%%%%%%%%%%%%%%%%%%%%%%%%%%%%%%%%%%%%%%%
%%%%%%%%%%%%%%%%%%%%%%%%%%%%%%%%%%%%%%%%%%%%%%%%%%%%
\subsection{Corollary of  iid~components.} A direct corollary of Theorem~\ref{theo} is the following case of iid~exponentially distributed failure times of components. We defer the proof of this result to Section~\ref{sec:proofs}.
%%%%%%%%%%%%%%%%%%%%%%%%%%%%%%%%%%%%%%%%%%%%%%%%%%%%
%%%%%%%%%%%%%%%%%%%%%%%%%%%%%%%%%%%%%%%%%%%%%%%%%%%%
%%%%%%%%%%%%%%%%%%%%%%%%%%%%%%%%%%%%%%%%%%%%%%%%%%%%

%%%%%%%%%%%%%%%%%%%%%%%%%%%%%%%%%%%%%%%%%%%%%%%%%%%%
%%%%%%%%%%%%%%%%%%%%%%%%%%%%%%%%%%%%%%%%%%%%%%%%%%%%
%%%%%%%%%%%%%%%%%%%%%%%%%%%%%%%%%%%%%%%%%%%%%%%%%%%%
\begin{corollary}\label{corol:iid}
Consider a semi-coherent system with $n>1$ components and signature $\mathbf s$ in $\R^n$, 
%Consider a system with $n>1$ components and assume it has a semi-coherent structure with signature $\mathbf s$ in $\R^n$
and assume that its components' lifetimes are iid~exponentially distributed with rate $\mu>0$.
Assume that the system operates under an $r$-out-of-$n$:R policy, for a given $r$ in $\{1,\ldots, n\}$.
Then the following results hold:

\begin{enumerate}
\item The probability that the first repair is due to a system failure is $p\r = \sum_{k=1}^r s_k$.

\item The expected time and cost, respectively, until the first repair, when starting with all components working, is
\begin{align*}
%\E \, C[0,\Tsys\r] & = \frac{1}{p\r} \left[ \sum_{k=1}^{r} s_k(\csys+\ccmp(k)) + \left( \sum_{k=r+1}^{n} s_k \right) \ccmp(r) \right] \\
\E \, C[0,\Trep\r] & = \quad\ \ \sum_{k=1}^{r} s_k(\csys+\ccmp(k)) + \ \left( \sum_{k=r+1}^{n} s_k \right) \ccmp(r) \\
%\E \, \Tsys\r & = \frac{1}{\mu p\r} \left[ \sum_{k=1}^r s_k \left( \sum_{i=n-k+1}^n \frac{1}{i} \right) \quad + \quad \left( \sum_{k=r+1}^n s_k \right) \left( \sum_{i=n-r+1}^n \frac{1}{i} \right) \right].
\E \, \Trep\r & = \frac{1}{\mu} \left[ \sum_{k=1}^r s_k \left( \sum_{i=n-k+1}^n \frac{1}{i} \right) \quad + \ \left( \sum_{k=r+1}^n s_k \right) \left( \sum_{i=n-r+1}^n \frac{1}{i} \right) \right].
\end{align*}
Also, $\E \, C[0,\Tsys\r]  = \E \, C[0,\Trep\r] /p\r$ and $\E \, \Tsys\r = \E \, \Trep\r /p\r$ when $p\r>0$.
%\E \, C[0,\Trep\r] = \csys \sum_{k=1}^{r} s_k + \sum_{k=1}^{r-1} \ccmp(k) s_k + \ccmp(r)\sum_{k=r}^n s_k

\item The long-term mean cost of the system operating with an $r$-out-of-$n$:R policy is
\begin{align*}
\lim_{t \to \infty} \frac{ C[0, t] }{ t } = \mu \, \frac{ \sum_{k=1}^{r} s_k \left( \csys+\ccmp(k) \right) \ + \ \left( \sum_{k=r+1}^{n} s_k \right) \ccmp(r) }{ \sum_{k=1}^r s_k \left( \sum_{i=n-k+1}^n \frac{1}{i} \right) \ + \ \left( \sum_{k=r+1}^n s_k \right) \left( \sum_{i=n-r+1}^n \frac{1}{i} \right) }
\end{align*}
\end{enumerate}
\end{corollary}
%%%%%%%%%%%%%%%%%%%%%%%%%%%%%%%%%%%%%%%%%%%%%%%%%%%%
%%%%%%%%%%%%%%%%%%%%%%%%%%%%%%%%%%%%%%%%%%%%%%%%%%%%
%%%%%%%%%%%%%%%%%%%%%%%%%%%%%%%%%%%%%%%%%%%%%%%%%%%%

%%%%%%%%%%%%%%%%%%%%%%%%%%%%%%%%%%%%
\subsection{Signature decomposition interpretation}\label{sec:signature}
%%%%%%%%%%%%%%%%%%%%%%%%%%%%%%%%%%%%

We now argue that Theorem~\ref{theo} state several extensions of the Samaniego signature decomposition in~\eqref{eq:samaniego}.
Indeed, the latter can be written as $$\P^\mathbf{s} \left( \Tsys >t \right) = \sum_{k = 1}^n s_k \, \P^{k:n} \left( \Tsys >t \right),$$ %where $\P^\mathbf{s}$ and $\P^{k:n}$ are the probabilities when considering the events, respectively, for a semi-coherent system with signature vector $\mathbf s$ and a $k$-out-of-$n$:F system.
where $\P^\mathbf{s}$ is the probability when the system is a semi-coherent one with signature vector $\mathbf s$, and $\P^{k:n}$ when the system is a $k$-out-of-$n$:F.
This holds because in the latter case, the system failure time is $\Tsys = T_{k:n}$; see Figure~\ref{fig:samaniego} for a simple example.

%We also remark that~\eqref{eq:samaniego} can be interpreted as the following probabilistic characterization of the system failure time $\Tsys$. Define $K$ as a discrete random variable that takes the values $\{1,\dots,n\}$ with probabilities $s_1,\dots,s_n$, respectively\footnote{$s_k \geq 0$ for all $k$ and $s_1 + \dots + s_n = 1$ since the system is semi-coherent}. It holds that $\Tsys$ is equal in distribution to the so-called \emph{mixed} system $T_{K:n}$ defined as $T_{i:n}$ with probability $s_i$ for $i=1,\dots,n$ --- see~\cite[Section 3]{samaniego2007system} for further information on mixed coherent systems.

%Indeed, first note that the latter can be written as $\P^\mathbf{s} \left( \Tsys >t \right) = \sum_{k = 1}^n s_k \, \P^{k:n} \left( \Tsys >t \right)$, where $\P^\mathbf{s}$ and $\P^{k:n}$ are the probabilities when considering the events, respectively, for a coherent system with signature vector $\mathbf s$ and a $k$-out-of-$n$ system, since for the latter system its system failure time is $\Tsys = T_{k:n}$. It follows that for a $k$-out-of-$n$ system with an $r$-out-of-$n$:R policy, the first repair time is $\Trep\r = \min\{ T_{k:n}, \, T_{r:n}\} = T_{\mini{k}{r}:n}$.

Similarly, noting that for a $k$-out-of-$n$:F system with an $r$-out-of-$n$:R policy its first repair time is $\Trep\r = \min\{ T_{k:n}, \, T_{r:n}\} = T_{\mini{k}{r}:n}$, it follows that Theorem~\ref{theo} extends as follows the Samaniego decomposition~\eqref{eq:samaniego} in terms of its signature vector $\mathbf s$ and of $k$-out-of-$n$ systems:
\begin{align}
\P^\mathbf{s} \left( \Trep\r = \Tsys\r \right) & \displaystyle = \sum_{k = 1}^n s_k \, \P^{k:n} \left( \Trep\r = \Tsys\r \right) = \sum_{k = 1}^n s_k \, \P \left( T_{\mini{r}{k}:n} = T_{k:n} \right) \label{eq:samaniego 1} \\
\P^\mathbf{s} \left( \Trep\r < \Tsys\r \right) & \displaystyle = \sum_{k = 1}^n s_k \, \P^{k:n} \left( \Trep\r < \Tsys\r \right) = \sum_{k = 1}^n s_k \, \P \left( T_{\mini{r}{k}:n} < T_{k:n} \right) \label{eq:samaniego 2} \\
\P^\mathbf{s} \left( N(\Trep\r) = j \right) & = \sum_{k = 1}^n s_k \, \P^{k:n} \left( N(\Trep\r) = j \right) = \sum_{k = 1}^n s_k \, \P \left( N(T_{\mini{r}{k}:n}) = j \right) \label{eq:samaniego 3} \\
\E^\mathbf{s} \, \Trep\r & = \sum_{k=1}^n s_k \, \E^{k:n} \, \Trep\r  = \sum_{k=1}^n s_k \, \E \, T_{\mini{r}{k}:n}  \label{eq:samaniego 5} \\
\E^\mathbf{s} \, C[0,\Trep\r] & = \sum_{k = 1}^n s_k \, \E^{k:n} \, C[0,\Trep\r] . \label{eq:samaniego 4}
\end{align}

However, when $\P^\mathbf{s} ( \Trep\r = \Tsys\r ) > 0$, Theorem~\ref{theo} also state the following decompositions that require correcting the weights $s_k$ into $s_k \, \P^{k:n} ( \Trep\r = \Tsys\r ) / \P^\mathbf{s} ( \Trep\r = \Tsys\r )$ as follows %, i.e.,  $s_k \, \P^{k:n} ( \Trep\r = \Tsys\r ) / \sum_{l=1}^n s_l \, \P^{l:n} ( \Trep\r = \Tsys\r )$, as follows:
\begin{align*}
\P^\mathbf{s} \left( N(\Tsys\r) = j \right) & = \sum_{k = 1}^n \frac{s_k \, \P^{k:n} \left( \Trep\r = \Tsys\r \right)}{\P^\mathbf{s} \left( \Trep\r = \Tsys\r \right)} \, \P^{k:n} \left( N(\Tsys\r) = j \right) \\
\E^\mathbf{s} \, \Tsys\r & = \sum_{k=1}^n \frac{s_k \, \P^{k:n} \left( \Trep\r = \Tsys\r \right)}{\P^\mathbf{s} \left( \Trep\r = \Tsys\r \right)} \, \E^{k:n} \, \Tsys\r \\
\E^\mathbf{s} \, C[0,\Tsys\r] & = \sum_{k = 1}^n \frac{s_k \, \P^{k:n} \left( \Trep\r = \Tsys\r \right)}{\P^\mathbf{s} \left( \Trep\r = \Tsys\r \right)} \, \E^{k:n} \, C[0,\Tsys\r].
\end{align*}

%\begin{align*}
%\frac{s_k \, \P^{k:n} \left( \Trep\r = \Tsys\r \right)}{\P^\mathbf{s} \left( \Trep\r = \Tsys\r \right)} = \frac{s_k \, \P^{k:n} \left( \Trep\r = \Tsys\r \right)}{\sum_{j=1}^n s_j \, \P^{j:n} \left( \Trep\r = \Tsys\r \right)} = \frac{s_k \, \P \left( T_{\min\{r,k\}:n} = T_{k:n} \right)}{\sum_{j=1}^n s_j \, \P \left( T_{\min\{r,j\}:n} = T_{j:n} \right)}
%\end{align*}

%\begin{align}
%\overline{C}^{(r)} = \frac{\pR}{1 - \pR} \sum_{i=r}^{n-1} w_i^R c_i + \sum_{i=1}^n w_i^F c_i + c_F \\
%= \sum_{j=r}^{n-1} (\frac{\pR}{1 - \pR} w_j^R + w_j^F) c_j + \sum_{i=1}^{r-1} w_i^F c_i + w_n^F c_n + c_F \\
%= \sum_{j=r}^{n-1} (\sum_{k = j+1}^n s_k \sum_{i = 0}^{r-1} \Q{i}{j} + \sum_{k = 1}^j s_k \sum_{i=0}^{\min\{k, r\}-1} \Q{i}{j}) c_j / \pF + \sum_{i=1}^{r-1} w_i^F c_i + w_n^F c_n + c_F \\
%\end{align}

%%%%%%%%%%%%%%%%%%%%%%%%%%%%%%%%%%%%%%%%%%%%%%%%%%%%
\section{Computational experiments}\label{sec:exps}
%%%%%%%%%%%%%%%%%%%%%%%%%%%%%%%%%%%%%%%%%%%%%%%%%%%%
 
In this section we show computational experiments to illustrate our results.
For that, we first simulate the performance of our proposed $r$-out-of-$n$:R policies and compare it with the theoretical values we derive. We do this for the small system with $n=3$ components in Figure~\ref{fig:basic}.
Then, we present the values that can be computed for a medium-sized system with $n=26$ components.

The parameters of the L\'evy-frailty Marshall-Olkin distribution are chosen in the following way.
We take as subordinator process $\mathbf L = (L(t): t \geq 0)$ a compound Poisson process (CPP) with rate $\lambda=1/5$, constant drift $\mu = 0.9$, and with jumps having an exponential distribution with rate $\gamma=1$, of the type in Figure~\ref{fig:LFMO ex} (Left).
That is, $L_t = \mu t + \sum_{i=0}^{N(t)} J_i$, where $\mathbf N$ is a PP($\lambda$) and $J_i \sim expo(\gamma)$ iid.
We restrict to these processes because, as argued in~\cite[ch.~XVII S.~2]{feller1971introduction}, any L\'evy subordinator can be approximated as close as desired by compound Poisson processes.
As argued in Section~\ref{sec:lifetimes}, in this case the Laplace exponent function $\Psi$ of $\mathbf L$ takes the form $\Psi(x) = \mu x + \lambda x / (\gamma + x)$, so in particular the expected failure time of each component is $1/\Psi(1)=1$. In this way, these values of $\lambda$, $\mu$ and $\gamma$ are chosen to normalize the time units, i.e., one time unit represents the mean time to failure of a single component. Also, a heuristic interpretation of the model is that ``nominally'' each component fails on average after $1/\mu=1.111\ldots$ time units; however, shocks of degradation affecting all components happen on average every $1/\lambda=5$ time units, and each shock kills a working component independently with probability $1-\lambda/(\lambda+1)=0.833\ldots \approx 83\%$.

We also consider that the cost of repairing $i$ failed components is $\ccmp (i)=i$, however, repairing a failed system costs additionally $\csys = 10 \times n$. That is, the disruption cost of a system is an order of magnitude higher than the value of all of its components.

%%%%%%%%%%%%%%%%%%%%%%%%%%%%%%%%%%%%%%%%%%%%%%%%%%%%
\subsection{Convergence for a small sized system}\label{sec:exps-small}
%%%%%%%%%%%%%%%%%%%%%%%%%%%%%%%%%%%%%%%%%%%%%%%%%%%%

In Table~\ref{tab:r_simple} we show the values computed with the formulas derived in Theorem~\ref{theo} for the system with $n=3$ components in Figure~\ref{fig:LFMO ex}.
Recall that $p\r := \P( \Trep\r = \Tsys\r )$ is the probability that the first repair is due to a system failure. We denote by $N[0,t]$ and $C[0,t]$ the total number of failures and cost, respectively, in the time window $[0,t]$. Also, LTMN and LTMC correspond, respectively, to the long-term mean number of failures $\lim_{t \to \infty} N[0,t]/t$ and long-term mean cost $\lim_{t \to \infty} C[0,t]/t$.
The total number of failures $N[0,t]$ is computed using the expressions for $C[0,t]$ with $\ccmp(i)=i$ and $\csys=0$.

%%%%%%%%%%%%%%%%%%%%%%%%%%%%%%%%%%%%%%%%%%%%%%%%%%%%
\begin{table}[h]
 \caption{Performance indicators for the system with $n=3$ components in Figure~\ref{fig:LFMO ex}, computed with the results in Theorem~\ref{theo}.}
    \label{tab:r_simple}
    \centering
    \begin{tabular}{|c|c|c||c|c||c|c|}
        \hline
        $r$ & $p\r$ & $\E \, \Tsys\r$ & $\E \, N[0, \, \Tsys\r]$ & $\E \, C[0,\Tsys\r]$ & LTMN & LTMC \\ \cline{3-5}
            &       & $\E \, \Trep\r$ & $\E \, N[0, \, \Trep\r]$ & $\E \, C[0,\Trep\r]$ &      &       \\
        \hline \hline
        1 & 0.0292 & 12.0   & 36.0   & 66.0    & 3.0     & 5.5     \\ 
          &        & 0.3509 & 1.0526 & 1.9298  &         &         \\ \hline
        2 & 0.6836 & 1.2434 & 3.0    & 33.0    & 2.4128  & 26.5409 \\
          &        & 0.85   & 2.0508 & 22.559  &         &         \\ \hline
        3 & 1.0    & 1.1664 & 2.3672 & 32.3672 & 2.0296  & 27.7505 \\
          &        & 1.1664 & 2.3672 & 32.3672 &         &         \\
        \hline
    \end{tabular}
\end{table}
%%%%%%%%%%%%%%%%%%%%%%%%%%%%%%%%%%%%%%%%%%%%%%%%%%%%

We observe first that, as $r$ grows, we repair less frequently, since $\E \, \Trep\r$ increases.
In turn, this decreases the mean time to system failure $\E \, \Tsys\r$ and the system breaks down more frequently.
%Also, as we repair less frequently, the costs associated with system failure, $\E \, N[0, \, \Tsys\r]$ and $\E \, C[0,\Tsys\r]$, decrease.
%Inversely, as $r$ grows, the quantities associated to components' repairs $\E \, \Trep\r$, $\E \, N[0, \, \Trep\r]$ and $\E \, C[0,\Trep\r]$ grow.
Note, however, that the repair policy for $r=3$ corresponds to only repairing when the system fails ---see Figure~\ref{fig:basic}--- so we obtain $p\r=1$ and $\Trep\r=\Tsys\r$ almost surely.
Notably, despite the LTMN (rate of components' failures) being the highest for $r=1$---because, as we repair more, there are more components to break with the simultaneous failures---, this policy attains the lowest cost rate, LTMC.

%%%%%%%%%%%%%%%%%%%%%%%%%%%%%%%%%%%%%%%%%%%%%%%%%%%%
\begin{figure}[h]
{
\centering
     \begin{subfigure}[b]{.45\textwidth}%{0.327\textwidth}
         \centering
         \includegraphics[width=\textwidth]{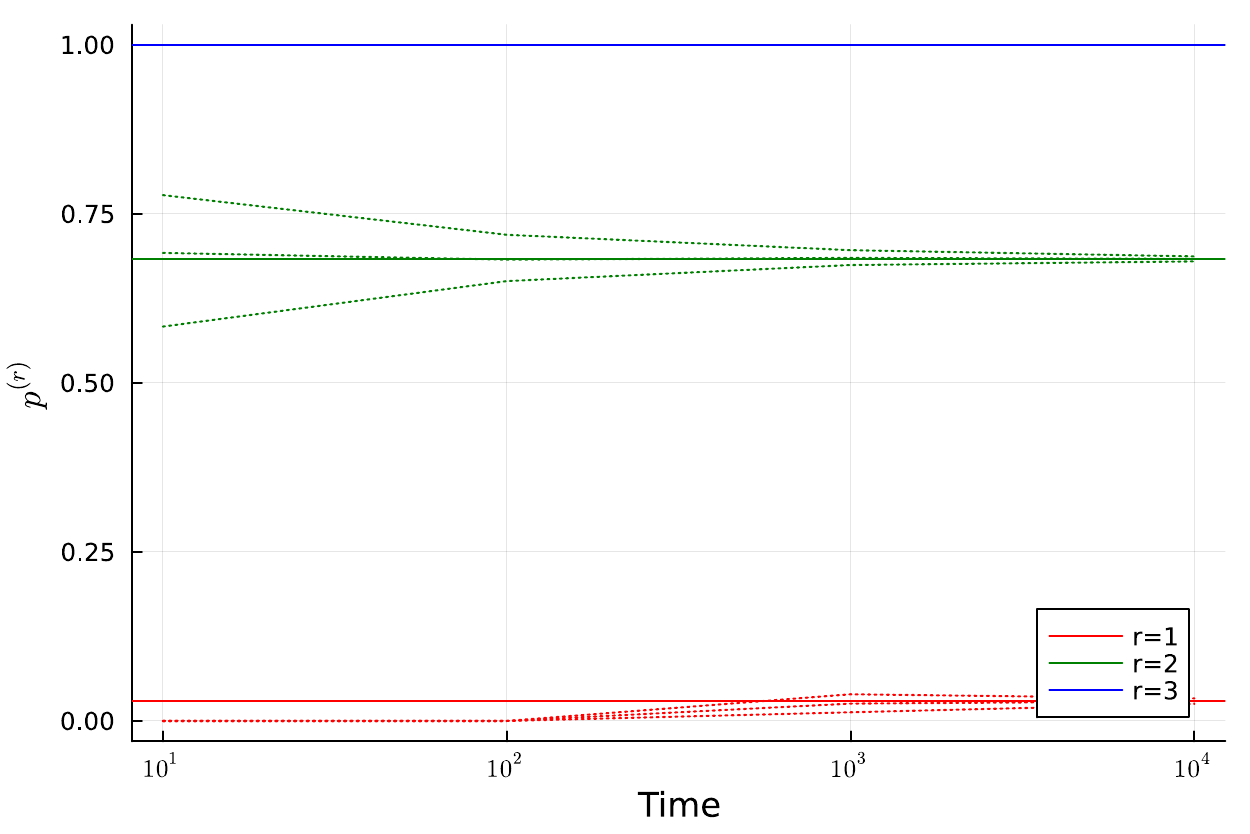}
     \end{subfigure}
%     \hfill
~
     \begin{subfigure}[b]{.45\textwidth}%{0.327\textwidth}
         \centering
         \includegraphics[width=\textwidth]{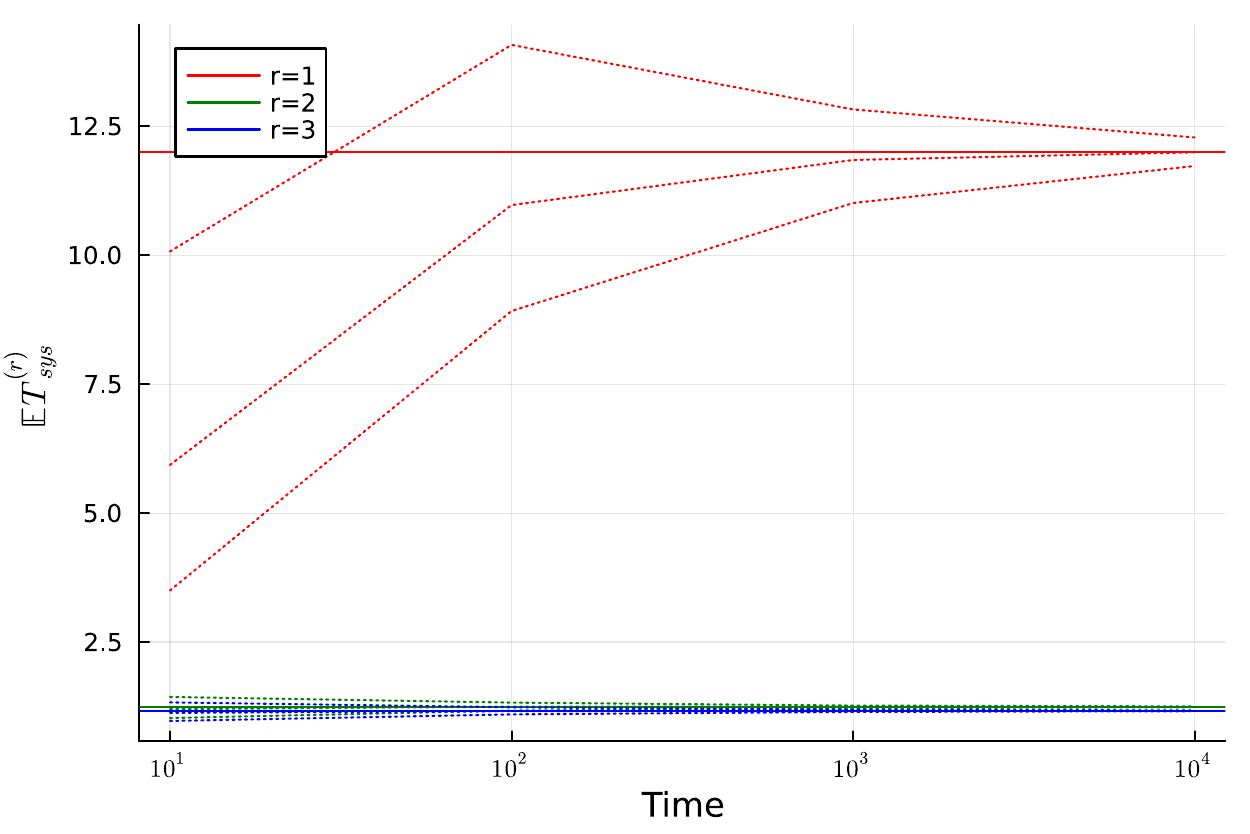}
     \end{subfigure}
%     \hfill
     \begin{subfigure}[b]{.45\textwidth}%{0.327\textwidth}
         \centering
         \includegraphics[width=\textwidth]{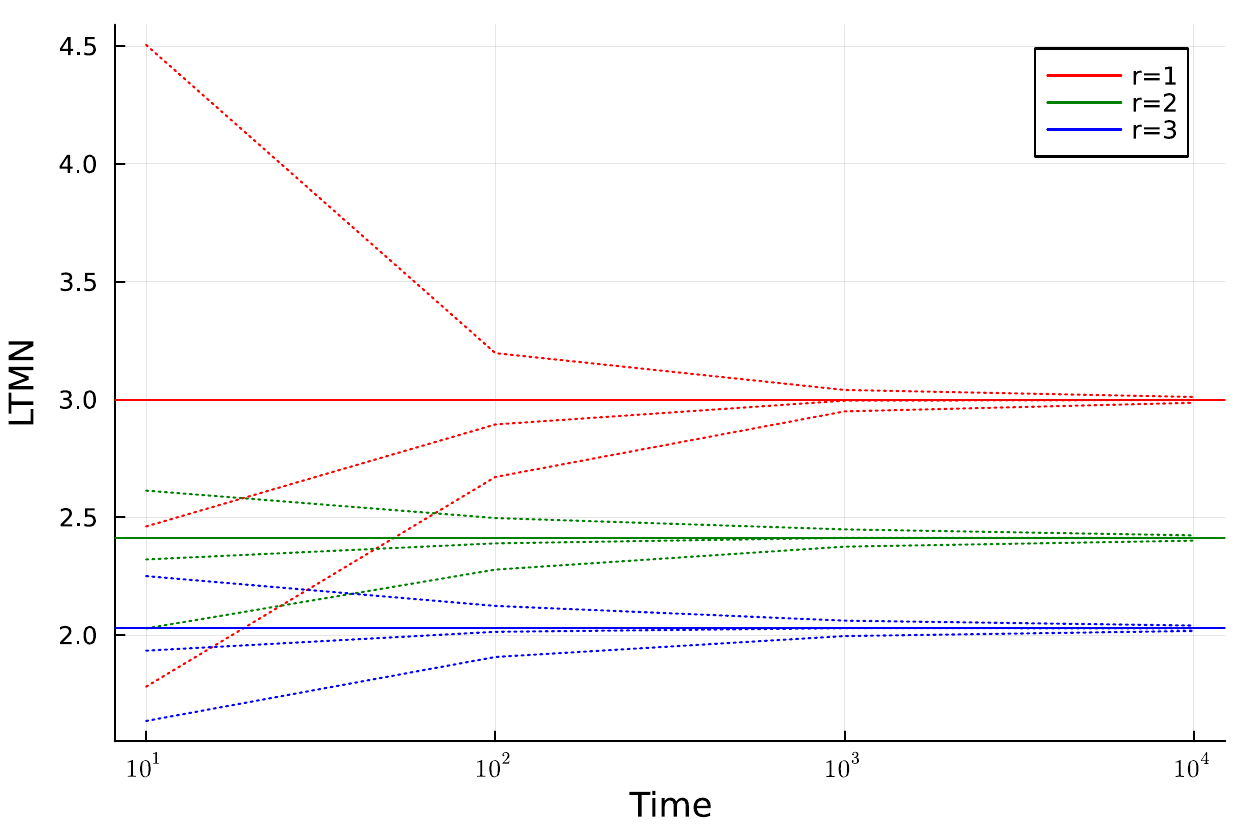}
     \end{subfigure}
%     \hfill
~
     \begin{subfigure}[b]{.45\textwidth}%{0.327\textwidth}
         \centering
         \includegraphics[width=\textwidth]{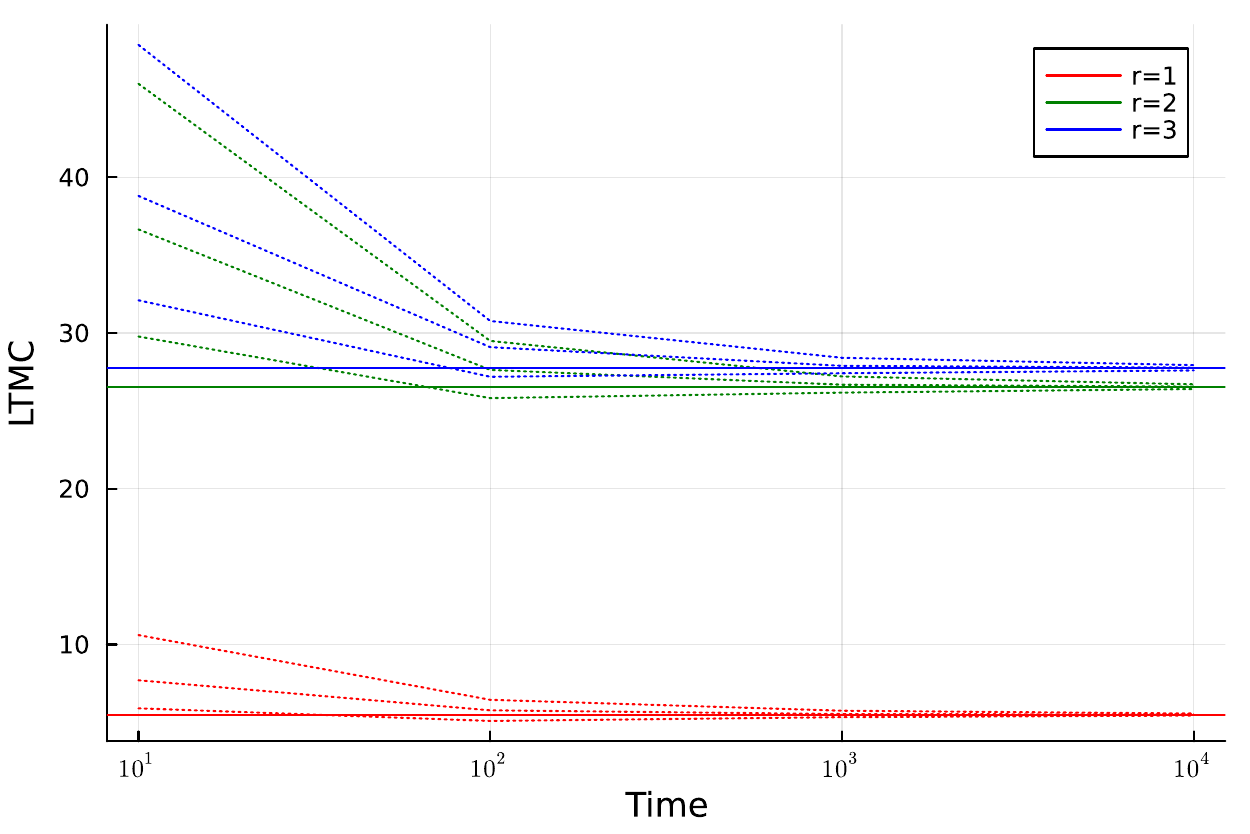}
     \end{subfigure}
\caption{
For each $r=1,2,3$, we simulate the system with $n=3$ components in Figure~\ref{fig:quants} operating under an $r$-out-of-$n$:R policy, over a simulation horizon of times $t = 10^1$, $10^2$, $10^3$ and $10^4$. We repeat this $10\,000$ times and obtain empirical distributions for estimators of $p\r$, $\E\Tsys\r$, $LTMN$ and $LTMC$. In dotted lines we show the quantiles 25\%, 50\% (median) and 75\% of each empirical distribution, and in solid line we show the corresponding theoretical value in Table~\ref{tab:r_simple}. The three values $r = 1, 2, 3$ are distinguished by line style (solid, dashed, dash-dotted) in addition to colour, so that the figure remains legible in grayscale.
}
\label{fig:quants}
}
\end{figure}
%%%%%%%%%%%%%%%%%%%%%%%%%%%%%%%%%%%%%%%%%%%%%%%%%%%%

In Figure~\ref{fig:quants} we compare the values in Table~\ref{tab:r_simple} obtained with the formulas, to the empirical values obtained using Monte Carlo simulation.
For that, we consider the same system structure and parameters of the LFMO distribution and perform the following simulation.
For each of $j=1, \ldots, 10,\!000$ repetitions, each time horizon $t$ in $10^1$, $10^2$, $10^3$ and $10^4$, and each $r=1, \, \ldots, \, n$, we do the following.
We simulate the system operating under an $r$-out-of-$n$:R policy over a simulation time horizon $[0,t]$, and register the following: the number of times when there was a repair, and the proportion---call it $\hat{p}^{j,t,(r)}$---of times of the latter where the repair was due to a system failure; the times between system failures, and the average---denote it by $\hat\E^{j,t} \, \Tsys\r$---of these times; the rate $N[0,t]/t$ of total number of components' failures per time unit---call it $\widehat{LTMN}^{j,t,(r)}$---; and the rate $C[0,t]/t$ of total cost per time unit---denote it $\widehat{LTMC}^{j,t,(r)}$.
In this way, for each time horizon $t$ in $10^1$, $10^2$, $10^3$ and $10^4$ we have the empirical distributions $\{\hat p^{1,\, t, \, (r)}, \ \ldots, \ \hat{p}^{1,000,\, t,\, (r)}\}$, $\{\hat\E^{1,\, t} \, \Tsys\r , \ \ldots, \ \hat\E^{1,000,\, t} \, \Tsys\r\}$, $\{\widehat{LTMN}^{1,\, t,\, (r)}, \ \ldots, \ \widehat{LTMN}^{1,000, \, t, \, (r)}\}$ and also $\{\widehat{LTMC}^{1,\, t,\,  (r)}, \ \ldots, \ \widehat{LTMC}^{1,000, \, t, \, (r)}\}$. For each of these sets, we obtain their quantiles $25\%$, $50\%$ and $75\%$, and plot them in Figure~\ref{fig:quants} with dotted lines, for each of the simulation time horizons $t$ in $10^1$, $10^2$, $10^3$ and $10^4$, and each $r=1, \, \ldots, \, 3$. We also plot in solid line the theoretical values we derive in our results, also shown in Table~\ref{tab:r_simple}. We do this to assess the convergence of the empirical distributions to our theoretical values.

We observe that as the time horizon grows, all the empirical distributions converge to the corresponding theoretical values. Also, in rigor, the convergences of the $LTMC$ and $LTMN$ are almost sure convergences due to Renewal Theory, whereas the convergences to $p^{(r)}$ and $\E \, \Tsys\r$ hold due to the strong Law of Large Numbers. Hence, the latter two depend on the time horizon $t$ only on the number of repetitions of the random variable that is observed during the simulated time horizon.

%%%%%%%%%%%%%%%%%%%%%%%%%%%%%%%%%%%%%%%%%%%%%%%%%%%%
\subsection{Results for a medium sized system}\label{sec:exps-med}
%%%%%%%%%%%%%%%%%%%%%%%%%%%%%%%%%%%%%%%%%%%%%%%%%%%%

%%%%%%%%%%%%%%%%%%%%%%%%%%%%%%%%%%%%%%%%%%%%%%%%%%%%
\begin{figure}[h]
{
\centering
     \begin{subfigure}[b]{.49\textwidth}%{0.327\textwidth}
         \centering
\centering
\resizebox{1\textwidth}{!}{%
\begin{circuitikz}
\tikzstyle{every node}=[font=\LARGE]
\draw  (12.5,9.5) circle (0.5cm);
\draw  (11.25,5.75) circle (0.5cm);
\draw  (7.5,2) circle (0.5cm);
\draw [, line width=0.6pt ] (5,6) circle (0.5cm);
\draw  (15.75,9.5) circle (0.5cm);
\draw  (10,2) circle (0.5cm);
\draw  (20.75,9.5) circle (0.5cm);
\draw  (7.5,9.5) circle (0.5cm);
\draw  (8.75,5.75) circle (0.5cm);
\draw  (18,9.5) circle (0.5cm);

\draw  (13.75,5.75) circle (0.5cm);
\draw  (16.25,5.75) circle (0.5cm);
\draw  (18.5,5.75) circle (0.5cm);
\draw  (19.75,7.75) circle (0.5cm);
\draw  (22.5,4) circle (0.5cm);
\draw  (12.5,2) circle (0.5cm);
\draw  (16.75,2) circle (0.5cm);
\draw  (14.75,2) circle (0.5cm);
\draw  (18.75,2) circle (0.5cm);

\draw  (23.75,5.75) circle (0.5cm);
\draw  (21,3) circle (0.5cm);
\draw [short] (8,9.5) -- (12,9.5);
\draw [short] (5,6.5) -- (7,9.5);
\draw [short] (5,5.5) -- (7,2);
\draw [short] (7.5,2.5) -- (7.5,9);
\draw [short] (7.75,9) -- (8.5,6.25);
\draw [short] (9,5.25) -- (9.75,2.5);
\draw [short] (10.25,2.5) -- (11,5.25);
\draw [short] (11.5,6.25) -- (12.25,9);
\draw [short] (12.75,9) -- (13.75,6.25);
\draw [short] (10.5,2) -- (12,2);
\draw [short] (8,2) -- (9.25,2);
\draw [short] (13,2) -- (14.25,2);
\draw [short] (15.25,2) -- (16.25,2);
\draw [short] (17.25,2) -- (18.25,2);
\draw [short] (14.25,5.75) -- (15.75,5.75);
\draw [short] (16.75,5.75) -- (18,5.75);
\draw [short] (18.75,6.25) -- (19.5,7.25);
\draw [short] (13,9.5) -- (15.25,9.5);
\draw [short] (16.25,9.5) -- (17.5,9.5);
\draw [short] (20.5,2.75) -- (19.25,2.25);
\draw [short] (23.5,5.25) -- (22.75,4.5);
\draw [short] (22,3.75) -- (21.5,3.25);
\draw [short] (20,8.25) -- (20.5,9);
\draw [short] (21.25,9.25) -- (23.5,6.25);
\draw [short] (18.5,9.5) -- (20.25,9.5);
\draw [short] (14.75,2.5) -- (16.25,5.25);
\node [font=\LARGE] at (3.5,6) {UCSB};
\node [font=\LARGE] at (7.5,10.75) {SRI};
\node [font=\Large] at (10,6.75) {STANFORD};
\node [font=\LARGE] at (11.75,4.75) {SDC};
\node [font=\LARGE] at (12.5,10.75) {UTAH};
\node [font=\LARGE] at (15.5,10.75) {NCAR};
\node [font=\LARGE] at (18,10.75) {AWS};
\node [font=\LARGE] at (21,10.75) {CASE};
\node [font=\LARGE] at (21.25,7) {ROME};
\node [font=\LARGE] at (20.5,5.25) {LINCOLN};
\node [font=\LARGE] at (16.25,6.75) {MIT};
\node [font=\LARGE] at (14.5,4.75) {ILLINOIS};
\node [font=\LARGE] at (7.5,0.75) {UCLA};
\node [font=\LARGE] at (10,0.75) {RAND};
\node [font=\LARGE] at (12.5,0.75) {ALB};
\node [font=\LARGE] at (14.5,0.75) {BBN};
\node [font=\LARGE] at (17.25,0.75) {HARVARD};
\node [font=\LARGE] at (21.25,1.5) {BURROUGHS};
\node [font=\LARGE] at (22.5,2.75) {ETAC};
\node [font=\LARGE] at (24.25,4) {MITRE};
\node [font=\LARGE] at (25.25,5.75) {CMU};
\end{circuitikz}
}%
\end{subfigure}
\caption{
We consider the classical ARPA network, see~\cite[S.~6]{satyanarayana1985linear}, where each of the $n=26$ components can be working or failed, and where the system is working if there is a path of working edges between UCSB and CMU.
}
\label{fig:arpa}
}
\end{figure}
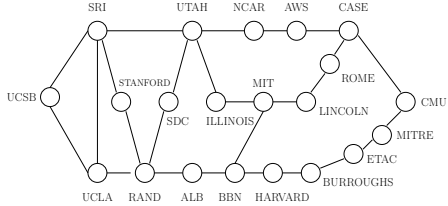
%%%%%%%%%%%%%%%%%%%%%%%%%%%%%%%%%%%%%%%%%%%%%%%%%%%%

%%%%%%%%%%%%%%%%%%%%%%%%%%%%%%%%%%%%%%%%%%%%%%%%%%%%
\begin{table}[p]
\caption{Performance indicators, computed with the results in Theorem~\ref{theo}, for the ARPA network in Figure~\ref{fig:arpa} that has $n=26$ components (the edges) that are subject to failures. The costs considered in this case are $\ccmp (i)=i$ and $\csys = 10 \times \ccmp(n)$, i.e., repairing a failed system (a corrective maintenance) costs an order of magnitude higher than replacing all components (a preventive maintenance).}    \label{tab:r_arpa}
\tiny
    \centering
    \begin{tabular}{|c|c|c||c|c||c|c|}
        \hline
        $r$ & $p\r$ & $\E \, \Tsys\r$ & $\E \, N[0, \, \Tsys\r]$ & $\E \, C[0,\Tsys\r]$ & LTMN & LTMC \\ \cline{3-5}
            &       & $\E \, \Trep\r$ & $\E \, N[0, \, \Trep\r]$ & $\E \, C[0,\Trep\r]$ &      &       \\
        \hline \hline
        1 & 0.00647927 &   6.54182 &  170.086 &  430.086 & 25.9998 & 65.7441 \\ 
          &            & 0.0423862 &  1.10203 &  2.78664 &         &         \\ \hline
        2 &  0.0406624 &   2.11763 &  53.9829 &  313.983 & 25.4921 & 148.271 \\
          &            & 0.0861081 &  2.19508 &  12.7673 &         &         \\ \hline
        3 &   0.126304 &    1.0295 &  25.7252 &  285.725 & 24.9881 & 277.538 \\
          &            &  0.130029 &  3.24919 &  36.0881 &         &         \\ \hline
        4 &   0.264619 &  0.648965 &  15.9031 &  275.903 & 24.5053 & 425.143 \\
          &            &  0.171729 &  4.20827 &  73.0093 &         &         \\ \hline
        5 &   0.430154 &  0.484459 &  11.6583 &  271.658 & 24.0646 & 560.746 \\
          &            &  0.208392 &  5.01486 &  116.855 &         &         \\ \hline
        6 &   0.591064 &  0.402885 &  9.54107 &  269.541 & 23.6819 & 669.027 \\
          &            &  0.238131 &  5.63938 &  159.316 &         &         \\ \hline
        7 &   0.725577 &  0.359054 &  8.38944 &  268.389 & 23.3654 & 747.491 \\
          &            &  0.260521 &  6.08718 &  194.737 &         &         \\ \hline
        8 &   0.826094 &  0.334493 &  7.73207 &  267.732 & 23.1158 & 800.412 \\
          &            &  0.276323 &  6.38741 &  221.172 &         &         \\ \hline
        9 &   0.895101 &  0.320503 &  7.34832 &  267.348 & 22.9275 & 834.152 \\
          &            &  0.286882 &  6.57749 &  239.304 &         &         \\ \hline
        10&    0.93943 &  0.312551 &  7.12349 &  267.123 & 22.7915 & 854.656 \\
          &            &   0.29362 &  6.69202 &  250.944 &         &         \\ \hline
        11&   0.966405 &  0.308099 &  6.99301 &  266.993 & 22.6973 & 866.583 \\
          &            &  0.297748 &  6.75808 &  258.023 &         &         \\ \hline
        12&   0.982082 &  0.305664 &  6.91864 &  266.919 & 22.6348 & 873.241 \\
          &            &  0.300187 &  6.79467 &  262.136 &         &         \\ \hline
        13&   0.990827 &  0.304372 &  6.87724 &  266.877 & 22.5949 & 876.814 \\
          &            &   0.30158 &  6.81416 &  264.429 &         &         \\ \hline
        14&   0.995519 &  0.303707 &  6.85484 &  266.855 & 22.5706 & 878.658 \\
          &            &  0.302346 &  6.82412 &  265.659 &         &         \\ \hline
        15&   0.997935 &  0.303378 &  6.84312 &  266.843 & 22.5564 & 879.573 \\
          &            &  0.302751 &  6.82899 &  266.292 &         &         \\ \hline
        16&   0.999119 &  0.303222 &  6.83725 &  266.837 & 22.5487 & 880.007 \\
          &            &  0.302955 &  6.83122 &  266.602 &         &         \\ \hline
        17&   0.999664 &  0.303152 &  6.83447 &  266.834 & 22.5447 &   880.2 \\
          &            &   0.30305 &  6.83218 &  266.745 &         &         \\ \hline
        18&   0.999891 &  0.303123 &  6.83328 &  266.833 & 22.5429 &  880.28 \\
          &            &   0.30309 &  6.83254 &  266.804 &         &         \\ \hline
        19&   0.999973 &  0.303113 &  6.83284 &  266.833 & 22.5422 & 880.308 \\
          &            &  0.303105 &  6.83266 &  266.826 &         &         \\ \hline
        20&   0.999996 &   0.30311 &  6.83271 &  266.833 &  22.542 & 880.315 \\
          &            &  0.303109 &  6.83269 &  266.832 &         &         \\ \hline
        21&        1.0 &   0.30311 &  6.83269 &  266.833 &  22.542 & 880.317 \\
          &            &   0.30311 &  6.83269 &  266.833 &         &         \\ \hline
        22&        1.0 &   0.30311 &  6.83269 &  266.833 &  22.542 & 880.317 \\
          &            &   0.30311 &  6.83269 &  266.833 &         &         \\ \hline
        23&        1.0 &   0.30311 &  6.83269 &  266.833 &  22.542 & 880.317 \\
          &            &   0.30311 &  6.83269 &  266.833 &         &         \\ \hline
        24&        1.0 &   0.30311 &  6.83269 &  266.833 &  22.542 & 880.317 \\
          &            &   0.30311 &  6.83269 &  266.833 &         &         \\ \hline
        25&        1.0 &   0.30311 &  6.83269 &  266.833 &  22.542 & 880.317 \\
          &            &   0.30311 &  6.83269 &  266.833 &         &         \\ \hline
        26&        1.0 &   0.30311 &  6.83269 &  266.833 &  22.542 & 880.317 \\
          &            &   0.30311 &  6.83269 &  266.833 &         &         \\ \hline
    \end{tabular}
\end{table}
%%%%%%%%%%%%%%%%%%%%%%%%%%%%%%%%%%%%%%%%%%%%%%%%%%%%

We now consider the classical ARPA (Advanced Research Projects Agency) computer network, the direct ancestor of the modern Internet, and a standard benchmark in network reliability studies, shown in Figure~\ref{fig:arpa}, see~\cite[S.~6]{satyanarayana1985linear}. To test our policies, we consider that each of its $n=26$ edges can be in a working or failed condition and the nodes are perfectly reliable. We also consider that the system is working if there is a path of working edges between the nodes UCSB and CMU.
We compute the signature vector $s$ using Definition~\ref{def:signature}, and its components in decimal form rounded to five decimal places for readability
%This system has the following signature vector $s$, that we compute using Definition~\ref{def:signature}. The entries are given in decimal
\footnote{For reproducibility, the exact rational entries of the structural signature vector are $s = ( 0,\ 
\frac{9}{325},\ 
\frac{209}{2,600},\ 
\frac{8,113}{59,800},\ 
\frac{54,567}{328,900},\ 
\frac{75,329}{460,460},\ 
\frac{636,819}{4,604,600},\ 
\frac{1,304,221}{12,498,200},\ 
\frac{22,582}{312,455},$  
$\frac{113,024}{2,414,425},\ 
\frac{110,911}{3,863,080},\ 
\frac{64,947}{3,863,080},\ 
\frac{67,233}{7,116,200},\ 
\frac{690,273}{135,207,800},\ 
\frac{10,227}{3,863,080},\ 
\frac{55,557}{42,493,880},\ 
\frac{3,214}{5,311,735},$  
$\frac{398}{1,562,275},\ 
\frac{289}{3,124,550},\ 
\frac{3}{115,115},\ 
\frac{1}{230,230},\ 
0,\ 
0,\ 0,\ 0,\ 0
)$.}, are
\begin{align*}
s = & (0, 0.02769, 0.08038, 0.13567, 0.16591, 0.16360, 0.13830,
0.10435, 0.07227,\\ 
& 0.04681, 0.02871, 0.01681, 0.00945, 0.00511, 0.00265, 0.00131, 0.00061,\\ 
&0.00025, 0.00009,
0.00003, 0.00000, 0, 0, 0, 0, 0).
%\label{ARPA signature}no se cita
\end{align*}

%%%%%%%%%%%%%%%%%%%%%%%%%%%%%%%%%%%%%%%%%%%%%%%%%%%%
%\begin{comment}
%\begin{align}
%s = \biggl( & 0,\ 
%\frac{9}{325},\ 
%\frac{209}{2,600},\ 
%\frac{8,113}{59,800},\ 
%\frac{54,567}{328,900},\ 
%\frac{75,329}{460,460},\ 
%\frac{636,819}{4,604,600},\ 
%\frac{1,304,221}{12,498,200},\ \nonumber \\
%& \frac{22,582}{312,455},\ 
%\frac{113,024}{2,414,425},\ 
%\frac{110,911}{3,863,080},\ 
%\frac{64,947}{3,863,080},\ 
%\frac{67,233}{7,116,200},\ 
%\frac{690,273}{135,207,800},\ \nonumber\\
%&\frac{10,227}{3,863,080},\ 
%\frac{55,557}{42,493,880},\ 
%\frac{3,214}{5,311,735},\ 
%\frac{398}{1,562,275},\ 
%\frac{289}{3,124,550},\ 
%\frac{3}{115,115},\ \nonumber\\
%&\frac{1}{230,230},\ 
%0,\ 
%0,\ 0,\ 0,\ 0
%\biggr) \label{ARPA signature}
%\end{align}
%\end{comment}

%%%%%%%%%%%%%%%%%%%%%%%%%%%%%%%%%%%%%%%%%%%%%%%%%%%%
For the failure times of the edges, we consider the same LFMO distribution as before, described at the beginning of Section~\ref{sec:exps}.

Regarding costs, we consider two cases.
In the first case, we consider the setting where $\ccmp (i)=i$, i.e., repairing each component costs one unit, however, repairing a failed system costs additionally $\csys = 10 \times n$.
This corresponds to a system disruption (a corrective maintenance), causing an additional cost that is an order of magnitude higher than replacing all components (a preventive maintenance).

This regime models communication, transportation or power-grid networks where an unplanned outage is much more expensive than a scheduled inspection --- e.g., a fibre-optic network where customer service level agreement's penalties dwarf the cost of repairing a single cable.

In Table~\ref{tab:r_arpa} we show the performance indicators obtained with Theorem~\ref{theo} in this case. We see that the long-term mean cost (LTMC) in~\eqref{eq:LTMC} attains the lowest value for the $r$-out-of-$n$:R policy with $r=1$, evidencing that corrective maintenances are much more expensive than preventive ones.
Even though this policy repairs more often and therefore exposes the system to more shocks (LTMN is highest), it eliminates almost all corrective maintenances ($p\r \approx 0.006$). This is consistent with industry practice in high-availability network operations.
We also see that the LTMC is monotone and grows with $r$.

The second case of costs that we consider is the setting where $\ccmp (i)=i$ as before; however, now repairing a failed system costs only  $\csys = 1$.
This regime models systems in which an outage is essentially as costly as preventive maintenance --- e.g., a small server cluster with hot-spare redundancy and short fail-over time, or a fleet of low-criticality IoT devices.

In Table~\ref{tab:r_arpa2} we show the performance indicators in this case, obtained with Theorem~\ref{theo}.
We consider this case because, interestingly, we see that the LTMC is not monotone in $r$: from $r=1$ to $r=3$ the LTMC decreases, then from $r=3$ to $r=6$ it increases, and then onwards it decreases again.
Overall, the lowest LTMC is when we let the system fail before any repair (the ``run-to-failure'' end), which is the same as the  $r$-out-of-$n$:R policies do for $r \geq 21$, according to Table~\ref{tab:r_arpa2}.
In such systems, our formulas allow the engineer to verify that doing little preventive maintenance is essentially as good as doing a moderate amount, which can simplify scheduling.

%%%%%%%%%%%%%%%%%%%%%%%%%%%%%%%%%%%%%%%%%%%%%%%%%%%%
\begin{table}[p]
\caption{Performance indicators, computed with the results in Theorem~\ref{theo}, for the ARPA network in Figure~\ref{fig:arpa} where the edges are subject to failures. The costs considered in this case are $\ccmp (i)=i$ and $\csys = 1$, i.e., a corrective maintenance is only marginally more expensive than a preventive maintenance.}
\tiny
    \centering
    \begin{tabular}{|c|c|c||c|c||c|c|}
        \hline
        $r$ & $p\r$ & $\E \, \Tsys\r$ & $\E \, N[0, \, \Tsys\r]$ & $\E \, C[0,\Tsys\r]$ & LTMN & LTMC \\ \cline{3-5}
            &       & $\E \, \Trep\r$ & $\E \, N[0, \, \Trep\r]$ & $\E \, C[0,\Trep\r]$ &      &       \\
        \hline \hline

  1 & 0.00647927 &   6.54182 &  170.086 &  171.086 & 25.9998 & 26.1526 \\ 
    &            & 0.0423862 &  1.10203 &  1.10851 &         &         \\ \hline
  2 &  0.0406624 &   2.11763 &  53.9829 &  54.9829 & 25.4921 & 25.9643 \\ 
    &            & 0.0861081 &  2.19508 &  2.23574 &         &         \\ \hline
  3 &   0.126304 &    1.0295 &  25.7252 &  26.7252 & 24.9881 & 25.9594 \\ 
    &            &  0.130029 &  3.24919 &  3.37549 &         &         \\ \hline
  4 &   0.264619 &  0.648965 &  15.9031 &  16.9031 & 24.5053 & 26.0463 \\ 
    &            &  0.171729 &  4.20827 &  4.47289 &         &         \\ \hline
  5 &   0.430154 &  0.484459 &  11.6583 &  12.6583 & 24.0646 & 26.1287 \\ 
    &            &  0.208392 &  5.01486 &  5.44501 &         &         \\ \hline
  6 &   0.591064 &  0.402885 &  9.54107 &  10.5411 & 23.6819 & 26.1639 \\ 
    &            &  0.238131 &  5.63938 &  6.23044 &         &         \\ \hline
  7 &   0.725577 &  0.359054 &  8.38944 &  9.38944 & 23.3654 & 26.1505 \\ 
    &            &  0.260521 &  6.08718 &  6.81276 &         &         \\ \hline
  8 &   0.826094 &  0.334493 &  7.73207 &  8.73207 & 23.1158 & 26.1054 \\ 
    &            &  0.276323 &  6.38741 &  7.21351 &         &         \\ \hline
  9 &   0.895101 &  0.320503 &  7.34832 &  8.34832 & 22.9275 & 26.0476 \\ 
    &            &  0.286882 &  6.57749 &  7.47259 &         &         \\ \hline
 10 &    0.93943 &  0.312551 &  7.12349 &  8.12349 & 22.7915 &  25.991 \\ 
    &            &   0.29362 &  6.69202 &  7.63145 &         &         \\ \hline
 11 &   0.966405 &  0.308099 &  6.99301 &  7.99301 & 22.6973 &  25.943 \\ 
    &            &  0.297748 &  6.75808 &  7.72448 &         &         \\ \hline
 12 &   0.982082 &  0.305664 &  6.91864 &  7.91864 & 22.6348 & 25.9063 \\ 
    &            &  0.300187 &  6.79467 &  7.77675 &         &         \\ \hline
 13 &   0.990827 &  0.304372 &  6.87724 &  7.87724 & 22.5949 & 25.8803 \\ 
    &            &   0.30158 &  6.81416 &  7.80499 &         &         \\ \hline
 14 &   0.995519 &  0.303707 &  6.85484 &  7.85484 & 22.5706 & 25.8632 \\ 
    &            &  0.302346 &  6.82412 &  7.81964 &         &         \\ \hline
 15 &   0.997935 &  0.303378 &  6.84312 &  7.84312 & 22.5564 & 25.8526 \\ 
    &            &  0.302751 &  6.82899 &  7.82692 &         &         \\ \hline
 16 &   0.999119 &  0.303222 &  6.83725 &  7.83725 & 22.5487 & 25.8466 \\ 
    &            &  0.302955 &  6.83122 &  7.83034 &         &         \\ \hline
 17 &   0.999664 &  0.303152 &  6.83447 &  7.83447 & 22.5447 & 25.8434 \\ 
    &            &   0.30305 &  6.83218 &  7.83184 &         &         \\ \hline
 18 &   0.999891 &  0.303123 &  6.83328 &  7.83328 & 22.5429 & 25.8419 \\ 
    &            &   0.30309 &  6.83254 &  7.83243 &         &         \\ \hline
 19 &   0.999973 &  0.303113 &  6.83284 &  7.83284 & 22.5422 & 25.8413 \\ 
    &            &  0.303105 &  6.83266 &  7.83263 &         &         \\ \hline
 20 &   0.999996 &   0.30311 &  6.83271 &  7.83271 &  22.542 & 25.8411 \\ 
    &            &  0.303109 &  6.83269 &  7.83268 &         &         \\ \hline
 21 &        1.0 &   0.30311 &  6.83269 &  7.83269 &  22.542 & 25.8411 \\ 
    &            &   0.30311 &  6.83269 &  7.83269 &         &         \\ \hline
 22 &        1.0 &   0.30311 &  6.83269 &  7.83269 &  22.542 & 25.8411 \\ 
    &            &   0.30311 &  6.83269 &  7.83269 &         &         \\ \hline
 23 &        1.0 &   0.30311 &  6.83269 &  7.83269 &  22.542 & 25.8411 \\ 
    &            &   0.30311 &  6.83269 &  7.83269 &         &         \\ \hline
 24 &        1.0 &   0.30311 &  6.83269 &  7.83269 &  22.542 & 25.8411 \\ 
    &            &   0.30311 &  6.83269 &  7.83269 &         &         \\ \hline
 25 &        1.0 &   0.30311 &  6.83269 &  7.83269 &  22.542 & 25.8411 \\ 
    &            &   0.30311 &  6.83269 &  7.83269 &         &         \\ \hline
 26 &        1.0 &   0.30311 &  6.83269 &  7.83269 &  22.542 & 25.8411 \\ 
    &            &   0.30311 &  6.83269 &  7.83269 &         &         \\ \hline
\end{tabular}
\label{tab:r_arpa2}
\end{table}
%%%%%%%%%%%%%%%%%%%%%%%%%%%%%%%%%%%%%%%%%%%%%%%%%%%%

%%%%%%%%%%%%%%%%%%%%%%%%%%%%%%%%%%%%%%%%%%%%%%%
%%%%%%%%%%%%%%%%%%%%%%%%%%%%%%%%%%%%%%%%%%%%%%%
%%%%%%%%%%%%%%%%%%%%%%%%%%%%%%%%%%%%%%%%%%%%%%%
\section{Proofs}\label{sec:proofs}
%%%%%%%%%%%%%%%%%%%%%%%%%%%%%%%%%%%%%%%%%%%%%%%
%%%%%%%%%%%%%%%%%%%%%%%%%%%%%%%%%%%%%%%%%%%%%%%
%%%%%%%%%%%%%%%%%%%%%%%%%%%%%%%%%%%%%%%%%%%%%%%

In this section we show the proofs of the results in Sections~\ref{sec:model} and~\ref{sec:main}. We group the proofs according to each section.

%%%%%%%%%%%%%%%%%%%%%%%%%%%%%%%%%%%%%%%%%%%%%%%%%%%%
\subsection{Proofs regarding the LFMO distribution}

%%%%%%%%%%%%%%%%%%%%%%%%%%%%%%%%%%%%%%%%%%%%%%%%%%%%
\begin{proof}[Proof of Lemma~\ref{lemma1}]% in Section~\ref{sec:lifetimes}]
\quad

To prove 1, we start by noting that the LFMO distribution can alternatively be defined as the MO distribution in~\eqref{def:MO eq} with the rates $\lambda^{(n)}_{|V|}$ in~\eqref{def:LV}; see e.g.~\cite[(3.3), p.~103]{matthias2017simulating}.
Hence, by the memoryless property of exponential random variables, when there are $n-i$ components alive, the time until the arrival of the next shock that takes down a given subset of $j-i$ of these working components is exponentially distributed with rate $\lambda^{(n-i)}_{j-i}$.
%Indeed, for any two sets $U$, $V \subseteq \{1,\ldots,n\}$ such that $U \setminus V \neq \emptyset$, it holds that given . 
Furthermore, there are $\binom{n-i}{j-i}$ of these possible subsets of $j-i$ components of the $n-i$ working ones, so the time until the arrival of any of these $\binom{n-i}{j-i}$ shocks is $\binom{n-i}{j-i} \lambda^{(n-i)}_{j-i}$.
%We conclude by noting that $$\binom{n-i}{j-i} \lambda_{j-i}^{(n-i)} = \sum_{k=i}^{j} (-1)^{j-k+1} \binom{n}{j} \binom{j}{k} \binom{k}{i} \Psi(n-k) / \binom{n}{i}$$ for $0 \leq i < j \leq n$, by~\eqref{def:LV 2}.
We now verify the second equality in the expression for the transition rate, namely $\binom{n-i}{j-i} \lambda_{j-i}^{(n-i)} = \sum_{k=i}^{j} (-1)^{j-k+1} \binom{n}{j} \binom{j}{k} \binom{k}{i} \Psi(n-k) / \binom{n}{i}$ for $0 \leq i < j \leq n$.
Starting from~\eqref{def:LV 2} with $|V|=j-i$ and $n$ replaced by $n-i$, we get
\begin{align*}
\binom{n-i}{j-i} \lambda^{(n-i)}_{j-i}
& = \binom{n-i}{j-i} \sum_{k=(n-i)-(j-i)}^{n-i} (-1)^{k-(n-i)+(j-i)+1} \binom{j-i}{(n-i)-k} \Psi(k) \\
& = \sum_{k=i}^j (-1)^{j-k+1} \binom{n-i}{j-i}\binom{j-i}{j-k} \Psi(n-k) \\
& = \sum_{k=i}^j (-1)^{j-k+1} \frac{\binom{n}{j} \binom{j}{k} \binom{k}{i}}{\binom{n}{i}} \Psi(n-k),
\end{align*}
where in the second equality we applied the change of summation index $k \mapsto n-k$ (which maps the range $k=n-j, \ldots, n-i$ to $k=i, \ldots, j$), and in the third equation we used the standard binomial identity
$\binom{n-i}{j-i}\binom{j-i}{j-k} = \binom{n}{j} \binom{j}{k} \binom{k}{i} / \binom{n}{i}$ for $0 \leq i \leq k \leq j \leq n$.

To prove 2, it is sufficient to note that, by basic properties of Markov chains, $P_{i,j}$ is equal to the rate $\binom{n-i}{j-i} \lambda^{(n-i)}_{j-i}$ from Part~1., divided by the sum of all the outgoing rates from state $i$.
Defining $\Delta \Psi(m) := \Psi(m+1)-\Psi(m)$, it holds that the latter sum of outgoing rates is
\begin{align*}
\lefteqn{ \sum_{k=i+1}^n \binom{n-i}{k-i} \lambda^{(n-i)}_{k-i}
= \sum_{k=1}^{n-i} \binom{n-i}{k} \lambda^{(n-i)}_{k}} \\
& = \sum_{k=1}^{n-i} \binom{n-i}{k} \sum_{l=0}^{k-1} \binom{k-1}{l} (-1)^l \Delta\Psi(n-i-k+l) \\
& = \sum_{k=1}^{n-i} \sum_{l=0}^{k} \binom{n-i}{k} \binom{k}{l} (-1)^{l+1} \Psi(n-i-k+l) \\
& = \sum_{p = 0}^{n-i} \Psi(n-i-p) \sum_{q=\max(p,1)}^{n-i} (-1)^{q-p-1} \binom{n-i}{q} \binom{q}{p} \\
& = \Psi(n-i) \sum_{q=1}^{n-i} (-1)^{q-1} \binom{n-i}{q} - \sum_{p = 1}^{n-i} \Psi(n-i-p) \binom{n-i}{p} (1-1)^{n-i-p} \\
& = \Psi(n-i) - 0,
\end{align*}
which proves~\eqref{TP}.

To prove 3, we note that \eqref{def:Q} comes from Part~2~and basic properties of Markov chains, conditioning on the jump that takes the chain into the state $j$ or the set $\{j, \ldots, k\}$.

To prove 4, for the formula for $\Q{i}{j}$, note that due to the basic properties of Markov chains and the process starting from zero, $\Q{i}{j} = \sum_{l = 0}^\infty (P^l)_{0, i} P_{i, j}$. % $Q_{i}{j} = \sum_{l = 0}^\infty (P^l)_{0, i} P_{i, j}$.
%$(P^l)_{0, i} P_{i, j}$ is the probability of going from $0$ to $i$ failed components in exactly $l$ steps and then going from $i$ to $j$ components in one step.
It follows that the only possible values for $l$ are $l=0, \ldots, i$, since for the chain $\Nd$ all possible transitions are from $i$ to $j$ with $i<j$.

To prove 5, for $r<k$, the event $\{ T_{r:n} < T_{k:n} \}$ corresponds to the cases when the chain $\Nd$ jumps from the set $\{0, \ldots, r-1\}$ to the set $\{r, \ldots, k-1\}$ before entering the set $\{k, \ldots, n\}$.
%goes from having 0 to at most $r-1$ components failed, and then making a one-step transition from there to a state in the set $\{r, \ldots, k-1\}$.
%In other words, the chain $\Nd$ does \emph{not} jumps from the set $\{0, \ldots, r-1\}$ to the set $\{k, \ldots, n\}$.
Noting the definition of $\QQ{i}{j}$ in 2., the probability $\P( T_{r:n} < T_{k:n} )$ is $\sum_{j=r}^{k-1} \QQ{r-1}{j} = \QQQ{r-1}{r}{k-1}$.
Similarly, the probability $\P( T_{r:n} = T_{k:n} )$ consists of the cases where the chain $\Nd$ jumps from the set $\{0, \ldots, r-1\}$ directly to the set $\{k, \ldots, n\}$, and this corresponds to the probability $\sum_{j=k}^n \QQ{r-1}{j} = \QQQ{r-1}{k}{n}$. Also, trivially, $\P( T_{r:n} < T_{k:n} )=0$ for $r \geq k$ and $\P( T_{k:n} = T_{k:n} )=1$.

Similarly, for~\eqref{PTkn eq j}, the event $\{ N(T_{k:n}) = j \}$ for $j \geq k$ corresponds to the case when the discrete chain of failed components $\Nd$ jumps from the set $\{0, \ldots, k-1\}$ to $j$, and this corresponds to $\QQ{k-1}{j}$. Lastly, the case $j<k$ clearly has a nil probability.
\end{proof}

%Indeed, from~(3.3) in~\cite{matthias2017simulating}, p.~103, we know that in a system with $d$ working components whose lifetimes follow a LFMO distribution with Laplace exponent $\Psi$, a shock that hits a subset of these components of cardinal $i$, for  $1 \leq i \leq d$, arrives after an exponentially distributed time with rate $\lambda^{(d)}_{i}$ defined in~\eqref{def:LV}. The term $\binom{d}{i}$ is the number of shocks that hit $i$-out-of the $d$ remaining components, so $\binom{d}{i}\lambda^{(d)}_{i}$ is the rate of arrival of {any} shock that hits exactly any $i$ of these $d$ components, and it holds that $\sum_{i=1}^d \binom{d}{i}\lambda^{(d)}_{i} = \Psi(d)$.

%%%%%%%%%%%%%%%%%%%%%%%%%%%%%%%%%%%%%%%%%%%%%%%%%%%%

%%%%%%%%%%%%%%%%%%%%%%%%%%%%%%%%%%%%%%%%%%%%%%%%%%%%
\subsection{Proofs regarding system structure}
%%%%%%%%%%%%%%%%%%%%%%%%%%%%%%%%%%%%%%%%%%%%%%%%%%%%

%%%%%%%%%%%%%%%%%%%%%%%%%%%%%%%%%%%%%%%%%%%%%%%%%%%%
\begin{proof}[Proof of Proposition~\ref{prop:signature}]% in Section~\ref{sec:structure}]
\quad

First, out of the $n!$ sequences of $n$ components, the number of sequences where the system continues to work right after the $k$-th failure is $n! \overline S_{k} = k! (n-k)! \ \# \{\mathbf x \in \{0,1\}^n : \lvert \mathbf x \rvert=n-k, \ \Phi(\mathbf x)=1 \}$.
Indeed, for any state $\mathbf x$ in $\{0,1\}^n$ with $k$ failed components (there are $\# \{\mathbf x \in \{0,1\}^n : \lvert \mathbf x \rvert=n-k, \ \Phi(\mathbf x)=1 \}$ of these states), there are $k!$ ways of sequencing first the failed components, and then $(n-k)!$ ways of sequencing the working ones. Due to the monotonicity of the system, if $\Phi(\mathbf x)=1$ then in each of these sequences the system is working right after the $k$-th failure.
This proves the second statement of Proposition~\ref{prop:signature}.
For the first statement, again due to the monotonicity of the system, the sequence of sets $\{ \{\mathbf x \in \{0,1\}^n : \lvert \mathbf x \rvert=n-k, \ \Phi(\mathbf x)=1 \}: \ k=1,\ldots,n \}$ is decreasing (nested) as $k$ grows. Hence, $n! (\overline S_{k-1} - \overline S_{k})$ is the number of sequences that fail at the $k$-th failure.
\end{proof}
%%%%%%%%%%%%%%%%%%%%%%%%%%%%%%%%%%%%%%%%%%%%%%%%%%%%

%%%%%%%%%%%%%%%%%%%%%%%%%%%%%%%%%%%%%%%%%%%%%%%%%%%%
\subsection{Proofs regarding semi-coherent systems with LFMO distribution}
%%%%%%%%%%%%%%%%%%%%%%%%%%%%%%%%%%%%%%%%%%%%%%%%%%%%

%%%%%%%%%%%%%%%%%%%%%%%%%%%%%%%%%%%%%%%%%%%%%%%%%%%%
\begin{proof}[Proof of Lemma~\ref{lemma:samaniego}]% in Section~\ref{sec:coherent}]
\quad

Let $\mathbf \varepsilon$ in $\R^n$ be the vector of triggers of the Definition~\ref{def:LFMO eq} of the LFMO distribution, and denote by $\sigma(\mathbf \varepsilon)$ the permutation of the set $\{1, \ldots, n\}$ that makes $\varepsilon_{\sigma(1)} < \ldots < \varepsilon_{\sigma(n)}$. This permutation is well defined almost surely, as the triggers are iid~exponentially distributed, so with probability one there are no ties between the triggers, and there is a unique permutation that does this. Denote also
$\Sigma^{\phi,k} := \{ \text{permutation } \sigma \text{ of } \{1,\ldots,n\} :  \phi(\mathbf 0_{\sigma,1:l})=1 \ \forall l=1,\ldots, k-1 \text{ and } \phi(\mathbf 0_{\sigma,1:l})=0 \ \forall l=k,\ldots, n \}$,
where the vector $\mathbf 0_{\sigma,1:l}$ in $\{0,1\}^n$ is
$$(\mathbf 0_{\sigma,1:l})_{\sigma(i)} = \begin{cases}
0 & i=1,\ldots,l \\
1 & i=l+1,\ldots,n
\end{cases}$$
That is, $\Sigma^{\phi,k}$ is the set of sequences (seen as permutations) where, when we turn off the components of the system sequentially one by one following the order of the permutation, the first system failure occurs when we turn off the $k$-th component.
In this way, the sets $\Sigma^{\phi,1}, \ldots, \Sigma^{\phi,n}$ forms a partition of the set of all permutations of $n$ elements, and $\sigma(\mathbf \varepsilon)$ is in one (and only one) of these sets.

Next, we claim that
\begin{align}\label{eq:Tsys decomp}
\Tsys = \sum_{k=1}^n T_{k:n} \, \I{\sigma(\mathbf \varepsilon) \in \Sigma^{\phi,k}}
\qquad \text{almost surely.}
\end{align}
Indeed, by the definition~\eqref{def:LFMO eq} of the LFMO distribution, at any time $t \geq 0$ the set of failed components is $\{ i : \varepsilon_i < L_t \}$, which ---since $\mathbf L$ is non-decreasing--- is always an initial segment of the trigger order; that is, the state of the system at time $t$ equals $\mathbf 0_{\sigma(\mathbf \varepsilon),1:\Nc(t)}$, where $\Nc(t)$ is the number of failed components at time $t$. We stress that $\Nc(\cdot)$ need not take every value in $\{0,1,\ldots,n\}$, since a jump of $\mathbf L$ may cross several triggers at once and thus produce simultaneous failures. Now, on the event $\{\sigma(\mathbf \varepsilon) \in \Sigma^{\phi,k}\}$ we have, by the definition of $\Sigma^{\phi,k}$, that $\phi(\mathbf 0_{\sigma(\mathbf \varepsilon),1:l}) = 1$ for all $l \leq k-1$ and $\phi(\mathbf 0_{\sigma(\mathbf \varepsilon),1:l}) = 0$ for all $l \geq k$. Hence the system is working precisely while $\Nc(t) \leq k-1$, and it is failed as soon as $\Nc(t) \geq k$, regardless of whether the value $k$ itself is attained by $\Nc(\cdot)$; in either case the first time at which $\Nc(t) \geq k$ is exactly $T_{k:n}$. This proves~\eqref{eq:Tsys decomp}, and in particular it remains valid when several components fail simultaneously, in which case $T_{k:n}$ coincides with one or more of the subsequent order statistics.

Define now the random index
\begin{align}\label{eq:defK}
K := \sum_{k=1}^n k \, \I{\sigma(\mathbf \varepsilon) \in \Sigma^{\phi,k}},
\end{align}
which is well defined with values in $\{1,\ldots,n\}$ since the sets $\Sigma^{\phi,1}, \ldots, \Sigma^{\phi,n}$ partition the set of permutations of $n$ elements. With this notation, identity~\eqref{eq:Tsys decomp} reads
\begin{align}\label{eq:TsysTK}
\Tsys = T_{K:n} \qquad \text{almost surely,}
\end{align}
which is item (iii) of the statement.

For item (i), recall that the triggers $\varepsilon_1, \ldots, \varepsilon_n$ are iid~with a continuous distribution, so the permutation $\sigma(\mathbf \varepsilon)$ is distributed uniformly over the set of the $n!$ permutations of $n$ elements. Hence, since $s_k$ is the proportion of permutations for which, when the components are turned off one by one, the system fails at the $k$-th failure,
\begin{align*}
\P(K = k) = \P \left( \sigma(\mathbf \varepsilon) \in \Sigma^{\phi,k} \right) = \frac{\# \Sigma^{\phi,k}}{n!} = s_k,
\qquad k=1,\ldots,n.
\end{align*}

For item (ii), denote by $\varepsilon_{(1)} < \ldots < \varepsilon_{(n)}$ the order statistics of the triggers. Since the $\varepsilon_i$ are iid~with a continuous distribution, it is classical that the rank permutation $\sigma(\mathbf \varepsilon)$ is independent of the vector $(\varepsilon_{(1)}, \ldots, \varepsilon_{(n)})$. On the other hand, from~\eqref{def:LFMO eq} we have
$$T_{k:n} = \inf \left\{ t \geq 0 \, : \, L_t > \varepsilon_{(k)} \right\}, \qquad k=1,\ldots,n,$$
so that $\mathbf T_{(\cdot)} = (T_{1:n}, \ldots, T_{n:n})$ is a measurable function of $\mathbf L$ and of $(\varepsilon_{(1)}, \ldots, \varepsilon_{(n)})$ only, and does not depend on the ranks. Since $\mathbf L$ is independent of $\mathbf \varepsilon$, we conclude that $\mathbf T_{(\cdot)}$ is independent of $\sigma(\mathbf \varepsilon)$, and therefore also of $K$, which is $\sigma(\mathbf \varepsilon)$-measurable by~\eqref{eq:defK}. This proves item (ii).

Finally, let $g : \{1,\ldots,n\} \times \R^n \to \R$ be measurable and such that the expectations below are well defined. Conditioning on the value of $K$ and using items (i) and (ii),
\begin{align*}
\E \left[ g\left( K, \mathbf T_{(\cdot)} \right) \right]
= \sum_{k=1}^n \E \left[ g\left( k, \mathbf T_{(\cdot)} \right) \I{K=k} \right]
= \sum_{k=1}^n \E \left[ g\left( k, \mathbf T_{(\cdot)} \right) \right] \P(K=k),
\end{align*}
which is~\eqref{eq:samaniego master}. The particular case~\eqref{eq:samaniego joint} follows by taking $g(k, \mathbf T_{(\cdot)}) := \I{T_{k:n}>t, \ T_{1:n}>t_1, \ \ldots, \ T_{n:n}>t_n}$ and using~\eqref{eq:TsysTK}.
\end{proof}
%%%%%%%%%%%%%%%%%%%%%%%%%%%%%%%%%%%%%%%%%%%%%%%%%%%%

%%%%%%%%%%%%%%%%%%%%%%%%%%%%%%%%%%%%%%%%%%%%%%%%%%%%
\subsection{Proof of the main theorem}
%%%%%%%%%%%%%%%%%%%%%%%%%%%%%%%%%%%%%%%%%%%%%%%%%%%%

%We now prove the main result of this paper.

%%%%%%%%%%%%%%%%%%%%%%%%%%%%%%%%%%%%%%%%%%%%%%%
\begin{proof}[Proof of Theorem~\ref{theo}]% in Section~\ref{sec:main}]
\quad

For the proof, we repeatedly use the strong Markov property and claim that once the system reaches its first system failure time or its first repair time, since all components are completely repaired, the system behaves stochastically as a brand new system.

In the following, we formalize the latter argument.
Let $(\mathbf \chi(t) : t \geq 0)$ be the \emph{component status} process at each time instant, i.e., $\mathbf \chi(t) \in \{0,1\}^n$ and $\chi_i(t)=1$ iff component $i$ is working at time $t$.
Similarly, let $(\chi\sys(t) := \Phi(\mathbf \chi(t)) : t \geq 0)$ be the \emph{system status} process at each time instant; i.e., $\chi\sys(t)=1$ iff the system is working ($\Phi(\mathbf \chi(t))=1$) at time $t$.
Also, when operating under an $r$-out-of-$n$:R policy, denote by $\Trep\r[1]$, $\Trep\r[2]$, \ldots and $\Tsys\r[1]$, $\Tsys\r[2]$, \ldots the sequences of, respectively, times of repairs and of system failure.
%Specifically, $\Trep\r(0) = \Tsys\r(0) := 0$ and $\Trep\r(m+1) := \inf \{t > \Trep\r(m) : \sum_{i=1}^n \chi_i(t) \geq r \text{ or } \chi\sys(t) =0 \}$ and $\Tsys\r(m+1) := \inf \{t > \Tsys\r(m) : \chi\sys(t) =0 \}$ and $\varepsilon_i(m+1) := L_{\Trep\r(m+1)} - \log U_{i,m}$ if $\chi_i(\Trep\r(m+1)) = 0$ and $\varepsilon_i(m)$ otherwise.
Also, define $\Trep\r[0] := 0$ and $\varepsilon_i[0] := E_{i,0}$ for all $i=1, \ldots, n$, where $E_{i,m}$, $i=1,\ldots,n$ and $m \geq 0$, are iid standard exponential random variables.
Then inductively define for $m=0, 1, \ldots$
\begin{align*}
\chi_{i} [m] (t) & := \I{L_t \geq \varepsilon_i[m]} \qquad \text{for all } i=1, \ldots, n \\
\chi\sys[m] (t) & := \Phi(\mathbf\chi [m] (t)) \\
\Trep\r[m+1] & := \inf \left\{t > \Trep\r[m] \ : \ \sum_{i=1}^n (1-\chi_{i} [m] (t)) \geq r \ \text{or} \ \chi\sys[m] (t) =0 \right\} \\
\varepsilon_i[m+1] & := \begin{cases} L_{\Trep\r[m+1]} + E_{i,m+1} & \text{if } \chi_{i} [m] \left( \Trep\r[m+1] \right) = 0 \\ \varepsilon_i[m] & \text{otherwise, for all } i=1, \ldots, n \end{cases}
\end{align*}
where $\mathbf\chi [m](t)  = (\chi_1 [m](t), \ \ldots, \ \chi_n [m](t))$.
Similarly, $\Tsys\r[0] := 0$ and for $m \geq 0$,
\begin{align*}
\Tsys\r[m+1] := \min \left\{\Trep\r[l] \ : \ l \geq 1 , \ \Trep\r[l] > \Tsys\r[m] \text{ and } \qquad \qquad \qquad \right. \\
\left.  \chi\sys[l-1] \left( \Trep\r[l] \right) = 0 \right\}.
\end{align*}
In this way, $\chi_i(t) := \sum_{m \geq 0} \chi_i [m](t) \I{\Trep\r[m] \leq t < \Trep\r[m+1]}$ is the status of component $i$ at time $t$, and $$\chi\sys(t) := \sum_{m \geq 0} \chi\sys [m](t) \I{\Trep\r[m] \leq t < \Trep\r[m+1]}$$ is the corresponding system status.

Note that, by construction, the times $\Trep\r[m]$, $m\geq 1$, and $\Tsys\r[m]$, $m\geq 1$, are sequences of stopping times with respect to the filtration induced by the process $\left( (\mathbf \chi(t), \, \chi\sys(t)) : t \geq 0 \right)$ --- in fact, by the sub-filtration induced by the process $\left( (\sum_{i=1}^n (1-\chi_i(t)), \, \chi\sys(t)) : t \geq 0 \right)$ of number of failed components and system status.
It follows that the component status process $(\mathbf \chi(t) : t \geq 0)$ is renewed at each repair time $\Trep\r[m]$, $m\geq 1$, or also at each system failure time $\Tsys\r[m]$.
This holds due to the memoryless property of the exponential triggers and the strong Markov property of the L\'evy subordinator $\mathbf L$ in the definition~\eqref{def:LFMO} of the LFMO distribution, and because under the $r$-out-of-$n$:R policy \emph{all} failed components are \emph{completely} repaired at the repair times.

We now prove parts 1 to 5 of Theorem~\ref{theo}.

To prove 1, first recall that $\Tsys\r$ is the first system failure time when operating with the $r$-out-of-$n$:R policy, assuming that we start with all components working; and in the same context, $\Trep\r$ is the first repair time.
To analyze the event $\{\Trep\r = \Tsys\r\}$ it is sufficient to focus on the failure times without repair policy.
Denote by $\Tsys$ the system failure time when there is no repair policy, and in this case the first repair would occur at time $\min\{\Tsys, T_{r:n}\}$, where, recall, $T_{1:n} \leq \ldots \leq T_{n:n}$ are the ordered failure times of the components, $T_1, \ldots, T_n$.
Hence, we have $\P\left( \Trep\r = \Tsys\r \right) = \P\left( \min\{T_{r:n}, \Tsys\} = \Tsys \right)$. By Lemma~\ref{lemma:samaniego}, $\Tsys = T_{K:n}$ almost surely for a random index $K$ with $\P(K=k) = s_k$ that is independent of $\mathbf T_{(\cdot)}$; the event $\{\min\{T_{r:n}, \Tsys\} = \Tsys\}$ is therefore of the form $\{g(K,\mathbf T_{(\cdot)})=1\}$ with $g(k,\mathbf T_{(\cdot)}) := \I{\min\{T_{r:n}, T_{k:n}\} = T_{k:n}}$, and~\eqref{eq:samaniego master} yields
$$\P\left( \min\{T_{r:n}, \Tsys\} = \Tsys \right) = \sum_{k=1}^n s_k \, \P\left( \min\{T_{r:n}, T_{k:n}\} = T_{k:n} \right).$$
We obtain~\eqref{def:pF} by noting the trivial facts that $\min\{T_{r:n}, T_{k:n}\} = T_{\mini{r}{k}:n}$ and $\P\left( T_{\mini{r}{k}:n} = T_{k:n} \right) = 1$ for $k\leq r$. Lastly,~\eqref{eq:1-pF} is a direct consequence of~\eqref{def:pF}.

To prove 2,  let $j$ in $1, \ \ldots, \ n$ and again note that for the event $\{ N(\Trep\r) = j \}$ it is sufficient to focus on the event $\{ N(\min\{\Tsys, T_{r:n}\}) = j \}$ when there is no repair policy, since each time there is a repair or a system failure, the system behaves equal in probability to a new system with no failed components, due to the Markov property of the L\'evy subordinator $L$ and the memoryless property of the triggers $\varepsilon_1, \ \ldots, \ \varepsilon_n$. It follows that
\begin{align*}
\P\left( N(\Trep\r) = j \right) & = \P\left( N(\min\{\Tsys, T_{r:n}\}) = j \right) \\
& = \sum_{k=1}^n s_k \, \P\left( N(\min\{T_{k:n}, T_{r:n}\}) = j \right) \\
& = \sum_{k=1}^n s_k \, \P\left( N(T_{\mini{r}{k}:n}) = j \right),
\end{align*}
where the second equality is~\eqref{eq:samaniego master} applied to $g(k, \mathbf T_{(\cdot)}) := \I{N(\min\{T_{k:n}, T_{r:n}\}) = j}$, which is a measurable function of $k$ and $\mathbf T_{(\cdot)}$ by the duality relation recalled before Lemma~\ref{lemma1}, together with $\Tsys = T_{K:n}$ from~\eqref{eq:TsysTK}.
Lastly, note that at time $T_{\mini{r}{k}:n}$ there must be at least $\mini{r}{k}$ dead components. So, in particular if $j<r$ and $k>j$ (so $\mini{r}{k}>j$), we have $\P\left( N(T_{\mini{r}{k}:n}) = j \right) = 0$. This proves~\eqref{def:wr}.

Now we prove 3. We start by proving~\eqref{eq:PNTrep eq 1} and~\eqref{eq:PNTrep eq 2}. Recall that we assume $p\r = \P\left( \Trep\r = \Tsys\r \right) >0$. Consider any $j = 1, \ldots, n$, and note that
\begin{align*}
& \P\left( N(\Tsys\r) = j \right) \\
& = \sum_{l \geq 1} \P\left( N(\Trep\r(l)) = j \left| \Trep\r(l) = \Tsys\r \right.\right) \P \left(\Trep\r(l) = \Tsys\r \right) \\ 
& = \sum_{l \geq 1} \P\left( N(\Trep\r(1)) = j \left| \Trep\r(1) = \Tsys\r \right.\right) \P \left(\Trep\r(l) = \Tsys\r \right) \\
& = \P\left( N(\Trep\r) = j \left| \Trep\r = \Tsys\r \right.\right) 
\end{align*}
where $\Trep\r(l)$ is the $l$-th time of repair, and the second equality is due to the strong Markov property of the repair process. Indeed, after any time of repair, since all components are completely repaired, the process as a completely new process starting at zero. This proves~\eqref{eq:PNTrep eq 1}.
Next, note that
\begin{align*}
& \P\left( N(\Trep\r) = j, \ \Trep\r = \Tsys\r \right) = \P\left( N(\min\{\Tsys, T_{r:n}\}) = j, \ \Tsys \leq T_{r:n} \right) \\
& = \sum_{k=1}^n s_k \, \P\left( N(\min\{T_{k:n}, T_{r:n}\}) = j, \ T_{k:n} \leq T_{r:n} \right) \\
& = \sum_{k=1}^j s_k \, \P\left( N(\min\{T_{k:n}, T_{r:n}\}) = j, \ T_{k:n} \leq T_{r:n} \right)
\end{align*}
where the second equality is~\eqref{eq:samaniego master} applied to $g(k, \mathbf T_{(\cdot)}) := \I{N(\min\{T_{k:n}, T_{r:n}\}) = j, \ T_{k:n} \leq T_{r:n}}$, and the last equation comes from the fact that $N(T_{k:n}) \geq k$, so for $k>j$ we have
$$\P\left( N(\min\{T_{k:n}, T_{r:n}\}) = j, \ T_{k:n} \leq T_{r:n} \right) \leq \P\left( N(T_{k:n}) = j \right) = 0.$$
It follows that for $k \leq j$ we have $\{ N(\min\{T_{k:n}, T_{r:n}\}) = j, \ T_{k:n} \leq T_{r:n} \} = \{ N(T_{\mini{r}{k}:n}) = j \}$. Indeed, the case $k \leq r$ is trivial as $T_{k:n} \leq T_{r:n}$ holds; and if $k>r$ then $\{ N(T_{\mini{r}{k}:n}) = j \} = \{ N(T_{r:n}) = j \}$ with $j>k$ implies that $T_{k:n} = T_{r:n}$. This proves~\eqref{eq:PNTrep eq 2}.

Lastly, for~\eqref{eq:PNTrep ineq}, we use $\P( N(\Trep\r) = j, \ \Trep\r < \Tsys\r ) = \P( N(\Trep\r) = j ) - \P( N(\Trep\r) = j, \ \Trep\r = \Tsys\r ) $ and conclude using~\eqref{def:wr} and~\eqref{eq:PNTrep eq 2}.

We now prove 4. We start by observing that, by Renewal Theory, see e.g.~\cite[Proposition~3.4.1]{resnick2013adventures}, it holds that 
\begin{align*}
\lim_{t \to \infty} \frac{C[0, t]}{t} = \frac{\E \, C[0,\Tsys\r]}{\E \, \Tsys\r} = \frac{\E \, C[0,\Trep\r]}{\E \, \Trep\r} \qquad \text{a.s.,}
\end{align*}
since the times $\Tsys\r$ and $\Trep\r$ are renewal times, due to the aforementioned Markov property. It follows that
\begin{align}
\E \, C[0,\Trep\r] = \E \left[ \csys \I{\Tsys\r=\Trep\r} + \sum_{j=1}^{n} \ccmp(j) \, \I{N(\Trep\r) = j} \right]
\label{eq:CTrep}
\end{align}
and
\begin{align}
\E \, \Trep\r = \E \left[ \min\{\Tsys, T_{r:n} \} \right] = \sum_{k=1}^n s_k \, \E \left[ \min\{T_{k:n}, T_{r:n} \} \right] = \sum_{k=1}^n s_k \, \E \left[ T_{\mini{k}{r}:n} \right], \label{eq:Trep}
\end{align}
where the second equality is~\eqref{eq:samaniego master} applied to $g(k, \mathbf T_{(\cdot)}) := \min\{T_{k:n}, T_{r:n}\}$, whose integrability follows from $\min\{T_{k:n}, T_{r:n}\} \leq T_{n:n}$ and~\eqref{eq:ETkn}.

Next, for $\E \, C[0,\Tsys\r]$ and $\E \, \Tsys\r$, we see that the process up to the first system failure time $\Tsys\r$ consists of a random number ---geometrically distributed on $0, 1, \ldots$ with parameter $1 - p\r = \P\left( \Trep\r < \Tsys\r \right)$--- of processes starting with all components working up to the first repair time $\Trep\r$ conditional on $\Trep\r < \Tsys\r$, concatenated with a last process starting with all components working, up to the first repair time $\Trep\r$ conditional on $\Trep\r = \Tsys\r$. In particular this implies that the system will fail almost surely.

\begin{align*}
& \E \, C[0,\Tsys\r] \\
& = \csys + \left( \frac{1}{p\r} - 1 \right) \sum_{j=1}^{n} \ccmp(j) \, \P\left( N(\Trep\r) = j \left|\, \Trep\r < \Tsys\r \right.\right) \\
& \qquad\qquad\qquad\qquad\qquad\qquad + \sum_{j=1}^{n} \ccmp(j) \, \P\left( N(\Trep\r) = j \left|\, \Trep\r = \Tsys\r \right.\right) \\
%& = \csys + \sum_{j=1}^{n} \ccmp(j) \, \left( \left( \frac{1}{p\r} - 1 \right) \frac{1}{1-p\r} \sum_{k=j+1}^n s_k \, \P \left( N\left( T_{\min\{r,k\}:n} \right) = j \right) \I{r \leq j} + \frac{1}{p\r} \sum_{k = 1}^j s_k \, \P \left( N\left( T_{\min\{r,k\}:n} \right) = j \right) \right) \\
& = \csys + \sum_{j=1}^{n} \ccmp(j) \, \left( \left( \frac{1}{p\r} - 1 \right) \frac{1}{1-p\r} \sum_{k=j+1}^n s_k \, \P \left( N\left( T_{\mini{r}{k}:n} \right) = j \right) \I{r \leq j} \right. \\
& \qquad\qquad\qquad\qquad\qquad\qquad\qquad\qquad\qquad\qquad + \left. \frac{1}{p\r} \sum_{k = 1}^j s_k \, \P \left( N\left( T_{r:n} \right) = j \right) \right) \\
& = \csys + \frac{1}{p\r} \sum_{j=1}^{n} \ccmp(j) \, \left(  \sum_{k=j+1}^n s_k \, \P \left( N\left( T_{\mini{r}{k}:n} \right) = j \right) \I{r \leq j} \right. \\
& \qquad\qquad\qquad\qquad\qquad\qquad\qquad\qquad \left. + \sum_{k = 1}^j s_k \, \P \left( N\left( T_{\mini{r}{k}:n} \right) = j \right) \right) \\
& = \csys + \frac{1}{p\r} \sum_{j=1}^{n} \ccmp(j) \, \sum_{k = 1}^{j \I{r > j} + n \I{r \leq j}} s_k \, \P \left( N\left( T_{\mini{r}{k}:n} \right) = j \right) \\
& = \csys + \frac{1}{p\r} \sum_{j=1}^{n} \ccmp(j) \, \P \left( N(\Trep\r) = j \right)
\end{align*}
and
\begin{align*}
& \E \, \Tsys\r = \left( \frac{1}{p\r} - 1 \right) \E \, \left[ \Trep\r \left|\, \Trep\r < \Tsys\r \right.\right] + \E \, \left[ \Trep\r \left|\, \Trep\r = \Tsys\r \right.\right] \\
& = \frac{1}{p\r} \left( \E \, \left[ \Trep\r; \ \Trep\r < \Tsys\r\right] + \E \, \left[ \Trep\r; \ \Trep\r = \Tsys\r \right] \right) = \frac{1}{p\r} \E \, \Trep\r
\end{align*}

Lastly, to prove 5, again by Renewal Theory we have that $\lim_{t \to \infty} C[0, t] / t = \E \, C[0,\Trep\r] / \E \, \Trep\r$ almost surely, since the repair times are renewal times. We conclude using~\eqref{eq:CTrep} and~\eqref{eq:Trep}, since they also hold when $p\r=0$.
\end{proof}

\subsection{Proof of Corollary~\ref{corol:iid}}

\begin{proof}%[Proof of Corollary~\ref{corol:iid}]
The proof is direct by noting that the iid~case corresponds to the pure drift process $L_t = \mu t$, in which case $\Psi(x) = \mu x$.
Hence, in~\eqref{def:LV}, we get the shock that hits the subset $V$ of components arrives after an exponentially distributed time with rate $\lambda^{(n)}_{V} = \mu$ for $|V|=1$ and 0 for $|V|>1$.
In~\eqref{TP} and~\eqref{def:Q} this implies that $P_{i,j} = \QQ{i}{j} = \I{j=i+1}$.
Hence, $p\r = \sum_{k=1}^{r} s_k$ in~\eqref{def:pF}; and in~\eqref{def:wr}, $\P(N(\Trep\r) = j) = s_j$ if $j < r$, $= \sum_{k=r}^n s_k$ if $j=r$, and $=0$ otherwise.
Also, in~\eqref{eq:ETkn} we have $\E T_{k:n} = \sum_{i=n-k+1}^n 1/(\mu i)$ since, by induction, $\sum_{i = n-k+1}^n \binom{n}{i} \binom{i-1}{n-k} (-1)^{i-n+k-1} / i = \sum_{i=n-k+1}^n 1/i$ for all $1 \leq k \leq n$.
We conclude by applying the latter to~\eqref{eq:LTMC}, \eqref{eq:ECost} and~\eqref{eq:ETfail}.
\end{proof}

%The case of iid exponentially distributed components, with rate $\mu>0$, corresponds to the pure drift process $L_t = \mu t$. In this case we have $\Psi(x) = \mu x$. In~\eqref{def:LV}, we get the shock that hits the subset $V$ of components arrives after an exponentially distributed time with rate $\lambda^{(n)}_{V} = \mu$ for $|V|=1$ and 0 for $|V|>1$. In~\eqref{eq:ETkn} we have $\E T_{k:n} = (1/n + 1/(n-1) + \ldots + 1/(n-k+1)) 1/\mu = \sum_{i=n-k+1}^n 1/\mu i$. Also, in~\eqref{TP} we get $P_{i,j} = 1$ if $j=i+1$ and $0$ otherwise; therefore, in~\eqref{def:Q}, $\QQ{i}{j} = 1$ if $j=i+1$ and $0$ otherwise; and in~\eqref{PTkn eq j}, $\P\left( N(T_{r:n}) = k \right) = 1$ if $r=k$ and $0$ otherwise. In~\eqref{def:w}, $\P(N(\Tsys) = j) = s_j$. In~\eqref{def:pF}, $p\r = \sum_{k=1}^{r} s_k$. In~\eqref{def:wr}, $\P(N(\Trep\r) = j) = s_j$ if $j < r$, $\sum_{k=r}^n s_k$ if $j=r$, and $0$ otherwise.

\section*{Acknowledgments}

GL acknowledges the funding of the grants ANID/\allowbreak FONDECYT/\allowbreak Iniciaci\'on 11230256 and ANID/\allowbreak FONDECYT/\allowbreak Regular~1231207.
JN was supported by Ministerio de Ciencia e Innovaci\'{o}n of Spain in the project PID2022-137396NB-I00, funded by MICIU/AEI/10.13039/501100011033 and by ``ERDF A way of making Europe''.
HO acknowledges the funding of the grant  ANID/\allowbreak FONDECYT/\allowbreak Regular~13220168.
We also thank the anonymous reviewers for several helpful suggestions that have served to improve the earlier version of this paper.

\section*{Declaration of generative AI and AI-assisted technologies in the manuscript preparation process}

During the preparation of this work, the authors used Claude (Anthropic) for the manuscript drafting and figure generation. The authors reviewed and edited the output as needed and take full responsibility for the content of the published article.

%%%%%%%%%%%%%%%%%%%%%%%%%%%%%%%%%%%%%%%%%%%%%%%%%%%%
%%%%%%%%%%%%%%%%%%%%%%%%%%%%%%%%%%%%%%%%%%%%%%%%%%%%
\bibliographystyle{elsarticle-num} 
\bibliography{bibliography}
%%%%%%%%%%%%%%%%%%%%%%%%%%%%%%%%%%%%%%%%%%%%%%%%%%%%
%%%%%%%%%%%%%%%%%%%%%%%%%%%%%%%%%%%%%%%%%%%%%%%%%%%%

\end{document}